\DeclareSymbolFont{AMSb}{U}{msb}{m}{n}	
\documentclass[noamsfonts,reqno]{amsart}	

\usepackage{amsmath}
\usepackage{microtype}
\usepackage{paralist}
\usepackage{thmtools}
\usepackage{tikz}
\usepackage{tikz-cd}
\usepackage[all,cmtip]{xy}
\usepackage{centernot}
\usetikzlibrary{matrix, arrows}
\usepackage{pgfplots}
\pgfplotsset{compat=1.13}
\usepackage{ stmaryrd }
\usepackage{subcaption}
\usepackage{mathtools}

\usepackage{hyperref}

\usetikzlibrary{shapes} 

\usepackage[textsize=footnotesize, colorinlistoftodos]{todonotes}
\makeatletter
\providecommand\@dotsep{5}
\makeatother

\newcommand{\todonull}[1]{}  
\newcommand{\todolow}[1]{}
\newcommand{\omitted}[1]{}

\newcommand{\omitnow}[1]{}

\definecolor{my-linkcolor}{rgb}{0.75,0,0}
\definecolor{my-citecolor}{rgb}{0,0.5,0}
\definecolor{my-urlcolor}{rgb}{0,0,0.75}
\hypersetup{
    colorlinks, 
    linkcolor={my-linkcolor},
    citecolor={my-citecolor}, 
    urlcolor={my-urlcolor}
}

\usepackage[charter]{mathdesign}   
\usepackage{relsize}
\usepackage{tikz}
\usetikzlibrary{arrows,automata}
\usetikzlibrary{positioning}
\usetikzlibrary{calc}
\usepackage{tikz-cd}
\usepackage{xy}
\usetikzlibrary{matrix, arrows}
\usepackage{pgfplots}
\pgfplotsset{compat=1.13}

\tikzset{
    state/.style={
           rectangle,
           rounded corners,
           draw=black, very thick,
           minimum height=2em,
           inner sep=2pt,
           text centered,
           },
}

\usepackage {xcolor}

\DeclareMathOperator{\bbC}{\mathbb{C}}

\DeclareMathOperator{\Rp}{R_{\lambda}}
\DeclareMathOperator{\cRp}{{\mathcal R}_{\lambda}}

\DeclareMathOperator{\Hfin}{\mathcal{H}_t}
\DeclareMathOperator{\DAHA}{\mathcal{H}^\mathrm{daff}_{q,t}}
\DeclareMathOperator{\Htrig}{\mathcal{H}_k}
\DeclareMathOperator{\Htrigkrega}{\mathcal{H}^{\mathrm{trig}}_{\krega}}
\DeclareMathOperator{\Htrigreg}{\mathcal{H}_{k,\reg}}
\DeclareMathOperator{\Htrigneg}{\mathcal{H}_{-k}}
\DeclareMathOperator{\Htrigsub}{H_k}
\DeclareMathOperator{\Htrigsubp}{H_{k'}}

\DeclareMathOperator{\Haff}{H^{\aff}_q}

\DeclareMathOperator{\Hext}{H_{q}^{\mathrm{ext}}}
\DeclareMathOperator{\Hq}{H_{q}}

\DeclareMathOperator{\J}{K}

\DeclareMathOperator{\cI}{\mathcal{I}}
\DeclareMathOperator{\cJ}{\mathcal{K}}

\newcommand{\tfact}[1]{{#1}^{\mathsf{f}}}

\newcommand{\symO}[1]{\left[#1\right]}

\newcommand{\Cs}[2]{C(#1,#2)}

\newcommand{\Radpm}{\mathrm{Rad}^{\pm}}
\newcommand{\Radp}{\mathrm{Rad}^{+}}

\newcommand{\pmRad}{{}^{\pm}\mathrm{Rad}}
\newcommand{\pRad}{{}^{+}\mathrm{Rad}}

\newcommand{\jtheta}{\mathbf{k}_\vartheta}

\DeclareMathOperator{\bbZ}{\mathbb{Z}}

\DeclareMathOperator{\calG}{\mathcal{G}}

\DeclareMathOperator{\calL}{\mathcal{L}}
\DeclareMathOperator{\calK}{\mathcal{K}}

\DeclareMathOperator{\frakh}{\mathfrak{h}}

\DeclareMathOperator{\fX}{\mathfrak{X}}

\newcommand {\h}{\mathfrak{h}}

\newcommand {\aand}{\qquad\text{and}\qquad}
\newcommand {\wrt}{with respect to }

\newcommand {\rhs}{right--hand side }

\newcommand {\IN}{\mathbb N}
\newcommand {\IC}{\mathbb C}

\newcommand {\krega}{_{k\text{--reg}_a}}
\newcommand {\kreg}{_{k\text{--reg}}}

\newcommand{\act}[2]{#2^{#1}}
\newcommand{\pact}[2]{#2^{#1}}

\DeclareMathOperator{\End}{End}
\DeclareMathOperator{\Sym}{S}

\DeclareMathOperator{\Br}{Br}

\DeclareMathOperator{\reg}{{reg}}
\DeclareMathOperator{\aff}{aff}
\DeclareMathOperator{\Ind}{Ind}

\DeclareMathOperator{\Res}{\mathrm{Res}}

\DeclareMathOperator{\Sh}{\Sym\negthinspace\frakh}
\DeclareMathOperator{\Shst}{\Sym\negthinspace\frakh^*}

\DeclareMathOperator{\Rad}{\mathrm{Rad}^{+}}

\DeclareMathOperator{\tPhi}{\Phi}
\DeclareMathOperator{\N}{N}

\newtheorem*{theorem}{Theorem}
\newtheorem*{proposition}{Proposition}
\newtheorem*{lemma}{Lemma}
\newtheorem*{conjecture}{Conjecture}
\newtheorem*{corollary}{Corollary}
\newtheorem*{definition}{Definition}
\theoremstyle{definition}
\newtheorem*{remark}{Remark}
\newtheorem*{example}{Example}

\numberwithin{equation}{section}

\newcommand {\valeriocomment}[1]{}

\newcommand {\robincomment}[1]{\footnote{\color{orange}{#1}}}

\newcommand{\Done}{}

\hypersetup{
	colorlinks=true,
	citecolor=blue,
	linkcolor=blue,
	filecolor=blue, 
	urlcolor=blue,
}

\newcommand {\frachst}{\IC(\frakh^*)}
\newcommand {\Omit}[1]{}
\newcommand {\VV}{\mathcal{V}}
\renewcommand {\O}{\mathcal{O}}
\newcommand {\germ}[1]{(#1)_{h_0}}
\newcommand {\epspm}{\epsilon_\pm}
\newcommand {\epsmp}{\epsilon_\mp}
\renewcommand {\IC}{\mathbb{C}}

\newcommand {\sff}{^{\mathsf{f}}}
\newcommand {\ICW}{\IC W}
\newcommand {\bfi}[1]{\mathbf{i_{#1}}}
\newcommand {\wt}[1]{\widetilde{#1}}
\newcommand {\sfA}{\mathsf{A}}
\newcommand {\sfB}{\mathsf{B}}
\newcommand {\sfC}{\mathsf{C}}
\newcommand {\sfD}{\mathsf{D}}
\newcommand {\IZ}{\mathbb{Z}}
\newcommand {\cow}[1]{\lambda_{#1}^{\vee}}
\renewcommand {\SS}{\mathfrak{S}}
\newcommand {\orb}{^{\operatorname{orb}}}
\newcommand {\ol}[1]{\overline{#1}}
\newcommand {\sfG}[1]{\Gamma\left(#1\right)}
\newcommand {\half}[1]{\frac{#1}{2}}

\newcommand{\omitvaleriocomment}[1]{}
\newcommand{\omitrobincomment}[1]{}
\newcommand {\ind}{\operatorname{ind}}

\newcommand {\Hom}{\operatorname{Hom}}
\newcommand {\calS}{\mathcal S}
\newcommand {\fa}{\mathfrak{a}}
\newcommand {\fe}{\mathfrak{e}}
\newcommand {\Aff}{\operatorname{Aff}}
\newcommand {\fsl}{\mathfrak{sl}}

\newcommand {\pos}{R_+}
\newcommand {\apos}{R^a_+}
\newcommand {\negr}{R_-}
\newcommand {\aneg}{R^a_-}
\newcommand {\cW}{\mathcal W}
\newcommand {\sign}{\operatorname{sign}}
\newcommand {\KZ}{^{\operatorname{\scriptscriptstyle{KZ}}}}
\newcommand {\ilambda}{\mathbf{i}_\lambda}
\newcommand {\iLambda}{\mathbf{i}_\Lambda}
\newcommand {\inu}[1]{\mathbf{i}_{#1}}
\newcommand {\ktheta}{\mathbf{k}_\vartheta}
\newcommand {\kTheta}{\mathbf{k}_\Theta}
\newcommand {\imu}{\mathbf{i}_\mu}
\newcommand {\csign}{+} 
\newcommand {\ncsign}{-} 

\newcommand {\HC}{Harish--Chandra }
\newcommand {\ie}{{\it i.e. }}
\newcommand {\TT}[2]{T_{#1,#2}}
\newcommand {\cT}[2]{\mathcal{T}_{#1,#2}}
\newcommand{\Rnu}[1]{R_{#1}}

\newcommand {\veps}{\varepsilon}
\newcommand {\triv}{{\operatorname{triv}}}
\newcommand {\Endd}{\End_0}

\newcommand {\Qeps}{E_{\eps}}

\newcommand {\ii}{\mathrm{i}}
\newcommand {\I}[1]{I(#1)}  
\newcommand {\dI}[1]{I_{#1}}  
\newcommand{\Kconn}[1][\vartheta,k]{\nabla^{K}(#1)}
\newcommand{\Iconn}[1][\lambda,k]{\nabla^{I}(#1)}
\newcommand{\Kconnvf}[2][\vartheta,k]{\nabla_{#2}^{K}(#1)}
\newcommand{\Iconnvf}[2][\lambda,k]{\nabla_{#2}^{I}(#1)}
\newcommand{\KL}[1][\vartheta,k]{\mathcal{K}(#1)}
\newcommand{\IL}[1][\lambda,k]{\mathcal{I}(#1)}
\newcommand{\Y}{Y}  
\newcommand{\X}{X}  

\newcommand {\cor}{\alpha^\vee}
\newcommand {\eg}{{\it e.g. }}
\newcommand {\Rneg}{R_-}

\title{On the Finkelberg--Ginzburg Mirabolic Monodromy Conjecture}

\author[V. Toledano Laredo]{Valerio Toledano Laredo}
\address[]{Department of Mathematics, Northeastern University, 
360 Huntington Ave., Boston, Massachusetts 02115 }
\email{V.ToledanoLaredo@northeastern.edu }

\author[R. Walters]{Robin Walters}
\address[]{Khoury College of Computer Sciences, Northeastern University,
360 Huntington Ave., Boston, Massachusetts 02115} 
\email{r.walters@northeastern.edu}

\thanks{The first author was supported by the National Science Foundation
through the grants DMS--1802412 and DMS–2302568, the second author
was supported by the NSF grant DMS--1503050 and the NSF RTG grant {\it
Algebraic Geometry and Representation Theory at Northeastern University}
DMS--1645877.}

\begin{document}

\begin{abstract}
We compute the monodromy of the mirabolic  $\mathcal{D}$--module for all values
of the parameters $(\vartheta,c)$ in rank 1, and outside an explicit codimension
2 set of values in ranks 2 and higher. This shows in particular that the Finkelberg--Ginzburg
conjecture, which is known to hold for generic values of $(\vartheta,c)$, fails at
special values even in rank 1. Our main tools are Opdam's shift operators and
normalised intertwiners for the extended affine Weyl group, which allow for the
resolution of resonances outside the codimension two set.
\end{abstract}

\setcounter{tocdepth}{1}

\maketitle
\tableofcontents

\raggedbottom
\pagebreak


\raggedbottom
\pagebreak

\section{Introduction}

\subsection{} 

Mirabolic $\mathcal{D}$--modules are a subcategory of regular holonomic $\mathcal{D}
$--modules on the variety 
$SL_n(\bbC) \times \bbC^n$. They were introduced by
Gan and Ginzburg in \cite{gan2006almost} as an analog of Lusztig's character sheaves.

In the classical setting, character sheaves are certain perverse sheaves defined over a
reductive group $G$ \cite{deligne1976representations, lusztig1985character}. Over $\bbC$,
they correspond to \emph{admissible} $\mathcal{D}$--modules \cite{ginsburg1989admissible}.
Hotta and Kashiwara defined an admissible $\mathcal{D}$-module called \emph{the \HC}
$\mathcal{D}$--module $\mathcal{G}_\vartheta$ \cite{hottakashiwara}, and Kashiwara proved
that $\mathcal{G}_\vartheta$ is the minimal extension of its restriction to the locus $G^
{\mathrm{reg}}\subset G$ of regular semisimple elements \cite{kashiwara2014invariant},
which is a local system $\mathscr{K}_\vartheta$ of rank $|W|$.

When $G=SL_n(\IC)$, Finkelberg and Ginzburg defined an analogous mirabolic \HC
$\mathcal{D}$-module $\calG_{\vartheta,c}$ on $\fX=G\times\IC^n$ \cite{finkelberg2010mirabolic}
in terms of parameters $\vartheta\in\frakh^*/W$ and $c\in\bbC$, where $\h\subset\fsl_n(\IC)$ is
a Cartan subalgebra and $W=\SS_n$ its Weyl group. Similarly to the classical setting, the restriction
of $\calG_{\vartheta,c}$ to an appropriate open subvariety $\fX^{\mathrm{reg}}$ is a local
system $\calK_{\vartheta,c}$ of rank $|W|$ and, for generic values of $c$, $\calG_{\vartheta,c}$
is the minimal extension of $\calK_{\vartheta,c}$ \cite[Cor. 1.5.4]{bellamy2015hamiltonian}.

\subsection{} 

Finkelberg and Ginzburg give the following alternative description of $\calK_{\vartheta,c}$
in terms of the degenerate affine Hecke algebra $\Htrigsub$ corresponding to $W$ 
\cite{finkelberg2010mirabolic}. Let $H\subset SL_n(\IC)$ be the maximal torus with Lie
algebra $\h$, $H^{\mathrm{reg}}$ the set of regular elements in $H$, and consider the map 
\begin{align*}
	\mathrm{spec}: \fX^{\mathrm{reg}} & \longrightarrow H^{\mathrm{reg}} /W    \\
	(g,v)                             & \longmapsto \text{ eigenvalues of } g 
\end{align*}
Using Hamiltonian reduction \cite{gan2006almost} and the trigonometric KZ functor \cite
{varagnolo2004double}, they show that $\mathcal{K}_{\vartheta,c}$ is the pull--back of a
local system $\KL$ on $H^{\mathrm{reg}} /W$, where $k=c-1$.

The local system $\KL$ arises from Cherednik's trigonometric KZ connection
$\Kconn$. Specifically, for $\vartheta\in\h^*/W$, let $\IC_\vartheta$ be the corresponding
1--dimensional representation of $\Sh^W$ on which $W$ acts trivially, and define the {covariant
representation} $\J_\vartheta$ of $\Htrigsub$ by
\[\J_\vartheta
=
\ind^{\Htrigsub}_{W\otimes\Sh^W}\IC_\vartheta\]
Then, $\Kconn$ defines a meromorphic, $W$--equivariant connection on the trivial
vector bundle over $H$ with fibre $\J_\vartheta$ and logarithmic singularities over the root 
hypertori. 

\subsection{} 

The connection $\Kconn$ descends to the adjoint torus $T=H/\Omega^\vee\subset
PGL_n(\IC)$, where $\Omega^\vee=P^\vee/Q^\vee\cong\IZ_n$, and $Q^\vee\subset P
^\vee\subset\h$ are the coroot and coweight lattices of $SL_n(\IC)$. Its monodromy therefore
defines a representation of the orbifold fundamental group $\pi_1\orb(T_{\reg}/W)$, which
factors through the extended affine Hecke algebra $\Hext$ generated by the group algebra
of $P^\vee=\pi_1(T)$, together with the finite Hecke algebra $\Hq$ of type $\sfA_{n-1}$,
where $q=\exp(2\pi \ii k)$.

Finkelberg and Ginzburg gave a conjectural description of the monodromy of $\Kconn$. Consider the dual torus $H^\vee=\Hom_{\IZ}(P^\vee,\IC^\times)$, set $\Theta
=\exp_{H^\vee}(\vartheta)\in H^\vee/W$, and let $\IC_\Theta$ be the corresponding
1--dimensional representation of $\IC[H^\vee]^W$. Define the covariant representation
$\J(\Theta)$ of $\Hext$ to be 
\[\J(\Theta)=\ind^{\Hext}_{\Hq\otimes\IC[H^\vee]^W}\IC_\Theta\]
Then, Finkelberg--Ginzburg formulated the following \cite[Conj. 6.4.1]{finkelberg2010mirabolic}.

\label{conj:monoconj}
\begin{conjecture}
Assume that $c=k+1$ is not a rational number of the form $\frac{p}{m}$, where $2 \leq m \leq n$,
$1 \leq p \leq m$ and $\mathrm{gcd}(p,m) = 1$. Then, the monodromy of the trigonometric KZ
connection $\Kconn$ with values in the covariant representation $\J_\vartheta$ of $\Htrigsub$
is isomorphic to the covariant representation $K(\Theta)$ of the extended affine Hecke algebra
$\Hext$. 
\end{conjecture}

One potential application of the above conjecture is in finding the representing object of the
trigonometric KZ functor defined in \cite{varagnolo2004double}.  Such an object $P_{\mathrm
{KZ}}$ was first described for the rational case in \cite{ginzburg2003category} and constructed
by Losev in the totally aspherical case \cite{losev2014totally}. The covariant representation of
the trigonometric Cherednik algebra, which corresponds to  $\KL[\vartheta,c]$
by the KZ functor, is a potential such object.

\subsection{Previous results}

Conjecture \ref{conj:monoconj} is known to hold if the covariant representation $K(\Theta)$
is irreducible (see \cite[Prop. 3.4]{cherednik1994integration} and Proposition \ref{pr:generic
monodromy}). In turn, by a criterion of T. Kato, this latter condition is equivalent to $(\vartheta,
k)$ lying on the complement of the affine hyperplanes $\left\{\lambda(\cor)+n=\pm k\right\}$,
where $\cor$ is a coroot, $n\in\IZ$, and $\lambda\in\h^*$ is a lift of $\vartheta$ \cite[Thm 2.2]
{Kato-irred}. The challenge is therefore to compute the monodromy of the trigonometric KZ
connection $\Kconn$ when $(\vartheta,k)$ lie on these hyperplanes.
\Omit{The challenge is to compute monodromy when
the connection is \emph{resonant}, \ie the eigenvalues of its residues at the point with
coordinates $e^{-\alpha_i}=1$ differ by integers.\valeriocomment{Don't think that is the
issue actually.}
}

\subsection{Rank 1}\label{ss:rk 1} 

In rank 1, the connection $\Kconn$ can be reduced to a hypergeometric equation, and its
monodromy computed for all values of the parameters $(\vartheta, k)$. To state our result,
let $\alpha$ be the positive root of $\mathfrak{sl}_2(\IC)$, $\lambda\in\h^*$ the preimage
of $\vartheta$ with $\Re(\lambda(\alpha^\vee)) \geq 0$, and identify $\frakh^*$ with $\bbC$
by $\nu\to\nu(\alpha^\vee)$. 
The extended affine Hecke algebra $\Hext$ in this case is generated by two invertible
elements $Y,T$ which satisfy
\begin{equation}\label{eq:AHA rk 1 intro}
 (T-1)(T+q) = 0
\aand
T \Y - \Y^{-1} T = 
(1-q) \Y
\end{equation}
Let $y,t\in\IC^\times$ be solutions of \eqref{eq:AHA rk 1 intro}, so that
\[(y,t)\in\left\{(\pm{q}^{-1/2},1),(\pm{q}^{1/2},-q)\right\}\] 
and denote  by $(y,t)$ the character of $\Hext$ given by
$\Y\to y$ and $T\to t$. Set $\I{\Lambda} = \mathrm{Ind}_{\IC[Y]}^{\Hext} \IC_\Lambda$
where $\Y$ acts on $\IC_\Lambda$ by multiplication by $\Lambda=e^{\pi \ii \lambda}$.

Then, the following holds (Thm. \ref{th:main 1}).

\label{th:i main 1}
\begin{theorem}
The monodromy of $\KL$ is isomorphic to $\J(\Theta)$ except in the following cases
\begin{enumerate}
\item If $\lambda,k\in\IZ$ and $|\lambda| \geq\max(k,1-k)$, $\KL$ is isomorphic to $(-\Lambda,1)\oplus(-\Lambda,-1)$.
\item If $\lambda,k\in\half{1}+\IZ$ and $|\lambda| \geq\max(k,1-k)$, $\KL$ is isomorphic to the induced
representation $\I{-\Lambda}$.
\item If one of $k \pm \lambda$ lies in $\IZ_{>0}$, $\lambda\notin\half{1}+\IZ$, and (1) does not apply, 
$\KL$ is isomorphic to $K(\Theta)^-$, where the latter is $\ind^{\Hext}_{\Hq\otimes\IC[H^
\vee]^W}\IC_\Theta^-$, where $T$ acts on $\IC_\Theta^-$ as multiplication by $-q$.
\end{enumerate}
\end{theorem}
The proof of Theorem \ref{th:main 1} shows that at the values of $(\lambda,k)$ specified above,
$(-\Lambda,1)\oplus (-\Lambda,-1)$, $\I{-\Lambda}$, and $K(\Theta)^-$ are not isomorphic to $K
(\Theta)$, and therefore that Conjecture \ref{conj:monoconj} fails in the cases (1)--(3). Figure 1
below represents these cases.

\begin{figure}[ht]
  \centering
  
  \begin{tikzpicture}[scale=0.3]
    \begin{axis}[
        color =black,
        xlabel=k,
        ylabel=$\lambda$,
        axis lines=middle,
        xmin=-3, xmax=3,
        ymin=-3, ymax=3,
        xtick={-3,-2,-1,0,1,2,3}, ytick={-3,-2,-1,0,1,2,3},
        width=18cm,
        line width=1pt,
        x label style={at={(axis description cs:1.01,0.5)},anchor=west,font=\huge},
        y label style={at={(axis description cs:0.5,1.01)},anchor=south,font=\huge},
    ]
    \addplot [color=red,line width=2pt,domain=-3:3, samples=2] {x-1};
    \addplot [color=red,line width=2pt,domain=-3:3, samples=2] {x-2};
    \addplot [color=red,line width=2pt,domain=-3:3, samples=2] {x-3};  
    \addplot [color=red,line width=2pt,domain=-3:3, samples=2] {x-4};  
    \addplot [color=red,line width=2pt,domain=-3:3, samples=2] {x-5};  
    \addplot [color=red,line width=2pt,domain=-3:3, samples=2] {x-6};  
    \addplot [color=red,line width=2pt,domain=-3:3, samples=2] {-x+1};
    \addplot [color=red,line width=2pt,domain=-3:3, samples=2] {-x+2};
    \addplot [color=red,line width=2pt,domain=-3:3, samples=2] {-x+3};
    \addplot [color=red,line width=2pt,domain=-3:3, samples=2] {-x+4}; 
    \addplot [color=red,line width=2pt,domain=-3:3, samples=2] {-x+5};
    \addplot [color=red,line width=2pt,domain=-3:3, samples=2] {-x+6};  
    
    \addplot [color=green,only marks,mark size=5pt] table[row sep=\\]{%
        0.5 0.5\\
        -0.5 1.5\\
        0.5 1.5\\
        1.5 1.5\\
        -1.5 2.5\\
        -0.5 2.5\\
        0.5 2.5\\
        1.5 2.5\\
        2.5 2.5\\
    };
    \addplot [color=green,only marks,mark size=5pt] table[row sep=\\]{%
        0.5 -0.5\\
        -0.5 -1.5\\
        0.5 -1.5\\
        1.5 -1.5\\
        -1.5 -2.5\\
        -0.5 -2.5\\
        0.5 -2.5\\
        1.5 -2.5\\
        2.5 -2.5\\
    };
    
    \addplot [color=white,only marks,mark size=5pt] table[row sep=\\]{%
        1.5 0.5\\
        2.5 0.5\\
        2.5 1.5\\
        1.5 -0.5\\
        2.5 -0.5\\
        2.5 -1.5\\
    };
    \addplot [color=blue,only marks,mark size=5pt] table[row sep=\\]{%
        0 -1\\
        1 -1\\
        -1 -2\\
        0 -2\\
        1 -2\\
        2 -2\\
        -2 -3\\
        -1 -3\\
        0 -3\\
        1 -3\\
        2 -3\\
        3 -3\\
        0 1\\
        1 1\\
        -1 2\\
        0 2\\
        1 2\\
        2 2\\
        -2 3\\
        -1 3\\
        0 3\\
        1 3\\
        2 3\\
        3 3\\
    };
    \end{axis}
  \end{tikzpicture}
  \caption{The cases from Theorem \ref{th:main 1} corresponding to different monodromy representations:
  (0) White $K(\Theta)$, (1) Blue $(-\Lambda,1) \oplus (-\Lambda,-1)$, (2) Green $\I{-\Lambda}$, (3) Red $K(\Theta)^{-}$} 
  
\end{figure}
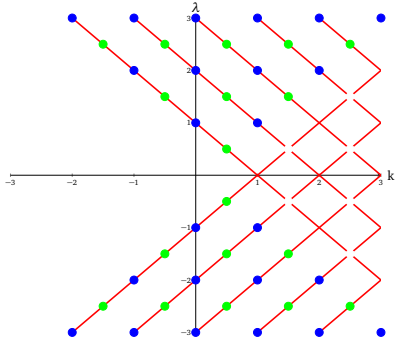
\newpage 

\subsection{Rank $\mathbf{n\geq 2}$} 

In higher rank, the monodromy of a trigonometric KZ connection $\nabla$ can be explicitly computed by 
reduction to rank 1, provided $\nabla$ is {\it non--resonant}, \ie such that the eigenvalues of its residues
at the (large volume limit) point with coordinates $e^{\alpha_i}=0$ do not differ by non--zero integers
\cite[1.2.3]{cherednik2005double}.\footnote{Note that this does not directly determine the monodromy
as a representation of the extended affine Hecke algebra, but reduces this problem to the analysis of
the explicit action of $\Htrig$.} 

When $\nabla=\Kconn$ corresponds to the covariant representation $\J_\vartheta$, we determine in Sect.
\ref{sec:monodromy-j-theta} the set $\Res\subset\h^*$ for which $\nabla$ is resonant. Define $B\subset\h$
by
\[B = \left\{\alpha^\vee_{j_1}+\cdots+\alpha^\vee_{j_i}\left
|\ \alpha^\vee_{j_k} \text{ are orthogonal coroots} \right.\right\}\]
then (see Prop. \ref{pr:non resonance})
\[\Res = \left\lbrace \lambda\in\h^* | \lambda(q) \in \IZ_{\neq 0} \text{ for some } q \in B \right\rbrace\]

Our main result in rank $n\geq 2$ is that, outside an explicit codimension 2 subset of $\h^*\times\IC$
the monodromy of a resonant connection $\Kconn$ is equivalent to that of a non--resonant $\Kconn$.
Specifically, define the following subsets of $\h^*$
 \begin{align*}
E & =  W\left\lbrace \lambda\ |\ \lambda(\alpha^\vee) \notin \mathbb{R}_{\leq k} \ \forall\ \alpha \in R^+ \right\rbrace \\
S & =  \left\lbrace \lambda\ |\ \lambda(q) \in \mathbb{Z} \text{ for a unique } q \in B \right \rbrace 
\end{align*}
\Omit{
L &= \lbrace \lambda_2 x_2 + \ldots + \lambda_n x_n\ |\ x_2+\ldots+x_n = k \text{ or } -k \rbrace,                      \\
L &=  
\left\lbrace -m_1 \lambda_1 - \ldots - m_l \lambda_l + x_{l+1} \lambda_{l+1} + \ldots + x_n \lambda_n : \right.  	\label{eqn:defL}\\
& \left. \qquad \qquad 1 \leq l \leq n  , m_i \in \mathbb{Z}_{\leq 0}, {\textstyle 0 \leq j \leq  \sum_{i=1}^l m_i }, q \in \mathbb{Z}_{\geq l+1}, \right.  \nonumber \\ 
it would be 1 \leq j if not for case 3
& \left. \qquad \qquad   x_{l+1} + \ldots +  x_q = \pm k + j  \right\rbrace, \label{eqn:defE} \\
I &  = \lbrace \lambda\ |\ -k < \lambda(\alpha^\vee) < k\ \forall\  \alpha \in\pos \rbrace, \nonumber  \\
}
Define $\lambda\in\h^*$ to be affine $k$--regular if $\lambda(\alpha^\vee) \not \in \mathbb{Z}$ for all $\alpha
\in R^+$. Then, the following holds (Thm. \ref{thm:sln-shift})

\label{thm:i sln-shift}
\begin{theorem}  
Assume that $n\geq 2$, and that $\Kconn$ is resonant, and let $\lambda\in\h$ be a preimage of $\vartheta$.
Then, if one of the following conditions holds $\Kconn$ has monodromy equivalent to that of a non--resonant
system
\begin{enumerate}
\item $\lambda$ is affine $k$--regular and $\lambda \in S$
\item $k \in \IZ$ and $\lambda \in E$
\item $k \in \IZ$ and $\lambda \in S$
\end{enumerate}

In particular, since condition (1) is generic in $\mathrm{Res}$, the set of points which are resonant and cannot be shifted to a non-resonant point has codimension 2.
\end{theorem} 

Note that, unlike Theorem \ref{th:main 1}, Theorem \ref{thm:sln-shift} does not describe the set of parameters $(\vartheta,k)$ 
for which Conjecture \ref{conj:monoconj} holds. However, when coupled with the elementary computation of the monodromy 
of $\Kconn$ for $k=0$, it yields the following (see Cor. \ref{cor:highrankFGfail})

\begin{corollary}
Assume that $k\in\IZ$ and that $\lambda\in\h^*$ lies in $E$, is regular and such that $\lambda
(\lambda_i^\vee) \in\IZ$ for all $i$. Then, Conjecture \ref{conj:monoconj} does not hold. In particular,
the latter fails in all ranks $n\geq 1$.
\end{corollary}

\subsection{Resonance and shifting}

The proof of Theorem of \ref{thm:i sln-shift} hinges on resolving the resonances of the
trigonometric KZ connection. To this end, we use two main tools: the shift operators in
$\lambda\in\h^*$ and the $k$--shift operators. These are also useful in rank 1, as they
facilitate the transcription from the vector valued KZ connection to the scalar valued
hypergeometric equation, and are therefore also used in the proof of Theorem \ref
{th:i main 1}.

The former are the intertwining operators for the extended affine Weyl group $W\ltimes
P\subset W\ltimes\IC[H]$. When made to act on the trivial vector bundle $\cI_\lambda$
over $H^{\reg}$ with fibre the induced representation $\dI{\lambda}=\ind^{\Htrigsub}_{\Sh}
\IC_\lambda$ of $\Htrigsub$, $\lambda\in\h^*$, they give rise to a $W$--equivariant integrable
system of difference equations in $\lambda$. They define morphisms of 
vector bundles with connection 
\[\cT{w}{\lambda}:\cI_\lambda\to\cI_{w\lambda},\quad w\in W^e\]
and we show that $\cT{w}{\lambda}$ is an isomorphism if and only if $\lambda(\alpha^\vee)
\neq\pm k$ for any positive affine coroot $\cor$ such that $w\cor$ is negative (Prop. 
\ref{pr:det cT}).

This allows in particular to replace
$\lambda$ by a shift $\lambda+\nu$, $\nu\in P$, which may be less resonant than $\lambda$,
while preserving the isomorphism class of the monodromy. This can then be brought to
bear on the KZ connection with values in the covariant representation since $\J_{\vartheta}$
is isomorphic to $\dI{\lambda}$ if $\lambda(\cor)\neq k$ for any positive root $\alpha$ (see
Thm. \ref{th:det R}), where $\lambda\in\h^*$ is any preimage of $\vartheta$.

\subsection{Shift Operators in $k$}    

Shift operators in the parameter $k$ were introduced by Opdam for the scalar valued
hypergeometric system $\Cs{\vartheta}{k}$ defined by the action of the symmetrized
Dunkl operators \cite{opd:89,opdam2001lectures}.

The local system $\Cs{\vartheta}{k}$ possesses an alternative, vector valued
description in terms of the trigonometric KZ local system $\IL[\mu,k']$ defined
by the induced representation $I_\mu$ of $\Htrigsubp$. Specifically, Matsuo defined
a morphism of local systems $m_{\lambda,k}:\IL \to \Cs{-\vartheta}{-k}$,
where $-\vartheta\in\h^*/W$ is the image of $-\lambda\in\h^*$, and proved that it is 
an isomorphism if (and only if) $\lambda(\alpha^\vee)\neq k_\alpha$ for any $\alpha\in\pos$
\cite{matsuo1992integrable}. Cherednik defined an analogous morphism
$\mathrm{ch}_{\lambda,k}:\IL \to \Cs{-\vartheta}{-k+1}$, and showed
that it is an isomorphism for all values of $(\lambda,k)$ \cite
{cherednik1994integration}.\valeriocomment{this is only true if by parameters one means $\theta$, not $\lambda$ as far as I can tell.}

Felder--Veselov pointed out that these two descriptions give rise to a shift operator for
the KZ system through the composition $\mathrm{ch}_{\lambda,k+1}^{-1}\circ m_
{\lambda,k}:\IL \to \IL[\lambda,k+1]$, and proved that the opposite
composition $m_{-\lambda,-k}\circ \mathrm{ch}_{-\lambda,-k}^{-1}:\Cs{\vartheta}{k}
\to\Cs{\vartheta}{k-1}$ is a multiple of Opdam's shift operator \cite{felder1994shift}.  

In this paper, we adapt these constructions to the local systems $\KL$ defined by the
trigonometric KZ connection with values in the covariant representation $\J_\vartheta$. We define
morphisms 
\[m_{\vartheta,k}:\KL\to \Cs{-\vartheta}{-k}
\aand
\mathrm{ch}_{\vartheta,k}:\KL \to \Cs{-\vartheta}{-k+1}\]
and obtain tight invertibility criteria for both of them. We then
prove the following result (Thm. \ref{thm:k-shift-op-invert})

\begin{theorem}
The shift operator
\[\calS_k = \mathrm{ch}_{\vartheta,k+1}^{-1}\circ m_{\vartheta,k}:\KL
\to\KL[\vartheta,k+1]\]
is invertible if $\lambda(\alpha^\vee)\neq k_\alpha$ for any $\alpha\in R$.
\end{theorem}

\subsection{ Outline of Paper} 

In \autoref{sec:Cherednikalgebras}, we review the definition of the trigonometric Cherednik
algebra $\Htrig$, and its differential-trigonometric and difference-rational polynomial representations,
which give rise to trigonometric KZ connections and shift operators in the spectral parameters
respectively.

In \autoref{sec:induced-representation}, we review the definition of the degenerate affine Hecke
algebra, that of its induced and covariant representations $\dI{\lambda}$ and $\J_\vartheta$, and
give necessary and sufficient condition for them to be isomorphic, and for their irreducibility.


In \autoref{sec:trig-kz} we adapt Matsuo and Cherednik's maps to define isomorphisms of local systems
between the trigonometric KZ system with values in the covariant representation $\J_\vartheta$ and the
hypergeometric system defined by $\vartheta$. We then use both to obtain a shift operator 
$\KL\to\KL[\vartheta,k+1]$,
and determine sufficient conditions for its invertibility.

In \autoref{sec:shiftoperatorsinlambda} we review the definition of the 
intertwining operators
for the affine Weyl group, and show that they give rise to a $W$--equivariant integrable systems of difference
equations in the parameter $\lambda\in\h^*$.

In \autoref{sec:AHA}, we review the definition of the induced and covariant representations of the extended
affine Hecke algebra.
In \autoref{sec:monodromy-j-theta} we show that Conjecture \ref{conj:monoconj} holds whenever $K(\Theta)$
is irreducible.
In \autoref{sec:monodromy-rank-one}, we compute the monodromy of the rank 1 trigonometric KZ connection with
values in $\J_\vartheta$ for all values of $(\vartheta,k)$.
In the final \autoref{sec:monodromy-higher-rank}, we determine this monodromy in rank $n\geq 2$ outside an explicit
codimension 2 set of parameters $(\vartheta,k)$.


\section{\Done Cherednik Algebras}
\label{sec:Cherednikalgebras}

We review below the definition of the trigonometric Cherednik algebra $\Htrig$, and of its differential--trigonometric
and difference--rational representations. 

\subsection{Root Systems}

Let $\fe$ be a Euclidean space, $\h$ its complexification, $R=\{\alpha\}$ a root system in $\fe^*$, and
$R^\vee=\{\alpha^\vee\}\subset\fe$ the dual root system. We denote by $Q\subset\h^*$ and $Q^\vee
\subset\h$ the root and coroot lattices generated by $R$ and $R^\vee$, and by $P\subset\h^*$ and
$P^\vee\subset\h$ the weight and coweight lattices dual to $Q^\vee$ and $Q$ respectively. 

Let
\[T = \Hom_{\IZ}(Q,\IC^\times)
\aand
H = \Hom_{\IZ}(P,\IC^\times) \]
be the tori of adjoint and 'simply connected' type respectively. We denote the natural basis of $\bbC
[H]=\IC P$ by $\{e^\lambda\}_{\lambda \in P}$.

Fix a system $\pos\subset R$ of positive roots, let $\Delta=\{\alpha_i\}_{i=1}^n\subset R$ and
$\Delta^\vee=\{\cor_i\}_{i=1}^n\subset R^\vee$ be the corresponding bases of simple (co)roots,
and $\{\lambda_i^{\vee}\}_{i=1}^n\subset P^\vee$ and $\{\lambda_i\}_{i=1}^n\subset P$ the
dual bases of fundamental (co)weights. 

\subsection{The extended affine Weyl group}
\label{subsec:weylgroup}

\omitvaleriocomment{Do we ever use the affine Weyl groups $W^a_{R} = W \ltimes Q^\vee$
and $W^e_{R} = W \ltimes P^\vee$ associated to $R$? If not, we could simply introduce the
groups $W \ltimes Q$ and $W \ltimes P$ and call them the affine and extended affine Weyl
groups with no specific reference to $R$ or $R^\vee$, as Opdam does. A related philosophical
question is whether the affine Weyl group 'associated' to $R$ should be $W \ltimes Q^\vee$
or $W \ltimes Q$.}

Let $W$ be the Weyl group of $R$, and
\[ W^a = 
W \ltimes Q
\aand
W^e = 
W \ltimes P  \]
the corresponding affine and extended affine Weyl groups. Since the action of 
$W$ on $P$ preserves $Q$--cosets, $W^a$ is a normal subgroup of $W^e$,
and $W^e/W^a\cong P/Q$.

The action of $W$ on $\h$ extends to a faithful affine action of $W^e$ by letting
$P$ act by translations $t_{\lambda}(h)=h+\lambda$. Let $\Aff(\h^*)$ be the group
of affine transformations of $\h^*$ and, for any $\cor\in R^\vee$ and $n\in\IN$,
denote by $s_{\cor+n}\in\Aff(\h^*)$ the affine reflection given by
\[s_{\cor+n} (h) = h - (\cor(h)+n) \alpha\]
Since $s_{\cor} \circ s_{\cor + n} = s_{-\cor + n} \circ 
s_{\cor}=t_{n\alpha}$, $W^a$ is generated by the reflections $\{s_{\cor+n}\}_
{\cor\in R^\vee, n\in\IZ}$. As such, it is a Coxeter group, with generators $s_0,s_1,\ldots,s_n$,
where 
\begin{equation}\label{eq:s0}
s_0 = s_{-\psi^\vee + 1} = t_{\psi} s_{\psi^\vee}
\end{equation}
with $\psi^\vee$ the highest coroot in $R^\vee$, so that $\psi$ is the only dominant
short root in $R$.

The extended affine Weyl group $W^e$ is not a Coxeter group. It is, however,
endowed with a length function $\ell:W^e\to\{0,1,\ldots\}$ which extends the 
one on $W^a$ and is defined as follows \cite[2.2]{macdonald2003affine}. Let $R^a=\{\alpha+n\}_{\alpha\in R, n\in\IZ}$
be the set of affine (real) roots corresponding to $R$, and $R^a_\pm=R_\pm\cup
\{\alpha\pm n\}_{\alpha\in R, n>0}$ those of positive (resp. negative) affine roots.
Then $\ell(w)$ is defined as $|R^a_+\cap w^{-1}R^a_-|$.

\subsection{The subgroup $\Omega\subset W^e$}\label{ss:Omega}

Let $\fa = \left\lbrace h \in \fe^* |1 \geq \cor(h) \geq 0, \cor \in\pos^\vee\right\rbrace$
be the closure of the fundamental alcove, and set
\[\Omega=\Omega(\fa)=\lbrace\omega\in W^e|\omega(\fa) = \fa \rbrace\]
$\Omega$ is the set of elements of $W^e$ of length $0$. Moreover, since $W^a$
acts simply transitively on the set of alcoves, the multiplication map $\Omega\times
W^a\to W^e$ is an isomorphism. It follows that $W^e$ is the semi--direct product
$\Omega\ltimes W^a$, and that $\Omega$ is isomorphic to $P/Q$.

To identify $\Omega$ explicitly, note that $\omega\in W^e$ lies in $\Omega$ if and
only if it permutes the simple affine coroots $\{\cor_i\}_{i=0}^n$, where $\cor_0=-\psi
^\vee+1$. Writing $\omega=t_\lambda w$ then readily implies that $w$ preserves
$\Delta^\vee\cup\{-\psi^\vee\}$, that $w(-\psi^\vee)$ is a special coroot $\cor_i$ \ie
one such that $\lambda_i(\psi^\vee)=1$, and that $\lambda=\lambda_i$ is the
corresponding fundamental weight.

Conversely, if $\cor_i$ is a special coroot, then $\Delta^\vee_i=\Delta^\vee\setminus
\{\cor_i\}\cup\{-\psi^\vee\}$ is a base of $R^\vee$ with highest coroot $-\cor_i$ (see
\eg \cite[3.1]{TLPhD}). It follows that there is a unique $w_i\in W$ such that $w_i\Delta
^\vee=\Delta ^\vee_i$ and $w_i\psi^\vee=-\cor_i$, and that $\omega_i=t_{\lambda_i}
w_i$ is an element of $\Omega$. Thus
\[\Omega\setminus\{1\}=\{ t_{\lambda_i}w_i \}_{i:\lambda_i(\psi^\vee)=1}\]

\subsection{Reflection Representation of \texorpdfstring{$W^{e}_R$}{We}.}
\label{subsec:reflection-rep}

The induced action of $W^e$ on $\bbC[\h^*]=S\h$ given by $^w\negthinspace p=p\circ w^{-1}$
does not preserve the $\IN$--grading, but respects the corresponding filtration $S\h_{\leq k }$.
Its restriction to $S\h_{\leq 1 }=\h\oplus\IC$ is called the {\it reflection representation} of $W^e$,
and is given by
\begin{equation}\label{eqn:reflrep}
	^{s_{\cor+n}}(h,z) = \left(s_{\alpha}(h),z-n \alpha(h)\right)
\end{equation}
\subsection{The trigonometric Cherednik algebra $\Htrig$ \cite{cherednik2005double,opdam2001lectures}} \label{subsec:TCA}

\Omit{Let $R^a = \{\alpha+n\}_{\alpha\in R, n\in\IZ}$ be the set of affine real roots corresponding to $R$,
$k \colon R^a \to \IC$ a $W^a$--invariant map,\valeriocomment{Is such a map necessarily invariant under
$W^e$? Also, we only seem to be using its restriction to $R$ here. Is a $W^a$--invariant function on $R^a$
the same as a $W$--invariant function on $R$?} and set $k_i = k_{\alpha_i}$ for any $0\leq i\leq n$.}

Let $k \colon R\to \IC$ be a $W$--invariant map.\footnote{If $R$ is simply--laced, we identify $k$ with
the number it associates to any root.} The algebra $\Htrig$ is generated by $\bbC[H]$,
$\ICW$ and $\Sh$. It has a faithful (polynomial) representation on $\bbC[H]$ obtained by letting
$\bbC[H]$ act by left multiplication, $W$ by $w e^\lambda
=e^{w\lambda}$, and $\frakh\subset S\h$ by the trigonometric Dunkl operators
\begin{equation}\label{eqn:dunklop}
\h\ni\xi \mapsto T_\xi =
\partial_\xi + \sum_{\alpha \in \pos} k_\alpha \alpha(\xi) \frac{1-s_\alpha}{1-e^{-\alpha}} -  \rho_k(\xi)
\end{equation}
where 
\[\partial_\xi(e^\lambda)  = \lambda(\xi) e^{\lambda}
\aand 
\rho_k = \frac{1}{2} \sum_{\alpha \in \pos} k_\alpha \alpha\in\h^*\]

The subalgebra of $\Htrig$ generated by $W$ and $\IC[H]$ is isomorphic to $W\ltimes\IC[H]$.
The commutation relations between $\Sh$ and $\IC[H]$ and between $\Sh$ and $W$ are, respectively,
\begin{xalignat}{2}
\xi f &= f \xi + \partial_\xi f + \sum_{\alpha \in\pos} k_\alpha \alpha(\xi) \frac{(1-s_\alpha)f}{1 - e^{-\alpha}} s_\alpha 
&\xi \in \mathfrak{h}, f \in \IC[H] 
\label{eqn:xifcommrel} \\
s_{i}  p  & =\, ^{s_i}\negthinspace p\,s_i - k_i\frac{p -\, ^{s_i}\negthinspace p}{\alpha_i^{\vee}} 
&1\leq i \leq n,p \in \Sh
\label{eqn:rpcommrel} 
\end{xalignat}
where $k_i = k_{\alpha_i}$. In particular, $s_{i}\xi=\, s_i(\xi)s_i - k_i\alpha_i(\xi)$ for any $\xi\in\h$.

The subalgebra $\Htrigsub\subset\Htrig$ generated by $W$ and $\Sh$ is the graded (or degenerate)
affine Hecke algebra of $W$.  

\subsection{Alternative presentation of $\Htrig$}\label{subsec:RDCalg}

The algebra $\Htrig$ has a different presentation, which stems from the observation that
$W\ltimes\IC[H]=W\ltimes\IC P$ is the group algebra of the extended Weyl group $W^e$. 
Specifically, the following relations
hold in $\End(\IC[H])$, for any $p\in S\h$, $0\leq i\leq n$ and $\omega\in\Omega$ \cite[Thm. 3.6]
{opdam2001lectures}
\begin{align}
s_{i}  p  	& =\, ^{s_i}\negthinspace p\,s_i - k_i\frac{p -\, ^{s_i}\negthinspace p}{\alpha_i^{\vee}} \label{eqn:s0pcommrel} \\
\omega  p & =\,^\omega\negthinspace p\,\omega
\end{align}
where $\alpha_0^\vee=-\psi^\vee+1$, 
$k_0=k_\psi$.
This presents $\Htrig$ as a (quotient of the) degenerate affine Hecke algebra of the
extended affine Weyl group $W^e$.

The above presentation gives rise to a (faithful) representation of $\Htrig$ on $\IC[\h^*]=\Sh$ by
letting $\Sh$ act by left multiplication, and $W_{R^\vee}^e$ by difference operators. Specifically,
the subgroup $W^a$ acts via the Demazure--Lusztig operators
\begin{equation}\label{eqn:Demaz-op}
S_i = s_i - k_i\frac{1 - s_i}{\alpha_{i}^\vee}\qquad\qquad 0\leq i\leq n
\end{equation}
and $\Omega$ by the restriction of the reflection representation defined in \ref{subsec:reflection-rep}.  

\section{Representations of the degenerate affine Hecke algebra \texorpdfstring{$\Htrigsub$}{Htrigsub}}
\label{sec:induced-representation}

In this section, we review the definition and main properties of the principal series
and covariant representations of $\Htrigsub$.

\subsection{A Relation in \texorpdfstring{$\Htrigsub$}{Htrigsub}}

For any $w\in W$, set $\N(w) = \lbrace \alpha \in\pos\ |\ w\alpha \in R^{-} \rbrace$.

\label{pr:reln in H}
\begin{proposition}
The following holds for any $w\in W$, $h\in\h$ and $p\in\Sh$.
\begin{align}
w h w^{-1} &= ^w\negmedspace h
+\sum_{\beta \in N(w^{-1})}  k_\beta \beta(wh) s_\beta
\label{eqn:hdeg-rel}
\\
p w &= w \pact{w}{p} + \sum_{l(y) < l(w)} y p_{w,y} 
\label{eq:p w}
\end{align}
where $p_{w,y}\in \Sh$ have degree less than $\deg(p)$.
\end{proposition}
\begin{proof}
The first identity follows by induction on the length of $w$, and the second is a direct consequence of the first
(see \cite[Prop. 1.5]{heckman1997dunkl}).
\end{proof}


\subsection{Intertwiners in $\Htrigsub$ \cite{lusztig1989affine,cherednik1991unification}}
\label{subsec:HdegIntertwiners}

For any $1 \leq i \leq n$, define $\tPhi_i\in\Htrigsub$ by
\[
	\tPhi_i = s_i \alpha_i^{\vee} + k_i = - \alpha_i^{\vee} s_i  - k_i
\]
Then, the following holds \cite[Thm. 4.2]{opdam2001lectures}.

\label{pr:inter}
\begin{proposition}	The intertwiners $\tPhi_i$ satisfy 
\begin{equation}\label{eqn-p-phii}
p \tPhi_i  = \tPhi_i \,^{s_i}\negthinspace p
\end{equation}
for any $p\in S\h$, and 
\begin{equation}\label{eqn-phii-squared}
\tPhi_i^2 =   k_i^2-(\alpha_i^{\vee})^2
\end{equation}
Moreover, given two reduced decompositions $s_{i_1} \ldots s_{i_r} = s_{j_1} \ldots s_{j_r}$
of $w\in W$, we have 
\[\tPhi_{i_1} \ldots \tPhi_{i_r} = \tPhi_{j_1} \ldots \tPhi_{j_r}\]
We denote the element of $\Htrigsub$ represented by either side of the equation as $\tPhi_{w}$.
\end{proposition}

\subsection{The Induced Module \texorpdfstring{$\dI{\lambda}$}{Ilambda}.}
\label{subsec:degen-induced-module}

For any $\lambda\in\frakh^*$, let $\IC_\lambda=\IC\bfi{\lambda}$ be the corresponding 1--dimensional
module over $\Sh$, and $\dI{\lambda} = \Ind_{\Sh}^{\Htrigsub} \IC_\lambda$ the \textit{principal series
representation} of $\Htrigsub$. As a $W$--module, $\dI{\lambda}$ is isomorphic to the left regular representation
via $\ICW\ni w\to w\bfi{\lambda}$.

\label{prop:weights-Ilambda}
\begin{proposition}
The set of $\Sh$--weights of $\dI{\lambda}$ is the $W$--orbit of $\lambda$, and the algebraic multiplicity
of any $\mu\in W\lambda$ is the order of the isotropy group of $\lambda$. In particular, if $\lambda$ is
regular, $\dI{\lambda}$ is a semisimple $\Sh$--module, with 1--dimensional weight spaces corresponding
to $W\lambda$.
\end{proposition}
\begin{proof}
It suffices to exhibit a basis of $\dI{\lambda}$ on which the action of $\Sh$ is upper triangular with
diagonal entries $(w\lambda)_{ w \in W }$.  Choose an ordering $w_1, \ldots, w_n$ of the elements
of $W$ by increasing length, so that $w_1 = e$ and $w_n$ is the longest element. We claim that $w_1
\bfi{\lambda},\ldots,w_n \bfi{\lambda}$ is the desired basis.  By Proposition \ref{pr:inter}, 
\[p w_i  = w_i p^{w_i}  + \sum_{j < i} w_j q_j \]
where $q_j\in\Sh$,
hence 
\[p (w_i \mathbf{i_\lambda}) 
= w_i p(w_i^{-1}\lambda) \mathbf{i_\lambda}  + \sum_{j < i} w_j q_j(\lambda) \mathbf{i_\lambda} 
=  p(w_i^{-1}\lambda) (w_i \mathbf{i_\lambda})  + \sum_{j < i} c_j w_j \mathbf{i_\lambda}\]
as claimed.
\end{proof}

\subsection{Intertwiners between induced representations}
\label{subsec:Intertwiners}

The intertwiners $\tPhi_{w}$ give rise to maps between induced representations. Specifically,
for any $\lambda\in\h^*$ and $w\in W$, define a morphism of $\Htrigsub$--modules
\begin{equation}\label{eq:btw I}
\TT{w}{\lambda}:\dI{\lambda}\to \dI{w\lambda}
\qquad\text{by}\qquad
\ilambda\to \tPhi_{w^{-1}}\inu{w\lambda}
\end{equation}
Clearly,
\begin{equation}\label{eq:v w}
\TT{v}{w\lambda}\circ\TT{w}{\lambda}=\TT{vw}{\lambda}
\end{equation}
whenever $\ell(vw)=\ell(v)+\ell(w)$.

\label{pr:det TT}
\begin{proposition}
Identify $\dI{\lambda}$ and $\IC W$ with its standard basis $\{w\}_{w\in W}$. Then, 
\[\det(\TT{w}{\lambda}) = 
\prod_{\alpha\in\pos\cap w^{-1}\Rneg}
\left(k_\alpha^2-\lambda(\alpha^\vee)^2\right)^{|W|/2}\]
In particular, $\TT{w}{\lambda}$ is an isomorphism if, and only if, $\lambda(\alpha^\vee)\neq\pm k_\alpha$
for any $\alpha\in\pos\cap w^{-1}\negr$.
\end{proposition}
\begin{proof}
By \eqref{eq:v w}, it suffices to compute $\det(\TT{s_i}{\lambda})$. To that end note that, for any $w\in W$,
\[\TT{s_i}{\lambda} w\ilambda=
w\left(s_i\alpha_i^\vee+k_{\alpha_i}\right)\inu{s_i\lambda}=
k_{\alpha_i}w\inu{s_i\lambda}-
\lambda(\alpha_i^\vee)w s_i\inu{s_i\lambda}
\]
Thus, $\TT{s_i}{\lambda}$ preserves the $\IC$--span of each of each right $\langle s_i\rangle$--coset in $W$,
and acts on it as the matrix 
\[\begin{pmatrix}
k_{\alpha_i}			&	-\lambda(\alpha_i^\vee) \\
-\lambda(\alpha_i^\vee)	&	k_{\alpha_i}
\end{pmatrix}\]
from which the result follows.
\end{proof}

\label{rk:inverse of TT}
\begin{remark}
The relation \eqref{eqn-phii-squared} readily implies that $\TT_{s_i,\lambda}^{-1}=(k_i^2-
\lambda(\cor_i)^2)^{-1}\cdot\TT_{s_i,s_i\lambda}$. More generally, for any $w\in W$,
\begin{equation}\label{eq:inverse of TT}
\TT_{w,\lambda}^{-1}=
\prod_{\alpha\in\pos\cap w^{-1}\Rneg}
\left(k_\alpha^2-\lambda(\alpha^\vee)^2\right)^{-1}
\cdot\TT_{w^{-1},w\lambda}
\end{equation}
\end{remark}

\subsection{The Covariant Representations \texorpdfstring{$\J_\vartheta^\veps$}{Ilambda}.}
\label{ss:cov deg}

Let $\varepsilon:W\to\{\pm 1\}$ be a one--dimensional character, $\vartheta\in\h^*/W$,
and $\IC_\vartheta^\veps$ the 1--dimensional representations of $\Sh^W\otimes\ICW$,
where $\Sh^W$ acts by evaluation at $\vartheta$ and $W$ by $\veps$. Set
\[\J_{\vartheta}^\veps = \Ind_{\Sh^W \otimes\ICW}^{\Htrigsub} \IC_{\vartheta}^\veps\]
As an $\Sh$--module, $\J_{\vartheta}^\veps$ is isomorphic to $\Sh/I_\vartheta$, where
$I_\vartheta\subset\Sh$ is the ideal generated by $\{f-f(\vartheta)\}_{f\in\Sh^W}$, and
is therefore of dimension $|W|$ \cite[Prop. 3.6]{humphreys1992reflection}. When
$\veps\equiv 1$ is the trivial character, we shall often denote $\J_{\vartheta}^\veps$
by $\J_{\vartheta}$.

Let $q \colon \h^* \to \h^*/W$ be the quotient map.

\label{co:weights of K}
\begin{proposition}
The set of $\Sh$--weights of $\J_\vartheta^\veps$ is $q^{-1}(\vartheta)$, with the algebraic
multiplicity of each $\lambda\in q^{-1}(\vartheta)$ equal to the order of the isotropy
group of $\lambda$. In particular, if $\vartheta$ is regular, $\J_\vartheta^\veps$ is a semisimple
module over $\Sh$, with 1--dimensional weight spaces corresponding to the elements
of $q^{-1}(\vartheta)$.

\end{proposition}
\begin{proof}
Let $\lambda\in\h^*$ be an $S\h$--eigenvalue of $\J_\vartheta^\veps$, and $v$ a corresponding
eigenvector. Since $S\h^W$ acts by evaluation at $\vartheta$ on $\J_\vartheta^\veps$, it follows
that $q(\lambda)=\vartheta$. If $\lambda$ is $k$--regular, $\tPhi_{w}v$ is an eigenvector
with eigenvalue $w\lambda$ for any $w\in W$. This proves the claim under the further 
assumption that $\lambda$ is regular. The general result follows by continuity.
\end{proof}

\subsection{The map $\mathbf{\Rnu{\lambda}^\veps}$}
\label{ss:Rp}

Let $\lambda\in\h^*$. The induced module $\dI{\lambda}\cong\IC W$ has a 1--dimensional
subspace $\dI{\lambda}^\veps$ where $W$ acts by $\veps$. It is spanned by $e_\veps\ilambda$,
where
\begin{equation}\label{eq:epm}
e_\veps=
\frac{1}{|W|} \sum_{w \in W} \veps(w)w\in\IC W
\end{equation} 
is the idempotent corresponding to $\veps$. 
Since $\Sh^W=Z(\Htrigsub)$ acts on $\dI{\lambda}$ by evaluation at $\vartheta=q(\lambda)$,
there is a morphism of $\Htrigsub$--modules,
\begin{equation}\label{eqn:defphi}
\Rnu{\lambda}^\veps:
\J_\vartheta^\veps  \longrightarrow \dI{\lambda}
\qquad 
\ktheta \longmapsto  e_\veps \ilambda 
\end{equation}
which is unique up to a scalar.

\label{re:irred ff}
\begin{remark}
Since $\Rnu{\lambda}^\veps$ is non--zero, and $\J_\vartheta^\veps$ and $\dI{\lambda}$
have the same dimension, it follows that $\J_\vartheta^\veps$ is irreducible if and only
if $\dI{\lambda}$ is, and in turn that $\dI{\lambda}$ is irreducible if and only if all $\{\dI{w\lambda
}\}_{w\in W}$ are. 
\end{remark}

\subsection{The maps $\mathbf{\Rnu{w\lambda}^\veps}$ and intertwiners $\TT{w}{\lambda}$}
\label{ss:Rp and Phi}

The following result relates the maps $\Rnu{w\lambda}^\veps$ corresponding to the $W$--orbit
of $\lambda$. 

\label{pr:TT and Rp}
\begin{proposition}
For any $w\in W$ 
\begin{equation}\label{eq:T R}
\TT{w}{\lambda}\circ\Rnu{\lambda}^\veps = 
\prod_{\alpha\in\pos:w\alpha\in \Rneg}\left(k_\alpha-\veps_\alpha\lambda(\alpha^\vee)\right)
\cdot
\Rnu{w\lambda}^\veps
\end{equation}
where $\veps_\alpha=\veps(s_\alpha)$.
\end{proposition}
\begin{proof}
By \eqref{eq:v w}, it suffices to prove the result for $w=s_i$. We have 
\[\TT{s_i}{\lambda}\circ\Rnu{\lambda}^\veps\,\ktheta
=
e_\veps\tPhi_{s_i}\inu{s_i\lambda}=
e_\veps\left(s_i\alpha_i^\vee+k_{\alpha_i}\right)\inu{s_i\lambda}=
(-\veps_{\alpha_i}\lambda(\alpha_i^\vee)+k_{\alpha_i})\,\Rnu{s_i\lambda}^\veps\ktheta
\]
\end{proof}

\subsection
{Isomorphism of \texorpdfstring{$\J_\vartheta^\veps$}{Jtheta} and \texorpdfstring{$\dI{\lambda}$}{Ilambda}.}
\label{ss:I & J}

The following gives a necessary and sufficient condition for $\J_\vartheta^\veps$ to
be isomorphic to $\dI{\lambda}$.

\label{th:det R}
\begin{theorem}\hfill
\begin{enumerate}
\item Identify $\J_\vartheta^\veps$ with $S\h/S\h^W$ and $\dI{\lambda}$ with $\IC W$
as vector spaces. Then,
\begin{equation}\label{eq:det R}
\det(\Rnu{\lambda}^\veps)=
c \prod_{\alpha\in\pos}\left(k_\alpha-\veps_\alpha\lambda(\alpha^\vee)\right)^{|W|/2}
\end{equation}
for some non--zero constant $c$. 
\item $\J_\vartheta^\veps$ and $\dI{\lambda}$ are isomorphic if and only if $\lambda
(\alpha^\vee)\neq \veps_\alpha k_\alpha$ for any $\alpha\in\pos$.\omitvaleriocomment
{The statement and proof of Theorem \ref{th:det R} should carry over to the affine
Hecke algebra. In fact, Kato has a n\&s criterion for $\I{\Lambda}$ to be isomorphic
to $\J_\Theta^+$.}
\end{enumerate}
\end{theorem}
\begin{proof}
(1) 
We shall show that $\det(\Rnu{\lambda}^\veps)$ has degree bounded above by $|R_+||W|/2$,
and that it is divisible by the \rhs of \eqref{eq:det R}. The latter statement requires that
$k=\{k_\alpha\}_{\alpha\in R}$ be treated as an indeterminate, which we henceforth
assume.

If $p\in S\h$ is of degree $d$, \eqref{eq:p w} implies that the coefficients of $\Rnu{\lambda}
^\veps(p\ktheta)=pe_\veps\ilambda$ in the standard basis $\{w\}$ of $\IC W$ are polynomials
in $(\lambda,k)$ of degree $\leq d$. Thus, if $\{p_i\}$ is a homogeneous basis of $S\h$ as
an $S\h^W$--module, with $\deg(p_i)=n_i$, then $\deg\det(\Rnu{\lambda}^\veps)\leq\sum n_i$.
The latter quantity is equal to $\left.td/dt P(t)\right|_{t=1}$, where $P=\prod_{i}(1-t^{d_i})/(1-t)$
with $d_i$ are the degrees of the basic invariants of $W$. This yields
\[\deg\det(\Rnu{\lambda}^\veps)
\leq
\left.t\frac{d}{dt} \prod_{i}\frac{1-t^{d_i}}{1-t}\right|_{t=1}
=
\sum_i \frac{d_i(d_i-1)}{2}\prod_{j\neq i}d_j
=
|R_+||W|/2
\]
where we used the identities $\prod d_i=|W|$ and $\sum (d_i-1)=|R_+|$ \cite
[Thm 3.9]{humphreys1992reflection}.

Let now $w\in W$. Taking determinants in \eqref{eq:T R}, and using Proposition \ref{pr:det TT} yields
\[\prod_{\alpha\in\pos:w\alpha\in \Rneg}
\left(k_\alpha+\veps_\alpha\lambda(\alpha^\vee)\right)^{|W|/2}
\det(\Rnu{\lambda}^\veps)
=
\prod_{\alpha\in\pos:w\alpha\in \Rneg}\left(k_\alpha-\veps_\alpha\lambda(\alpha^\vee)\right)^{|W|/2}
\det(\Rnu{w\lambda}^\veps)
\]
Choosing $w$ to be the longest element in $W$ shows that $\det(\Rnu{\lambda}^\veps)$
is divisible by the \rhs of \eqref{eq:det R}.

(2) is a direct consequence of (1) and the fact that, up to a scalar, $\Rnu{\lambda}^\veps$
is the only morphism $\J_\vartheta^\veps\to \dI{\lambda}$.
\end{proof}

\begin{remark}
The "if" part of Theorem \ref{th:det R} (2) was proved by Cherednik by a different
method, and under the additional assumption that the isotropy group of $\lambda$ is generated
by simple reflections \cite[Thm 2.14]{cherednik1991unification}. 
In \cite[Thm 4.7]{cherednik1994integration}, Cherednik also states, without proof, that for a given
$\vartheta\in\h^*/W$ there is a $\lambda\in q^{-1}(\vartheta)$ such that $\lambda(\alpha^\vee)\neq
\veps_\alpha k_\alpha$ for any $\alpha\in\pos$, provided $k_\alpha\neq 0$ for any $\alpha$. This
would imply that any covariant representation $\J_\vartheta^\veps$ is isomorphic to an induced one
in this case.\valeriocomment{Point out that (2) is equivalent to the unique $W^\veps$--fixed vector
in $\dI{\lambda}$ being cyclic, and that a necessary and sufficient criterion for the latter was obtained
by Kato for affine Hecke algebra using similar methods.}
\end{remark}

\subsection{Irreducibility of $\dI{\lambda}$ and $\J_\vartheta^\veps$}

Define $\lambda\in\h^*$ to be {\it $k$--regular} if 
\begin{equation}\label{eq:k reg}
\lambda(\alpha^{\vee}) \neq \pm k_{\alpha}\quad\text{ for all } \alpha \in R
\end{equation}
and denote the set of $k$--regular elements in $\h^*$ by $\frakh^*\kreg$.
Similarly, define $\vartheta\in\h^*/W$ to be $k$--regular if some, and therefore
all, $\lambda\in q^{-1}(\vartheta)\subset\h^*$ are.

\label{th:irr I}
\begin{theorem}
The following holds 
\begin{enumerate}
\item $\dI{\lambda}$ is irreducible if and only if $\lambda$ is $k$--regular.
\item $\J_\vartheta$ is irreducible if and only if $\vartheta$ is $k$--regular.
\end{enumerate}
\end{theorem}
\begin{proof}
(1) Assume that $\lambda(\alpha^\vee)=\veps_\alpha k_\alpha$ for some $\alpha\in\pos$
and $\veps_\alpha\in\{\pm 1\}$, and let $\veps:W\to\{\pm 1\}$ be a character such that
$\veps(s_\alpha)=\veps_\alpha$. By Theorem \ref{th:det R}, the non--zero intertwiner
$\Rnu{\lambda}^\veps:\J_\vartheta^\veps\longrightarrow\dI{\lambda}$ is not invertible,
so that $\dI{\lambda}$ is reducible.

Conversely, assume that $\lambda$ is $k$--regular, and let $J\subset \dI{\lambda}$ be a
non--zero submodule.
By Proposition \ref{prop:weights-Ilambda}, $J$ contains an eigenvector for $\h\subset\Htrigsub$
with eigenvalue $w\lambda$, for some $w\in W$, and therefore admits a non--zero map $\dI{w
\lambda}\to J$. The $k$--regularity of $\lambda$ and Theorem \ref{th:det R} imply that the
unique $W$--fixed vector in $\dI{w\lambda}$ is cyclic, and therefore cannot map to zero in $J$.
It follows that its image is the $W$--fixed vector in $\dI{\lambda}$ and, since the latter is cyclic by
Theorem \ref{th:det R}, that $J=\dI{\lambda}$.

(2) follows from (1) and Remark \ref{re:irred ff}.
\end{proof}

\begin{remark}
Part (1) of Theorem \ref{th:irr I} is stated without proof in \cite[Thm. 1.2.2 (a)]{cherednik2005double}.
The analogous statement for affine Hecke algebras is due to Kato \cite[Thm. 2.2]{Kato-irred}.
The proof above is a straightforward adaptation of Kato's.
\end{remark}

\section{Shift operators in the parameter $k$}\label{sec:trig-kz}

\newcommand {\eps}{\varepsilon}

In this section, we generalize the $k$--shift operators of \cite{opdam2001lectures} and \cite{felder1994shift} to the trigonometric
KZ system corresponding to the covariant representation $\J_\vartheta$.  The construction relies on defining analogues of
Matsuo's \cite{matsuo1992integrable} and Cherednik's maps \cite{cherednik1994integration}  for $\J_\vartheta$.

\subsection{The induced and covariant representations of $\Htrig$} 
\label{ss:ind cov trig}

For any $\lambda\in\h^*$ and $\vartheta\in^*/W$, let $\dI{\lambda}$ and $\J_\vartheta$
be the induced and covariant representations of $\Htrigsub$ introduced in \ref{subsec:degen-induced-module}
and \ref{ss:cov deg} respectively, and set
\begin{align*}
\cI_\lambda		&= \Ind^{\Htrig}_{\Htrigsub} \dI{\lambda} = \Ind_{\Sh}^{\Htrig} \IC_\lambda\\
\cJ_{\vartheta}		&= \Ind_{\Htrigsub}^{\Htrig} \J_\vartheta = \Ind_{\Sh^W \otimes\ICW}^{\Htrig} \IC_{\vartheta}
\end{align*}

As $W\ltimes \IC[H]$--modules, $\cI_\lambda$ and $\cJ_{\vartheta}$ are isomorphic to
$\IC[H]\otimes \dI{\lambda}$ and $\IC[H]\otimes \J_\vartheta$ respectively, where $W$
acts on each tensor factor. In particular, as $W$--equivariant sheaves on $H$, they are
isomorphic to the trivial vector bundle with fiber the left regular representation $W$.

\subsection{The trigonometric KZ functor.}\label{subsec:trigKZconn}

The following construction is due to Varagnolo--Vasserot \cite{varagnolo2004double}, and generalizes
the KZ functor for rational Cherednik algebras defined in \cite{ginzburg2003category}.

Let 
\[\Delta = \prod_{\alpha \in\pos} (e^{\alpha/2} - e^{-\alpha/2}) =
e^{\rho} \prod_{\alpha \in\pos} (1 - e^{-\alpha})\in\bbC[H]\]
be the Weyl denominator 
and $H_{\reg} = \lbrace h \in H |\Delta(h) \neq 0 \rbrace$, so that $\bbC[H_{\reg}] = \bbC[H][\Delta^{-1}]$.
Let $\Htrigreg\subset\End(\bbC[H_{\reg}])$ be the subalgebra generated by $\Htrig$ and $\IC[H_{\reg}]$.
Writing \eqref{eqn:dunklop} as
\begin{equation}
	\label{eqn:new dunklop}
	\partial_\xi = T_\xi - \sum_{\alpha \in \pos} k_\alpha \alpha(\xi) \frac{1-s_\alpha}{1-e^{-\alpha}} +  \rho_k(\xi)
\end{equation}
shows that $\Htrigreg=W\ltimes \mathcal{D}_{H_{\reg}}$, where $\mathcal{D}_{H_{\reg}}$ is the Weyl
algebra of $H_{\reg}$. 

Let $V$ be a module over $\Htrigsub$ and $\VV = \mathrm{Ind}_{\Htrigsub}^{\Htrig} V$ the corresponding
induced representation of $\Htrig$. As a $\IC[H]$-module, $\VV$ is a vector bundle over $H$ with fiber $V$.
Localizing to $H_{\reg}$ endows $\VV$ with an action of $\mathcal{D}_{H_{\reg}} \rtimes W = \Htrigreg$. By
\eqref{eqn:new dunklop}, the corresponding $W$--equivariant integrable meromorphic connection has 
logarithmic poles along $\lbrace \Delta = 0 \rbrace$, and the covariant derivative along a translation invariant
vector field $\xi\in\h$ is given by
\[
\nabla_\xi = \xi - \sum_{\alpha \in \pos} k_\alpha \alpha(\xi) \frac{1- s_\alpha}{1-e^{-\alpha}} + \rho_k(\xi)
\]
Identifying $\VV$ with $\bbC[H]\otimes V$ as $\bbC[H]$--modules, and using the fact that $[\nabla_\xi,f]
=\partial_\xi f$ for any $f\in\bbC[H]$, shows that
\begin{equation}\label{eqn:KZconnection}
\nabla_\xi=d-\sum_{\alpha \in \pos} k_\alpha \alpha(\xi) \frac{1- \tfact{s}_\alpha}{1-e^{-\alpha}} + \rho_k(\xi)+ \tfact{\xi} 
\end{equation}
where $\tfact{X}=1\otimes X$ denotes the $\IC[H]$--linear, fibrewise action of $X\in\Htrigsub$.

Denote the connections on the vector bundles $\cJ_{\vartheta}$ and $\cI_\lambda$ by $\Kconn$
and $\Iconn$ respectively, and the corresponding local systems on $H_{\reg}/W$ as $\KL$ and
$\IL$.

\subsection{The KZ connection \cite{matsuo1992integrable}}

Matsuo introduced the following meromorphic connection $\nabla\KZ$ on the holomorphically
trivial bundle $\cW$ over $H$ with fibre $\IC W$. Let $\Endd(\IC W)$ be the algebra of diagonal
endomorphisms in the standard basis $\{w\}_{w\in W}$ of $\IC W$ and, for any $\alpha\in R$, define
\[\eps_\alpha\in\Endd(\IC W)\qquad\text{by}\qquad\eps_\alpha\,w=-\sign(w^{-1}\alpha)w\]
Fix $\lambda\in\h^*$, and define a linear map 
\[e_\lambda:\h\to\Endd(\IC W)\qquad\text{by}\qquad e_\lambda (\xi)w=(w\lambda,\xi)w\]
Given a $W$--invariant set of weights $k=\{k_\alpha\}_{\alpha\in R}$, the connection
$\nabla\KZ$ is defined by
\[\nabla\KZ=
d+\half{1}\sum_{\alpha\in\pos}k_\alpha d\alpha
\left(\frac{1+e^{-\alpha}}{1-e^{-\alpha}}\otimes (1-s_\alpha)+1\otimes s_\alpha\eps_
\alpha\right)-de_\lambda\]

The following identifies Matsuo's connection with the one arising from the trigonometric
KZ functor applied to the induced representation of the Cherednik algebra $\Htrigneg$
with {\it negated} weights $-\lambda$ and $-k=\{-k_\alpha\}$.\footnote{The negation of
$k$ can be avoided if the Dunkl operators \eqref{eqn:dunklop} are defined with $-k$
instead of $k$, that is as $T_\xi = \partial_\xi - \sum_{\alpha \in \pos} k_\alpha \alpha(\xi) \frac{1-s_\alpha}
{1-e^{-\alpha}} +  \rho_k(\xi)$, as done for example in \cite{cherednik1994integration}.}

\label{le:KZ KZ}
\begin{lemma}
The connection $\nabla\KZ$ coincides with the trigonometric KZ connection $\Iconn[-\lambda,-k]$ arising from the trigonometric Cherednik algebra for the weight $-k$.
\end{lemma}
\begin{proof}
Since $\frac{1+x}{1-x}=\frac{2}{1-x}-1$ and $\rho_k=\half{1}\sum_{\alpha\in\pos}k_\alpha\alpha$
we have, for any $\xi\in\h$
\[\nabla\KZ_\xi=\partial_\xi+
\sum_{\alpha\in\pos}k_\alpha\alpha(\xi)
\frac{1-s_\alpha}{1-e^{-\alpha}}
-\rho_k(\xi)+
\sum_{\alpha\in\pos}k_\alpha\alpha(\xi)\,s_\alpha\delta_\alpha
-e_\lambda(\xi)
\]
where $\delta_\alpha=\half{1}(1+\eps_\alpha)\in\Endd(\IC W)$ maps $w$ to $\delta_{w^{-1}
\alpha\in\negr}w$.

Note now that the action of $\xi\in H_{k'}$ on $I_\mu$ is given by
\[
\xi w\imu
=
w\left(^{w^{-1}}\xi+\sum_{\beta\in\pos:w\beta\in\negr} k'_\beta w\beta(\xi)s_\beta\right)\imu\\
\Omit{=
\left(e_\mu(\xi)
-
\sum_{\alpha\in\pos:w^{-1}\alpha\in\negr} k'_\alpha\alpha(\xi)s_\alpha
\right)w\imu
}
=
\left(e_\mu(\xi)
-
\sum_{\alpha\in\pos} k'_\alpha\alpha(\xi)s_\alpha\delta_\alpha
\right)w\imu
\]
where we used \eqref{eqn:hdeg-rel}. Comparing with \eqref{eqn:KZconnection}
shows that $\nabla\KZ_\xi=\Iconnvf[-\lambda,-k]{\xi}$.
\end{proof}

We denote by $\cW(\lambda,k)$ the local system on $H_{\reg}/W$ defined by $\nabla\KZ$.
By Lemma \ref{le:KZ KZ}, $\cW(\lambda,k)=\IL[-\lambda,-k]$.

\subsection{The hypergeometric system}

For any $h_0 \in H_{\reg}$, let $\O_{h_0}$ be the germ of holomorphic functions at $h_0$,
set $\O_{Wh_0} = \bigoplus_{w \in W}\O_{wh_0}$, and let $\O^W_{Wh_0} \subseteq
\O_{Wh_0}$ be the germ of $W$--invariant functions at $Wh_0$. We denote by $\symO
{\cdot} \colon \O_{h_0} \to \O_{Wh_0}^W$ the isomorphism given by $f \mapsto \{wf
\}_{w\in W}$.

By \eqref{eqn:dunklop}, the Dunkl operators act on $\O_{Wh_0}$. The corresponding
action of $\Sh^W$ restricts to one on $\O^W_{Wh_0}$, which is given by differential
operators. Fix $\vartheta \in \mathfrak {h}^*/W$. The hypergeometric system is defined
by
\begin{equation}\label{eqn:CSsystem}
\Cs{\vartheta}{k} = \left\{f \in \O^W_{Wh_0}\left|\,p f = p(\vartheta)f,\,p \in \Sh^W\right.\right\}
\end{equation}
By \cite[Cor. 4.1.8]{heckman1995harmonic}, $\Cs{\vartheta}{k}$ is a local system over
$H_{\reg}/W$ of rank $|W|$.

\subsection{The Matsuo map \texorpdfstring{$m_{\dI{\lambda}}$}{mIlambda}}\label{subsec:matsuo-isomorphism}

Let $\germ{\cI_\lambda}$ be the germ of holomorphic sections of $\cI_\lambda$ at $h_0$.
In \cite{matsuo1992integrable}, Matsuo defines an $\O_{h_0}$--linear map
\omitvaleriocomment
{Not quite correct: the identfication is between $\calL (-\lambda,-k) \longrightarrow
\Cs{\vartheta}{k}$: $k$ and $\lambda$ are {\bf negated}. Give references for this (Matsuo,
Opdam Acta's paper, Heckman's Bourbaki seminar).}
$m_{\dI{\lambda}}\colon (\cI_\lambda)_{h_0} \longrightarrow \O_{Wh_0}^W$. 
Trivializing $\germ{\cI_\lambda}$ as $\O_{h_0} \otimes \dI{\lambda}$, $m_{\dI{\lambda}}$ is given by
\[
\sum_{w \in W} f_w w\pos 
\longmapsto \symO{\sum_{w \in W} f_w}
\]

\Omit{
Let $\epsilon_{+} = |W|^{-1}\sum_{w} w$ and $\epsilon_{-} = |W|^{-1}\sum_{w} \mathrm{sgn}(w)w$
be the idempotents corresponding to the trivial and sign representations of $W$.}
Let $\epsilon_+\in\IC W$ be the symmetric idempotent \eqref{eq:epm}. Then, if $f \in\O
_{h_0}\otimes \dI{\lambda}$,
\begin{equation}\label{eq:matsuo 2}
m_{\dI{\lambda}}(f) = [\eta_+\epsilon_{+} f]
\end{equation}
where $\eta_+:\dI{\lambda}^W\to\IC$ is the isomorphism given by $\epsilon_+\ilambda\to 1$.

Let $q:\h^*\to\h^*/W$ be the projection, and set $\pm\vartheta=q(\pm\lambda)$. Then, $m_
{\dI{\lambda}}$ restricts to a morphism of local systems $\IL \longrightarrow\Cs
{-\vartheta}{-k}$,\omitvaleriocomment{You had stated that $m_{\dI{\lambda}}$ is a morphism $\IL \to\Cs{\vartheta}{k}$, but since Matsuo's map is really a map $\cW(\lambda,k)\to
\Cs{\vartheta}{k}$, and $\IL=\cW(-\lambda,-k)$ by Lemma \ref{le:KZ KZ}, both
statements cannot be true. On the other hand, I thought you checked your statement in rank
1 with Mathematica. In that case $-\vartheta=\vartheta$, but still $k\neq-k$. Can we reconcile
this?} which is an isomorphism if and only if $\lambda(\alpha^\vee)\neq -k_\alpha$ for all $\alpha\in\pos$
\cite{matsuo1992integrable,opdam2001lectures}.\footnote{Matsuo's isomorphism is formulated
in terms of the local system $\cW(-\lambda,-k)$ which is identified with $\IL$
via Lemma \ref{le:KZ KZ}.}


\subsection{The radial part maps $\Radpm$ \cite[Lemma 5.10]{opdam2001lectures}}

The following will be needed in \ref{subsec:matsuomap} to spell out the analogue of
Matsuo's map for the representation $\J_\vartheta$.

\label{pr:rad-parts}
\begin{proposition}
The following holds. 
\begin{enumerate}
\item \label{lem-item-2} 
$\epsilon_{\pm} \Htrigsub \epsilon_{\pm} = \Sh^W \epsilon_{\pm}$. 
\item The map $\Sh^W\ni p\mapsto p \epsilon_{\pm}\in \Sh^W\epsilon_\pm$
is an algebra isomorphism. It therefore induces a map of $\Sh^W$--modules 
\[\Radpm:\Htrigsub\longrightarrow \Sh^W
\quad\text{defined by}\quad
\Radpm(f)\epspm=\epspm f\epspm\]
\end{enumerate}
\Omit{
Set now \[\pi_{\pm} = \prod_{\alpha \in \pos} \left( \alpha^{\vee} \pm k_\alpha \right)\in \Sh \subset \Htrigsub\]
Then,\hfill
\begin{enumerate}[(c)]
\item \label{lem-item-1} $\epsilon_{\mp} \pi_{\pm} \epsilon_{\pm} = \pi_{\pm} \epsilon_{\pm}$.
\item[(d)] \label{lem-item-3} 
$\epsilon_{\mp} \Htrigsub \epsilon_{\pm} = \Sh^{W} \pi_{\pm} \epsilon_{\pm}$.  
\item[(e)] The map $\Sh^W\ni h \mapsto h \pi_{\pm} \epsilon_{\pm}\in\Sh^W \pi_{\pm} \epsilon_{\pm}$
is a linear isomorphism. It therefore induces a map of $\Sh^W$--modules 
\[\pmRad:\Htrigsub \longrightarrow \Sh^W
\quad\text{by}\quad
\pmRad(f)\pi_\pm\epspm=\epsmp f\epspm\]
\end{enumerate} 
}
\end{proposition}

\subsection{The Matsuo map \texorpdfstring{$m_{\J_\vartheta}$}{mJtheta}}
\label{subsec:matsuomap}

We now define an analog of Matsuo's map for $\J_\vartheta$. 

Let $\lambda\in\h^*$ and set $\vartheta=q(\lambda)\in\h^*/W$.
\omitvaleriocomment
{I am denoting the projection by $q$ since we use $p$ for elements of $\Sh$ and $\pi$ for the
elements $\pi_\pm$. Make sure the notation is consistent in the paper.} 
Let $\Rp \colon \J_\vartheta \to \dI{\lambda}$ be the $\Htrigsub$--linear map \eqref{eqn:defphi},
and set $\cRp = \Ind_{\Htrigsub}^{\Htrig} \Rp \colon \cJ_\vartheta
\to \cI_\lambda$, so that
\[\cRp:\sum_i f_i p_i \jtheta \mapsto
\sum_i f_i p_i  \epsilon_{+} \ilambda \]
where $f_i \in \IC[H]$ and $p_i \in \Sh$.
Let $m_{\J_\vartheta}:\germ{\J_\vartheta}\to\O^{W}_{Wh_0}$ be the $\O_{h_0}$--linear map defined
by the composition
\[
	\begin{tikzcd}
		(\cJ_\vartheta)_{h_0}  \arrow[rd,"m_{\J_\vartheta}"'] \arrow[r,"\cRp"]  & (\cI_\lambda)_{h_0} \arrow[d,"m_{\dI{\lambda}}"]   \\
		& \O^{W}_{Wh_0}
	\end{tikzcd}
\]

The following gives an explicit description of $m_{\J_\vartheta}$. Notice first that the linear
map $\Sh\to\IC$, $p\to\Radp(p)(\vartheta)$ descends to a $W$--invariant linear map $\J_
{\vartheta}\mapsto\IC$, which will be denoted by the same symbol.

\label{prop:desc-mJ}
\begin{proposition}\hfill
\begin{enumerate}
\item The map $m_{\J_\vartheta}$ 
is the $\O_{h_0}$--linear
map induced by $\Radp(\cdot)(\vartheta):\J_{\vartheta}\mapsto\IC$. In particular, it is independent
of the choice of $\lambda\in q^{-1}(\vartheta)$. 
\item 
$m_{\J_\vartheta}$ restricts to a morphism of local systems $\KL \to \Cs{-\vartheta}{-k}$.
\end{enumerate}
\end{proposition}
\begin{proof}
For any $f_i\in\IC[H]$ and $p_i\in\Sh$, we have
\begin{multline*}
m_{\J_\vartheta} \left( \sum_i f_i p_i \jtheta \right) 
=  
m_{\dI{\lambda}} \left( \sum_i f_i p_i  \epsilon_{+} \ilambda \right)\\        
=
\sum_i \symO{f_i\,\eta_+\epsilon_{+} p_i  \epsilon_{+} \ilambda}
=
\sum_i \Radp(p_i)(\vartheta)\symO{f_i\,\eta_+\epsilon_{+}\ilambda}
=
\sum_i \symO{f_i}\Radp(p_i)(\vartheta)
\end{multline*}
where we used \eqref{eq:matsuo 2}.
\Omit{
\[m_{\J_\vartheta} \left( \sum_i f_i p_i \jtheta \right) 
=  
m_{\dI{\lambda}} \left( \sum_i f_i p_i  \epsilon_{+} \ilambda \right)          
=
\sum_i \symO{f_i}\symO{\eta_+\epsilon_{+} p_i  \epsilon_{+} \ilambda}
=
\sum_i \symO{f_i}\Radp(p_i)(\vartheta)\]
}
The map $m_{\J_\vartheta}$ restricts to a map of local systems because $m_
{\cI_\lambda}$ restricts to a map $\IL \longrightarrow \Cs
{-\vartheta}{-k}$, and $\cRp$ is a morphism of vector bundles
with connection.
\end{proof}


%
%
%
%

\subsection{Equivariance under $\Omega^\vee=P^\vee/Q^\vee$}\label{ss:equi 1}

The torus $H$ admits a translation action of $\Omega^\vee$ which commutes
with that of $W$ since $s_i \nu =\nu\mod Q^\vee$ for any $\nu\in P^\vee$, and preserves $H_{\reg}$ since $Q
(P^\vee)=\IZ$. This induces a $W$--equivariant isomorphism of germs $\O_{Wh}\to\O_{W\nu h}$ for any $h\in
H_{\reg}$ and $\nu\in P^\vee$, which commutes with the action of Dunkl operators since the derivatives and
$\rho_k$ term are translation invariant, while the reflection part only uses the action of $W$ and of functions
in $\IC Q$, which are invariant under $P^\vee$.
Restricting to $W$--invariants therefore yields an isomorphism of hypergeometric systems
\[\Cs{\vartheta}{k}_{Wh}\to\Cs{\vartheta}{k}_{W\negthinspace\nu h}\]
so that $\Cs{\vartheta}{k}$ is a $\Omega^\vee$--equivariant local system on $H_{\reg}/W$.

Similarly, the KZ connections $\Kconn$ and $\Iconn$ are invariant under translation
by $\Omega^\vee$, and Matsuo's morphisms $m_{\dI{\lambda}},m_{\J_\vartheta}$ are readily seen
to be equivariant \wrt $\Omega^\vee$.

\subsection{Conditions for $m_{\J_\vartheta}$ to be an isomorphism}

\label{prop:mJtheta-iso}
\begin{proposition}
If $\vartheta$ is $k$--regular, the map $m_{\J_\vartheta} \colon \KL
\to\Cs{-\vartheta}{-k}$ is an isomorphism.\valeriocomment{State this as an iff?}
\end{proposition}
\begin{proof}
The proof closely follows \cite[Thm. 4.6]{cherednik1994integration}. Since both $\KL$ and $\Cs{-\vartheta}{-k}$
are of rank $|W|$, it is sufficient to prove the injectivity of $m_{\J_\vartheta}$.

Let $f\in\germ{\KL}$ be a local horizontal section such that $m_{\J_\vartheta}(f) = 0$. We first claim that
$m_{\J_\vartheta}((\Htrigsub)\sff f) = 0$, where the superscript $\mathsf{f}$ denotes the fibrewise action of $\Htrigsub$.
Since $m_{\J_\vartheta}$ is invariant under $\tfact{W}$, it suffices to show that $m_{\J_\vartheta}(\tfact{p}f) = 0$ for
all $p \in \Sh$. Assume by induction that this holds for all $p$ with $\mathrm{deg}(p) \leq m $.  Then, $\Kconnvf[\vartheta,k]{\xi} f=0$
implies that
\[-\partial_\xi f =\tfact{\xi} f - \sum_{\alpha \in \pos} k_\alpha \alpha(\xi) \frac{1- \tfact{s}_\alpha}{1-e^{-\alpha}} f + \rho_k(\xi) f\]
Applying $\tfact{p}\in\Sh^{\leq m}$ gives 
\[\begin{split}
-\partial_\xi (\tfact{p}f)
&=  \tfact{p}\tfact{\xi} f - \sum_{\alpha \in \pos} k_\alpha \alpha(\xi) \frac{\tfact{p} f - \tfact{p}\tfact{s}_\alpha f}{1-e^{-\alpha}} + \rho_k(\xi) \tfact{p} f \\
&=  \tfact{p}\tfact{\xi} f -  \sum_{\alpha \in \pos} k_\alpha \alpha(\xi) \frac{\tfact{p} f - \sum_{y \in W} \tfact{y}\tfact{p}_{s_\alpha,y} f}{1-e^{-\alpha}} + \rho_k(\xi) \tfact{p} f
\end{split}\]
where the second equality follows from Proposition \ref{pr:reln in H}, $y\in W$ is such that $\ell(y)\leq\ell(s_\alpha)$, and $\deg
(p_{s_\alpha,y}) \leq m$.  Applying now $m_{\J_\vartheta}$ implies that
\begin{multline*}
m_{\J_\vartheta}( \tfact{p}\tfact{\xi} f ) \\
=
-\partial_\xi m_{\J_\vartheta}(\tfact{p}f)
 +  \sum_{\alpha \in \pos} k_\alpha \alpha(\xi) \frac{m_{\J_\vartheta}(\tfact{p} f) - \sum_{y \in W}m_{\J_\vartheta}( \tfact{y}\tfact{p}_{s_\alpha,y} f)}{\symO{1-e^{-\alpha}}}
 - \rho_k(\xi) m_{\J_\vartheta}(\tfact{p} f)     
\end{multline*}
The inductive hypothesis and the $\tfact{W}$--invariance of $m_{\J_\vartheta}$ then imply that $m_{\J_\vartheta}( \tfact{p}\tfact{\xi} f )
=0$. Since $\xi\in\h$ is arbitrary, it follows that $m_{\J_\vartheta}(\tfact{(\Sh^{\leq m+1})}f )= 0$ as claimed.

Let now $M_f = \Htrigsub f(h_0)\subset\J_\vartheta$ be the submodule generated by $f(h_0)$ in the
fiber of $\cJ_\vartheta$ over $h_0$. Since $\vartheta$ is $k$--regular, $\J_\vartheta$ is irreducible by
Theorems \ref{th:irr I} and \ref{th:det R} and $M_f$ is either $0$ or $\J_\vartheta$. We have shown
above that $\Rad(M_f)(\vartheta) = 0$.  Since $\Rad(1)(\vartheta) = 1$, $M_f$ must be $0$, whence
$f(h_0) =0$ and therefore $f=0$ since $\Kconnvf[\vartheta,k]{\xi} f=0$.
\end{proof}

\subsection{The radial part maps $\pmRad$ \cite[Lemma 5.8]{opdam2001lectures}}\label{ss:PRad}

The following is a counterpart to Proposition \ref{pr:rad-parts}, and will be used in \ref{subsec:chJtheta}
to spell out the analogue of Cherednik's map for the representation $\J_\vartheta$.

\label{pr:+rad-parts}
\begin{proposition}
Set 
\[\pi_{\pm} = \prod_{\alpha \in \pos} \left( \alpha^{\vee} \pm k_\alpha \right)\in \Sh \subset \Htrigsub\]
Then,\hfill
\begin{enumerate}
\item \label{lem-item-1} $\epsilon_{\mp} \pi_{\pm} \epsilon_{\pm} = \pi_{\pm} \epsilon_{\pm}$.
\item \label{lem-item-3} 
$\epsilon_{\mp} \Htrigsub \epsilon_{\pm} = \Sh^{W} \pi_{\pm} \epsilon_{\pm}$.  
\item The map $\Sh^W\ni h \mapsto h \pi_{\pm} \epsilon_{\pm}\in\Sh^W \pi_{\pm} \epsilon_{\pm}$
is a linear isomorphism. It therefore induces a map of $\Sh^W$--modules 
\[\pmRad:\Htrigsub \longrightarrow \Sh^W
\quad\text{defined by}\quad
\pmRad(f)\pi_\pm\epspm=\epsmp f\epspm\]
\end{enumerate} 
\end{proposition}

\subsection{The Cherednik map $ch_{\J_\vartheta}$}
\label{subsec:chJtheta}

\newcommand {\DDelta}{\Delta^{-1}}

In \cite{cherednik1994integration}, Cherednik defines an $\O_{h_0}$--linear map $ch_{\dI{\lambda}}:
(\cI_\lambda)_{h_0} \longrightarrow \O_{Wh_0}^W$ which restricts to $\IL \longrightarrow
\Cs{-\vartheta}{-k\csign 1}$, where $(-k\csign 1)_\alpha=-k_\alpha\csign 1$.\omitvaleriocomment{I am 
not sure about the sign of the shift, because I don't know whether Cherednik's original map is a
morphism $\cW(\lambda,k)\to\Cs{\vartheta}{k+1}$ or $\cW(\lambda,k)\to\Cs{\vartheta}{k-1}$. I think
you had checked which one it was, but then negated the sign because your rank 1 computations with Mathematica didn't agree with
Cherednik's sign of the shift. What is the correct sign of the shift here, given that, by Lemma \ref
{le:KZ KZ}, $\cW(\lambda,k)=\IL[-\lambda,-k]$? Note that a) I created a macro in this file called
$\backslash$csign, which produce a $\pm 1$ for now, and we can instantiate to the correct sign from
Cherednik's paper. b) If his morphism is $\cW(\lambda,k)\to\Cs{\lambda}{k+\epsilon}$, where
$\epsilon=\pm 1$, then the map we are using is $\IL=\cW(-\lambda,-k)
\to\Cs{-\lambda}{-k+\epsilon}$ i.e. the sign of 'our' shift starting from $\IL$ or
$\KL$ is the same as Cherednik's. Do you agree?}
It is similar to Matsuo's map, but defined
in terms of the sign character of $W$ by
\[f=\sum_{w \in W} f_w w \ilambda
\longmapsto
\symO{  \DDelta\sum_{w \in W} (-1)^{l(w)} f_w}=
\symO{  \DDelta\eta_-\epsilon_- f}
\]
where $\epsilon_-$ is the sign idempotent \eqref{eq:epm}, and $\eta_-:\epsilon_{-}\dI{\lambda}\to\IC$
the isomorphism mapping $\epsilon_-\ilambda$ to $1$.

Analogously to \ref{subsec:matsuomap}, define the $\O_{h_0}$--linear map $\wt{ch}_{\J_\vartheta}:
(\cJ_\vartheta)_{h_0}\to\O_{W h_0}^W$ by the composition
\[
	\begin{tikzcd}
		(\cJ_\vartheta)_{h_0}  \arrow[rd,"\wt{ch}_{\J_\vartheta}"'] \arrow[r,"\cRp"]  & (\cI_\lambda)_{h_0} \arrow[d,"ch_{\dI{\lambda}}"]   \\
		& \O_{W h_0}^W
	\end{tikzcd}
\]

%
\begin{proposition}
The map $\wt{ch}_{\J_\vartheta}:\germ{\J_\vartheta}\to\O^{W}_{Wh_0}$ is the $\O_{h_0}$--linear
map induced by the map
\[\J_{\vartheta}\mapsto\IC[\DDelta]
\qquad
p\jtheta\to\pi_-(\lambda)\ \pRad(p)
(\vartheta)\symO{\DDelta}\]
\Omit{
	The map $\wt{ch}_{\J_\vartheta}$ is the $\O_{h_0}$-linear map determined by 
	\begin{align*}
		\J_{\vartheta} & \longrightarrow  \bbC                                                     \\
		p          & \longmapsto   \pi_{-}(\lambda)\pRad(p)(\vartheta)
	\end{align*}
	That is, it induces a map 
	\begin{align*}
		\wt{ch}_{\J_\vartheta} \colon (\cJ_{\vartheta})_{h_0} & \longrightarrow  \O_{Wh_0}^{W}                                                     \\
		\sum_i f_i p_i          & \longmapsto \symO{ \sum_i f_i  \pi_{-}(\lambda)\pRad(p_i)(\vartheta) }
	\end{align*}
}	
\end{proposition}
\begin{proof}
Let $f_i\in\IC[H]$ and $p_i\in\Sh$, then
\[\begin{split}
\wt{ch}_{\J_\vartheta} \left( \sum_i f_i p_i \jtheta \right) 
&= 
ch_{\dI{\lambda}} \left( \sum_i f_i p_i  \epsilon_{+} \ilambda \right)\\           
&=
\sum_i [f_i] \symO{\DDelta \eta_- \epsilon_{-} p_i  \epsilon_{+} \ilambda}\\
&=
\sum_i [f_i] \symO{\DDelta\eta_-  \epsilon_{-} \pRad(p_i) \pi_{+}  \epsilon_{+} \ilambda  }\\
&=
\sum_i [f_i]\ \pRad(p_i)(\vartheta) \symO{\DDelta\eta_-  \epsilon_{-} \pi_{+}  \epsilon_{+} \ilambda  }
\end{split}\]		                                                                      
where the third equality follows from Proposition \ref{pr:+rad-parts}, and the fourth from the fact that $\pRad
(p_i)$ is central.

By Proposition \ref{pr:+rad-parts}, $\pi_+ \epsilon_+ \ilambda$ is an anti--invariant element of
$\dI{\lambda}$, and is therefore equal to $c(\lambda) \epsilon_- \ilambda$ for some constant
$c(\lambda)$. To determine it, we compare the coefficients of the term $w_0 \ilambda$.
Applying Proposition \ref{pr:reln in H} for $\pi_+$ and each $w \in W$ gives
\[\pi_+ \epsilon_+ \ilambda 
= \frac{1}{|W|} \left( w_0 (\pi_+)^{w_0}( \lambda) + \sum_{l(y) < l(w_0)} y p_y(\lambda) \right) \ilambda 
\]
for some $p_y \in \Sh$.  Note that 
	\[
	(\pi_+)^{w_0} = \left(\prod_{\alpha \in \pos} \left(\alpha^\vee + k_\alpha \right) \right)^{w_0} = \prod_{\alpha \in \pos} \left(-\alpha^\vee + k_\alpha \right) = (-1)^{l(w_0)} \pi_-
	\]
	Thus
	\begin{align*}
	\pi_+ \epsilon_+ \ilambda &= 1/|W| (-1)^{l(w_0)} \pi_-(\lambda) w_0 \ilambda + \ldots \\
	\epsilon_- \ilambda &= 1/|W| (-1)^{l(w_0)} w_0 \ilambda + \ldots 
	\end{align*}
where the lower order terms $\ldots$ lie in $\bigoplus_{w\neq w_0}\IC w\ilambda$. This implies that
$\pi_+ \epsilon_+ \ilambda  = \pi_-(\lambda) \epsilon_- \ilambda$, and therefore that
\[ \wt{ch}_{\J_\vartheta} \left( \sum_i f_i p_i \jtheta \right) 
=
\pi_-(\lambda) \sum_i [f_i]\ \pRad(p_i)(\vartheta) \symO{\DDelta}
\]
\end{proof}

Note that, unlike the map $m_{\J_\vartheta}$, $\wt{ch}_{\J_\vartheta}$ depends on the choice of
$\lambda\in q^{-1}(\vartheta)$ and can be identically zero, due to the factor $\pi_-(\lambda)$. To
remedy both of these issues, define the $\O_{h_0}$--linear map 
\[ch_{\J_\vartheta} \colon (\cJ_\vartheta)_{h_0}\to\O_{Wh_0}^W
\qquad\text{by}\qquad
p\jtheta\to
\pRad(p)(\vartheta)\symO{\DDelta}\]
Then, $ch_{\J_\vartheta}$ is independent of $\lambda$, anti--invariant \wrt the fibrewise action
of $W$, and such that $\pi_-(\lambda)ch_{\J_\vartheta}=\wt{ch}_{\J_\vartheta}$. Moreover, it
restricts to a morphism of local systems $\KL\longrightarrow\Cs{-\vartheta}{-k\csign 1}$.

\subsection{Equivariance of $ch_{\J_{\vartheta}}$ under $\Omega^\vee$}\label{ss:equi 2}

Unlike Matsuo's map, $ch_{\J_{\vartheta}}$ is not equivariant \wrt the translation action of $\Omega^\vee
$ because it involves $\Delta=e^\rho\prod_{\alpha\in R_+}(1-e^{-\alpha})\in e^\rho\cdot\IC Q$, and $\rho
\notin Q$ in general.

To measure its lack of equivariance, note that $2\rho\in Q$, so that for any $\nu\in\Omega^\vee$, $\nu
e^{\rho}=\eps_\rho(\nu) e^{\rho} $, where
\begin{equation}\label{eq:eps rho}
\eps_\rho(\nu)=e^{2\pi\ii \rho(\nu)}\in\{\pm 1\}
\qquad\text{since}\qquad
2\rho(\nu)\in\IZ
\end{equation}
Thus, $\eps_\rho$ is a $\IZ/2\IZ$--valued character of $\Omega^\vee$, and $ch_{\J_\vartheta} \colon (\cJ_
\vartheta)_{h_0} \to \O_{Wh_0}^W$ is equivariant \wrt $\Omega^\vee$, provided one of the two actions 
is twisted by $\eps_\rho$.

Note that $\eps_\rho$ may be trivial, and that this holds if and only if $\rho\in Q$. For example, in
types
\begin{align*}
\text{$\sfA_n$:}\,\,
2\rho=&n\alpha_1+\cdots+i(n-i+1)\alpha_i+\cdots+n\alpha_n\in 2Q\Leftrightarrow n\in 2\IN\\
\text{$\sfB_n$:}\,\,
2\rho=&(2n-1)\alpha_1+\cdots+i(2n-i)\alpha_i+\cdots+n^2\alpha_n\notin 2Q\\
\text{$\sfC_n$:}\,\,
2\rho=	&2n\alpha_1+\cdots+i(2n-i+1)\alpha_i+\cdots\\
		&\cdots +(n-1)(n+2)\alpha_{n-1}+\half{1}n(n+1)\alpha_n\in 2Q\Leftrightarrow n=0,3\mod 4\\
\text{$\sfD_n$:}\,\,
2\rho=	&2(n-1)\alpha_1+\cdots+2(in-\frac{i(i+1)}{2})\alpha_i+\cdots\\
		&\cdots+\frac{n(n-1)}{2}(\alpha_{n-1}+\alpha_n)\in 2Q\Leftrightarrow n=0,1\mod 4
\end{align*}

\subsection{Conditions for $ch_{\J_\vartheta}$ to be an isomorphism}

\Omit{
\label{lem:fisRad}
\begin{lemma}
Identifying $\J_\vartheta^*$ with linear maps on $S$ which are 0 on $M_\vartheta$, the element $f$ in $\J_\vartheta^*$ can be identified with the restriction 
  \[
  {}^+\mathrm{Rad} |_S \colon S \to S^W.
\]  
\end{lemma}
\begin{proof}
	   By Proposition \ref{pr:+rad-parts}, ${}^+\mathrm{Rad}(\pi_+) = 1$.  Since up to lower degree terms $\pi_+ \equiv \delta$ we also have $f(\pi_+) = 1$.  Now let $p$ be homogeneous of $\deg(p) < \deg(\delta)$.  By Proposition \ref{pr:+rad-parts}, $\epsilon_- p \epsilon_+ = q \pi_+ \epsilon_+$ for $q \in S^W$.  Since $\deg(q) + \deg(\pi_+) < \deg(p)$ we must have $q = 0$.  Thus ${}^+\mathrm{Rad}(p) = 0$. 
\end{proof}
}

\label{prop:chiso}
\begin{proposition}
The map $ch_{\J_\vartheta}$ restricts to an isomorphism of local systems $\KL \to \Cs
{-\vartheta}{-k\csign 1}$ for any $\vartheta\in\h^*/W$ and $k$.\valeriocomment{It is natural to ask
whether the HG system, which is (the solutions of) a $W$--equivariant $\mathcal{D}$--module on
$H_{\reg}$, can be obtained from the trigonometric KZ functor applied to a category $\mathcal O$
representation $M_\vartheta$ of $\Htrig$. Matsuo gives a map $KZ(\mathcal I_\lambda)\to HG(
\vartheta;k)$ but it is not always invertible, so isn't a completely satisfactory answer to this question.
On the other hand, Proposition \ref{prop:chiso} seems to say that ${\mathcal K}_\vartheta$ is the
correct answer to this question. If so, remark this explicitly, and promote this proposition to a Theorem.
Also, double check that Cherednik may not have proved something along those lines in {\it Integration...}}

\end{proposition}
\begin{proof}
If $\pi_-(\lambda)\neq 0$, then $ch_{\J_\vartheta}=\pi_-(\lambda)^{-1}\wt{ch}_{\J_\vartheta}=\pi_-(\lambda)
^{-1}ch_{\dI{\lambda}}\circ\cRp$ restricts to a map $\KL \to \Cs{-\vartheta}{-k-1}$.
The result therefore holds for all $\lambda$ by continuity.

The proof that $ch_{\J_\vartheta}$ is an isomorphism is similar to that of Proposition \ref{prop:mJtheta-iso}.
Specifically, if $f\in\germ{\cJ_{\vartheta}}$ is a germ of a flat section such that $ch_{\J_\vartheta}(f)=0$ then, 
arguing as in \ref {prop:mJtheta-iso}, one shows that $ch_{\J_\vartheta}((\Htrigsub)\sff f)=0$, where the
superscript $\mathsf{f}$ denotes the fibrewise action of $\Htrigsub$. 

By \cite[Prop. 1.6]{cherednik1994integration}, the dual $\J_\theta^*$ is isomorphic to $\J_{\vartheta}^-$.
Since $\pRad(\cdot)(\vartheta)\in\J_\theta^*\setminus\{0\}$ is antiinvariant under $W$, it follows
that it is the generating vector of $\J_{\vartheta}^-$, and therefore does not vanish on any non--trivial
submodule of $\J_{\vartheta}$. The constraint $ch_{\J_\vartheta}((\Htrigsub)\sff f)=0$ then implies
that $(\Htrigsub)\sff f(h_0)=0$, and therefore that $f=0$.
\end{proof} 

\subsection{Shift operators in the parameter \texorpdfstring{$k$}{k}}  
\label{sec:shiftk}

We now generalize Felder and Veselov's construction of $k$--shift operators using Matsuo
and Cherednik's maps \cite{felder1994shift} to $\cJ_\vartheta(k)$. Composing one with the
inverse of the other gives a map
\[
S_{\J_\vartheta}(k) = ch_{\J_\vartheta}^{-1}(k\csign 1) \circ m_{\J_\vartheta}(k) \colon \cJ_\vartheta(k) \to \cJ_\vartheta(k\csign 1)
 \]
which restricts to a map of local systems $S_{\J_\vartheta}(k): \KL \to \KL[\vartheta,k\csign 1]$.

\Omit{
The shift operators $S_{\J_\vartheta}$ fit into the following commutative diagram,
\begin{equation}\label{eqn:FVdiagram}
	\begin{tikzcd} 
  	\cdots &	\KL[\vartheta,k-1] \arrow[r,"S_{\J_\vartheta}(k-1)"] \arrow[d,"m_{\J_\vartheta}(k-1)"] & \KL \arrow[r,"S_{\J_\vartheta}(k)"] \arrow[d,"m_{\J_\vartheta}(k)"] \arrow[ld,"ch_{\J_\vartheta}(k)"]  & \KL[\vartheta,k+1] \arrow[ld,"ch_{\J_\vartheta}(k+1)"]  & \cdots
	 \\
		\cdots &  \Cs{\vartheta}{-k+1}&  \Cs{\vartheta}{k}   &  & \cdots 
	\end{tikzcd}
\end{equation}
}

The shift operators $S_{\J_\vartheta}$ fit into the following commutative diagram,\omitvaleriocomment
{The shift operators would go in the opposite direction as what you had written before is the correct
sign to take in the initial Cherednik map is $-1$, but this is the dictated by the corrected
Matsuo and Cherednik's map, modulo checking that these map as claimed. Is this an issue?}
\begin{equation}\label{eqn:FVdiagram}
\begin{tikzcd} 
 \cdots&\KL[\vartheta,k\ncsign 1] \arrow[r,"S_{\J_\vartheta}(k\ncsign 1)"] \arrow[d,"m_{\J_\vartheta}(k\ncsign 1)"] & \KL \arrow[r,"S_{\J_\vartheta}(k)"] \arrow[d,"m_{\J_\vartheta}(k)"] \arrow[ld,"ch_{\J_\vartheta}(k)"]  & \KL[\vartheta,k\csign 1]f \arrow[ld,"ch_{\J_\vartheta}(k\csign 1)"]  & \cdots\\
\cdots&\Cs{-\vartheta}{-k\csign 1}&  \Cs{-\vartheta}{-k}&&\cdots 
\end{tikzcd}
\end{equation}
Propositions \ref{prop:mJtheta-iso} and \ref{prop:chiso}, and Sections \ref{ss:equi 1} and \ref{ss:equi 2}
then imply the
following.

\label{thm:k-shift-op-invert}
\begin{theorem}\hfill
\begin{enumerate}
\item If $\vartheta$ is $k$--regular, the map $S_{\J_\vartheta}(k) \colon \KL \to \KL[\vartheta,k\csign 1]$
is an isomorphism.
\item $S_{\J_\vartheta}(k)$ is equivariant \wrt the translation action of $\Omega^\vee=P^\vee/Q^\vee$, provided
one of the local systems is tensored with the character $\eps_\rho:\Omega^\vee\to\{\pm 1\}$ defined by \eqref
{eqn:FVdiagram}.
\end{enumerate}
\end{theorem}

\begin{remark}
Opdam defines a $k$--shift operator $\tilde{\mathcal{D}}_{\vartheta,k} \colon \Cs{\vartheta}{k} \to \Cs{\vartheta}{k+1}$ \cite[Def. 5.9]{opdam2001lectures}.  A multiple of Opdam's operator $\mathcal{D}_{\vartheta,k} = d(\vartheta,k) \tilde{\mathcal{D}}_{\vartheta,k}$ where $d(\vartheta,k) \in \IC(\vartheta,k)$ fits into the commutative diagram \eqref{eqn:FVdiagram}.
\end{remark}

\section{Shift operators in the parameter \texorpdfstring{$\lambda$}{lambda}}
\label{sec:shiftoperatorsinlambda}

In this section, we review the definition of $\lambda$--shift operators between induced
representations of $\Htrig$ for any element of the extended affine Weyl group $W^e$.
We also give necessary and sufficient conditions for their invertibility, thus extending
the results obtained in \ref{subsec:Intertwiners} for the degenerate affine Hecke algebra.

\subsection{Affine Intertwiners in $\Htrig$ \cite{cherednik1991unification}}
\label{subsec:affineHdegIntertwiners}

The following extends the definition of the intertwiners $\{\tPhi_i\}_{i=1}^n\subset
\Htrigsub$ given in \ref{subsec:HdegIntertwiners} to the affine node $i=0$. Let $
\cor_0=-\psi^\vee+1$ and $s_0=s_{-\psi^\vee + 1}$ be the affine simple coroot
and reflection \eqref{eq:s0}, set $k_0=k_{\psi}$, and define $\tPhi_0\in\Htrig$ by
\[\tPhi_0 = s_0 \cor_0 + k_0 = - \cor_0 s_0  - k_0\]

Then, \eqref{eqn:s0pcommrel} implies that $\tPhi_0$ satisfies the commutation
and squaring relations \eqref{eqn-p-phii}--\eqref{eqn-phii-squared} for $i=0$.
Moreover, given two reduced decompositions $s_{i_1} \ldots s_{i_r} = s_{j_1}
\ldots s_{j_r}$ of $w\in W^a$, the following holds \cite[Thm. 4.2]{opdam2001lectures}
\[\tPhi_{i_1} \ldots \tPhi_{i_r} = \tPhi_{j_1} \ldots \tPhi_{j_r}\]
and we denote by $\tPhi_{w}\in\Htrig$ the element represented by either side
of the equality. 

Finally, if $w=\omega s_{i_1}\cdots s_{i_k}$ is a reduced decomposition of $w
\in W^e$, with $\omega\in\Omega$ and $0\leq i_j\leq n$ for any $j=1,\ldots,k$,
set
\[\tPhi_w=\omega \tPhi_{i_1} \ldots \tPhi_{i_k}\in\Htrig\]
Then, $p\tPhi_w=\tPhi_w\ ^{w}p$ for any $p\in S\h$, and $\tPhi_v\tPhi_w=\tPhi
_{vw}$ whenever $\ell(vw)=\ell(v)+\ell(w)$.

\subsection{Intertwiners between induced representations}
\label{subsec:affine Intertwiners}

Let $\lambda\in\h^*$, and $\cI_\lambda$ be the induced $\Htrig$--module introduced
in \ref{ss:ind cov trig}. For any $w\in W^e$, define a morphism of $\Htrig$--modules
\begin{equation}\label{eq:affine btw I}
\cT{w}{\lambda}:\cI_{\lambda}\to\cI_{w\lambda}
\qquad\text{by}\qquad
\ilambda\to \tPhi_{w^{-1}}\inu{w\lambda}
\end{equation}
Clearly,
\begin{equation}\label{eq:affine v w}
\cT{v}{w\lambda}\circ\cT{w}{\lambda}=\cT{vw}{\lambda}
\end{equation}
whenever $\ell(vw)=\ell(v)+\ell(w)$.

Extend the weight function $k$ to a $W^e$--invariant function $R^a\to\IC$ on the set of affine roots.\valeriocomment
{does this require a check?} Then, the following holds

\label{pr:det cT}
\begin{proposition}\hfill
\begin{enumerate}
\item Identify $\cI_\lambda$ with the free $\IC[H]$--module generated by $\IC W$ by
$f w\inu{\lambda}\to f\otimes w$, $f\in \IC[H]$, $w\in W$. Then, for any $w\in W^a$
\[\det(\cT{w}{\lambda}) = 
\prod_{\alpha\in\apos\cap w^{-1}\aneg}
\left(k_\alpha^2-\cor(\lambda)^2\right)^{|W|/2}\]
\item $\cT{w}{\lambda}$ is an isomorphism if and only if $\lambda(\alpha^\vee)\neq\pm k_\alpha$
for any $\alpha\in\apos\cap w^{-1}\aneg$.
\end{enumerate}
\end{proposition}
\begin{proof}
(1) By \eqref{eq:affine v w} and Proposition \ref{pr:det TT}, it suffices to prove the result for $w=s_0$.
The proof is similar to that of Proposition \ref{pr:det TT}. Namely, for any $w\in W$,
\[\begin{split}
\cT{s_0}{\lambda} w\ilambda
&=
w\left(s_0\cor_0+k_0\right)\inu{s_0\lambda}
=
k_0 w\inu{s_0\lambda}-
\cor_0(\lambda)w s_0\inu{s_0\lambda}\\
&=
k_0 w\inu{s_0\lambda}-
\cor_0(\lambda)w e^{\psi} s_{\psi^\vee}\inu{s_0\lambda}
=
k_0 w\inu{s_0\lambda}-\cor_0(\lambda)e^{w\psi} w s_{\psi^\vee}\inu{s_0\lambda}
\end{split}\]
Thus, $\cT{s_0}{\lambda}$ preserves the $\IC[H]$--span of each of each right $\langle s_{\psi^\vee}
\rangle$--coset in $W$, and acts on it as the matrix 
\[\begin{pmatrix}
k_0					&	-\cor_0(\lambda)e^{-w\psi} \\
-\cor_0(\lambda)e^{w\psi}	&	k_0
\end{pmatrix}\]
from which the result follows.

(2) is a direct consequence of (1), the fact that $\Omega$ preserves $\apos$ and that
$\cT{\omega}{\lambda}$ is clearly invertible for any $\omega\in\Omega$ and $\lambda
\in\h^*$.
\end{proof}

\begin{remark} Similarly to Remark \ref{rk:inverse of TT}, if $w\in W^e$ the inverse
of $\cT{w}{\lambda}$ is readily seen to be
\[\cT_{w,\lambda}^{-1}=
\prod_{\alpha\in\apos\cap w^{-1}\aneg}
\left(k_\alpha^2-\cor(\lambda)^2\right)^{-1}
\cdot\cT_{w^{-1},w\lambda}
\]
\end{remark}

\subsection{Affine $k$--regularity}

\begin{definition}\label{def:affinekregular}
An element $\lambda \in \frakh^*$ is \textit{affine $k$--regular} if
\[
	\cor(\lambda) \neq \pm k_{\alpha} \text{ for all } \cor \in (R^{\vee})^a
\]
that is if $\cor(\lambda) \pm k_\alpha \notin \mathbb{Z}$ for all $\cor\in R^\vee$.  
\end{definition}

We denote the set of affine $k$--regular $\lambda$ by $\frakh^*\krega$. 

\label{cor:k-reg-lambda-shift}
\begin{corollary}
If $\lambda\in\frakh^*\krega$, $\cT{w}{\lambda}: \cI_\lambda \to \cI_{w\lambda}$
is an isomorphism for any $w\in W^e$.
\end{corollary}

\subsection{Intertwiners and the KZ connection}

Since the affine intertwiners $\{\cT{w}{\lambda}\}_{w\in W^e}$ are morphisms of $\Htrig
$--modules, the application of the trigonometric KZ functor of \ref{subsec:trigKZconn} to
induced representations implies that they define morphisms $\cI_{\lambda}\to\cI_{w\lambda}$
of $W$--equivariant vector bundles with integrable connections over $H_{\reg}$, and in
particular that
\[\nabla^{I}(w\lambda,k)\circ\cT{w}{\lambda}=\cT{w}{\lambda}\circ\nabla^{I}(\lambda,k)\]

As pointed out in \ref{ss:equi 1}, $(\cI_\lambda,\Iconn)$ is equivariant \wrt the translation
action of $\Omega^\vee=P^\vee/Q^\vee$. The morphisms $\cT{w}{\lambda}$ corresponding
to the affine Weyl group $W^a=W\ltimes Q$ intertwine the action of $\Omega^\vee$ since 
their definition involves the action of the Weyl group $W$, the group algebra $\IC Q=\IC
[H/\Omega^\vee]$ and the Dunkl operators, all of which commute with $\Omega^\vee$. 

This, however, is not the case for the action of the intertwiners $\{\cT{\omega}{\lambda}\}
_{\omega\in\Omega}$. Specifically, consider the bilinear map 
\[\langle-,-\rangle:P/Q\otimes_\IZ P^\vee/Q^\vee\to U(1)
\qquad
\nu\otimes\nu^\vee\to\exp\left(2\pi\ii(\nu,\nu^\vee)\right)\]
Then, the application of $\cT{\omega}{\lambda}$ twists the $\Omega^\vee$--action by the
character $\langle\omega,-\rangle$. It follows that, for any $w\in W^e$ and $\nu^\vee\in
\Omega^\vee$
\[T_{\nu^\vee}\circ\cT{w}{\lambda}=\langle\pi(w),\nu^\vee\rangle\cT{w}{\lambda}\circ T_{\nu^\vee}\]
where $\pi$ is the quotient map $W^e\to P/Q$.

\section{Representations of the extended affine Hecke algebra}\label{sec:AHA}

\subsection{The Extended Affine Hecke Algebra $\Hext$}

Let $P^\vee\subset\h$ be the coweight lattice, and $H^\vee=\Hom_{\IZ}(P^\vee,\IC^\times)$ the dual torus of
simply connected type. We denote the element of $\IC P^\vee=\IC[H^\vee]$ corresponding to $\lambda\in P^
\vee$ by $\Y^\lambda$ or $e^\lambda$, and for any $1\leq i\leq n$, set
\[\X_i=\Y^{\alpha^\vee_i}\aand\Y_i=\Y^{\lambda^\vee_i}\]

Let $q:R\to\IC^\times$ be a W--invariant function, set $q_\alpha=q(\alpha)$ and $q_i=q_{\alpha_i}$.
The extended affine Hecke algebra $\Hext$ is the $\IC$--algebra generated by $\{T_i\}_{1\leq i\leq n}$
and $\IC P^\vee$, with relations
\begin{gather}
\underbrace{T_i T_j T_i\cdots }_{m_{ij}} = \underbrace{T_j T_i T_j\cdots }_{m_{ij}}  \label{eqn:AHA-TT-rel} \\
(T_i-1)(T_i+q_i)  = 0 \label{eqn:AHA-T-rel}\\
T_i \Y^\lambda - \Y^{s_i\lambda} T_i =
(q_i-1) \frac{\Y^{s_i\lambda} - \Y^\lambda }{1-\X_i^{-1}} 
\label{eqn:AHA-TY-rel}
\end{gather}
where $m_{ij}$ is the order of $s_i s_j$ in $W$ \cite[6.2]{opdam2001lectures}.
 Note that the \rhs of \eqref{eqn:AHA-TY-rel} lies
in $\IC P^\vee$ since $\Y^{s_i\lambda}-\Y^\lambda=\Y^\lambda(\X^{-\lambda(\alpha_i)}_i-1)$.

Let $\Hq$ be the (finite) Hecke algebra generated by $\{T_i\}_{i=1}^n$ with relations \eqref{eqn:AHA-TT-rel}--\eqref
{eqn:AHA-T-rel}. Then, $\Hext$ is isomorphic to $\Hq\otimes\IC P^\vee$ as $(\Hq,\IC P^\vee)$--bimodule,
and its center is the algebra $(\IC P^\vee)^W=\IC[H^\vee]^W$ of $W$--invariants \cite{lusztig1989affine}.

\subsection{The Induced and Covariant Representations}

Let $\Lambda \in H^\vee$, $\IC_\Lambda = \IC \iLambda$ the corresponding evaluation
representation of $\IC[H^\vee]$, and set
\[\I{\Lambda} = \mathrm{Ind}_{\IC[H^\vee]}^{\Hext} \IC_\Lambda\]
As an $\Hq$--module, $\I{\Lambda}$ is isomorphic to the left regular representation.

Let $\eps \colon W \to \{ \pm 1\}$ be a character, set $\eps_\alpha=\eps(s_\alpha)$ and
$\eps_i = \eps(s_i)$. Let $\Theta\in H^\vee/W$, and $\IC_{\Theta}^\eps$ the character
of $\Hq\otimes\IC[H^\vee]^W$ on which $\IC[H^\vee]^W$ acts by evaluation at $\Theta$,
and $\Hq$ by $T_i\to \eps_i q_i^{(1-\eps_i)/2}$. Define the covariant representation $\J
(\Theta)^\eps$ by\omitvaleriocomment{Generalise this with an arbitrary
$W$--invariant $\veps:\{s_i\}\to\{\pm 1\}$ as in 3.5. Note, however, that this is a little less
convenient to define since $T_i$ satisfy $(T_i-1)(T_i+q_i)=0$ rather than $T_i^2-q_i^2=0$.
In the latter convention, the representations would have been defined by $T_i\to \eps_i q_i
^{\eps_i}$. I wonder whether this suggests that the latter convention for $T_i$ is a better
one? Perhaps the former is better adapted to the monodromy of the trigo KZ connection,
though, since the residues are $1-s_\alpha$, so have eigenvalues $0$ and $2$ rather than
the more symmetric $\pm 1$.}
\[
\J(\Theta)^\eps = \mathrm{Ind}_{\Hq\otimes\IC[H^\vee]^W}^{\Hext} \IC_{\Theta}^\eps
\]
As a $\IC[H^\vee]$--module, $\J(\Theta)^\eps$ is isomorphic to $\IC[H^\vee]/M_\Theta$,
where $M_\Theta$ is the ideal generated by $\{f-f(\Theta)\}_{f\in\IC[H^\vee]^W}$, and is
therefore of dimension $|W|$. When $\eps$ is the trivial or sign character, we denote
$\J(\Theta)^\eps$ by $\J(\Theta)$ or $\J(\Theta)^-$ respectively.  

\subsection{Isomorphism between $\J(\Theta)^\eps$ and $\I{\Lambda}$}

Fix a character $\eps:W\to\{\pm 1\}$.

\label{le:idempotent}
\begin{lemma}
There is an element $\Qeps\in\Hq$ such that $T_i\Qeps=\Qeps T_i=\eps_i q_i^{(1-\eps_i)/2}
\Qeps$. $\Qeps$ is unique up to a scalar, and given by the formula
\[\Qeps=\sum_{w\in W} a_w T_w
\qquad
\text{where}
\qquad 
a_w= \eps_{i_1} q_{i_1}^{-(\eps_{i_1}+1)/2}\cdots \eps_{i_k}  q_{i_k}^{-(\eps_{i_k}+1)/2}\]
and $T_w=T_{i_1}\cdots T_{i_k}\in\Hq$ is the element corresponding to any
reduced decomposition $w=s_{i_1}\cdots s_{i_k}$.
\end{lemma}
\begin{proof}
Fix a simple reflection $s_i$.  We prove $\Qeps T_i=\eps_i q_i^{(1-\eps_i)/2}T_i$; the argument
for $T_i \Qeps$ is similar. Set $W_i=\{ w \in W|\ell(w s_i) = \ell(w) + 1 \}$, so that $W=W_i\sqcup
W_is_i$. Then,
\begin{equation}\label{eqn:Eeps}
E_\eps
 = \sum_{w \in W_i} a_w T_w ( 1 + \eps_i q_i^{-(\eps_i+1)/2} T_i )
\end{equation}
The relation \eqref{eqn:AHA-T-rel} can be rewritten as 
\begin{equation}\label{eqn:AHA-T-rel-eps}
\left(1 + \eps_i q_i^{-(\eps_i+1)/2} T_i \right) T_i = \left(1 + \eps_i q_i^{-(\eps_i+1)/2} T_i \right) \eps_i q_i^{(1-\eps_i)/2}
\end{equation}
Expanding $\Qeps T_i$ using \eqref{eqn:Eeps} and applying \eqref{eqn:AHA-T-rel-eps} gives the desired result.
An upward induction on the length of $w$ shows that $\Qeps$ is unique up to a scalar.
\end{proof}

Let now $\Lambda \in H$, and denote its image in $H/W$ by $\Theta$. It follows from Lemma
\ref{le:idempotent} that there is a non--zero intertwiner of $\Hext$--modules 
\[\Rnu{\Lambda}^\veps:\J(\Theta)^\eps\longrightarrow \I{\Lambda},\qquad
\kTheta \longmapsto  \Qeps \iLambda\]
which is unique up to a scalar. In particular, $\J(\Theta)^\eps$ is irreducible if and only if $\I{\Lambda}$
is and, in turn, $\I{\Lambda}$ is irreducible if and only if all $I(w\Lambda)$ are.

The following result is due to Kato \cite{Kato-irred}, and is an analogue of Theorem
\ref{th:det R}.

\label{thm:iso-ILambda-JTheta}
\begin{theorem}
The map $\Rnu{\Lambda}^\veps$ is an isomorphism if and only if 
\begin{equation}\label{eq:Kato iso crit}
e^{\alpha}(\Lambda)\neq q_\alpha^
{\eps_\alpha}
\quad\text{for all $\alpha \in R^+$}
\end{equation}
In particular, $\I{\Lambda}$ and $\J(\Theta)^\eps$ are isomorphic if and only if \eqref{eq:Kato iso crit} holds.
\end{theorem}
\begin{proof}
This is proved in \cite[Thm 2.4]{Kato-irred} for $\J(\Theta)^-$.  The result follows for arbitrary $\eps$ using the isomorphism
$a_{\eps} \colon H_{q}^{\mathrm{ext}} \to H_{q^\eps}^{\mathrm{ext}}$ given by $T_i \mapsto \eps_i q_i
^{(1-\eps_i)/2} T_i$.  The pullback of $\J(\Theta)$ under $a_{\eps}$ is $\J(\Theta)^\eps$.
\end{proof}

\subsection{Irreducibility of $\I{\Lambda}$}

The following result is due to Kato \cite[Thm 2.2]{Kato-irred}, and is an analogue of Theorem \ref{th:irr I}.
\omitvaleriocomment
{Kato's criterion is more complicated for a general $L^\vee$ in the following ways: 1) Let $W_\Lambda\subset$ be
the isotropy group of $\Lambda$, and $W_{(\Lambda)}\subset W_\Lambda$ the (normal) subgroup generated
by the reflections $s_\alpha$, where $\alpha$ is such that $e^{\alpha}(\Lambda)=1$. Then a further condition
for irreducibility is that $W_{(\Lambda)}=W_\Lambda$. This is true if $L^\vee=P^\vee$ (see below), but can
fail otherwise. E.g. if $T$ is the adjoint rank 1 torus, so that $e^\alpha$ gives an identification $T\to\IC^
\times$, and $\Lambda=-1$, then $W_{\Lambda}=W$, but $W_{(\Lambda)}=1$. If $L^\vee=P^\vee$, so that $H^\vee
_{L^\vee}=H^\vee\cong\h^*/Q$, $h\in H^\vee$ and $\wt{h}\in\h^*$ is a preimage, it is easy to see that the quotient map
$W^a=W\rtimes Q\to W$ maps $W^a_{\wt{h}}$ onto $W_h$. Since $W^a$ is a Coxeter group
(this is where $L^\vee=P^\vee$ is used), $W^a_{\wt{h}}$ is generated by the reflections across the affine hyperplanes
$Ker(\alpha^\vee-\alpha^\vee(\wt{h}))$, where $\alpha^\vee\in R^\vee$ is such that $\alpha^\vee(\wt{h})\in\IZ$, and it is easy to identify
those with the $\alpha\in R^\vee$ such that $e^\alpha(h)=1$. 2) In addition to this, Kato's conditions take the above nice form provided $\alpha(L^\vee)\neq 2\IZ$ for
all roots $\alpha$. This is clearly true for $L^\vee=P^\vee$ but can fail otherwise, e.g. in rank 1 for $L^\vee=Q^\vee$.}

\label{th:Kato}
\begin{theorem}\hfill
\begin{enumerate}
\item $\I{\Lambda}$ is irreducible if and only if $e^\alpha(\Lambda)\neq q_\alpha$ for
any $\alpha\in R$.
\item $\J(\Theta)^\eps$ is irreducible if and only if $e^\alpha(\Lambda)\neq q_\alpha$ for
any $\alpha\in R$.
\end{enumerate}
\end{theorem}

\section{Monodromy of \texorpdfstring{$\cJ_\vartheta$}{J theta} for generic parameters}\label{sec:monodromy-j-theta}

In this section, we prove that Conjecture \ref{conj:monoconj} holds for generic values of $(\vartheta,k)$.

\subsection{Monodromy Representations of the Affine Hecke Algebra}
\label{subsec:affineheckealg}

Let $V$ be an $\Htrigsub$--module, and $\VV$ the corresponding induced representation
of $\Htrig$. By \eqref{eqn:KZconnection}, the KZ connection on $\VV=\IC[H]\otimes V$ has
the form
\[
\nabla=d-
\left( \sum_{\alpha \in \pos} k_\alpha  \frac{d\alpha}{1-e^{-\alpha}} \left(1- \tfact{s}_\alpha\right)
-\tfact{\iota}-\rho_k(\iota) \right)\] 
where $\tfact{X}=1\otimes X$ for $X\in\Htrigsub$, and $\iota$ is the $\h$--valued translation--invariant
1--form on $H$ which identifies $T_h H$ and $\h$.

The connection $\nabla$ is $W$--equivariant, and descends to the torus $T=\Hom_{\IZ}(Q,\IC^
\times)$. Its monodromy therefore yields a representation $\mu$ of the orbifold fundamental group $\pi
_1\orb(T_{\reg}/W)$, which is the extended affine braid group $\mathrm{Br}^\mathrm{e}$. Moreover,
$\mu$ factors through the extended affine Hecke algebra $\Hext$ \cite[Corollary 4.3.8]{heckman1995harmonic}. \omitrobincomment{I think we are inconsistent on when we use $\mathcal{H}$ and $H$ for ext. aff. Hecke algebra. See section 6} 
\omitvaleriocomment{Name the generators as $\{T_i,\Y^{\cow{i}}\}_{i=1}^n$, possibly denoting $\Y^{\cow{i}}$
by $\Y_i$ since we do not seem to have much use for $\Y^{\alpha_i}$, and say what they look like topologically,
as paths in $\h$ (Cherednik does this).}

\subsection{Canonical coordinates on $T$}

Let $\{Z_i\}_{i=1}^n$ be the coordinates on $T$ given by $Z_i = e^{\alpha_i}$.  These give rise
to a toric compactification 
\[T \hookrightarrow \ol{T}= \mathrm{Spec}(\bbC[Z_1,\ldots,Z_n])\cong\IC^n\]

For any positive root $\alpha$, write $\alpha=\sum_{i} m^i_{\alpha} \alpha_i$, where $m^i_\alpha
=\alpha(\cow{i})\in\IZ_{\geq 0}$. Then,
\[ d\alpha = \sum_i m^i_\alpha d\alpha_i=\sum_i m^i_\alpha d\log Z_i
\aand
\iota
=
\sum_i \cow{i}d\alpha_i=\sum_i \cow{i}d\log{Z_i}\]
It follows that in the coordinates $Z_i$, the connection $\nabla$ has the form
\begin{equation}\label{eq:nabla on T}
\nabla= 
d	
-
\sum_{\alpha \in \pos}\sum_{i} k_\alpha\left(\frac{Z^{m_\alpha}}{1-Z^{m_\alpha}}   \frac{m_\alpha^i
d Z_i}{Z_i} \right)\left( 1-\tfact{s}_\alpha \right)+\sum_i
\left(\tfact{(\cow{i})}+\rho_k(\cow{i})\right)\frac{d Z_i}{Z_i}  
\end{equation}
where $Z^{m_\alpha} = \prod_j Z_j^{m_\alpha^j}$.

Note that the first term in \eqref{eq:nabla on T} is regular on each of the divisors $\{Z_i = 0\}$
since only the terms in which $m_\alpha^i > 0$ contribute, and $Z^{m_\alpha}/Z_i$ is regular
for these.

\Omit{
\subsection{Canonical coordinates on $T$}

Let $\{Z_i\}_{i=1}^n$ be the coordinates on $T$ given by $Z_i = e^{-\alpha_i}$.  These give a toric
compactification $\ol{T}\cong\IC^n$ of $T$ via
\[T \hookrightarrow \ol{T}= \mathrm{Spec}(\bbC[Z_1,\ldots,Z_n])\]

For any positive root $\alpha$, write $\alpha=\sum_{i} m^i_{\alpha} \alpha_i$,
where $m^i_\alpha=\alpha(\cow{i})\in\IZ_{\geq 0}$. Then,
\[ d\alpha = \sum_i m^i_\alpha d\alpha_i=-\sum_i m^i_\alpha d\log Z_i
\aand
\iota
=
\sum_i \cow{i}d\alpha_i=-\sum_i \cow{i}d\log{Z_i}\]
It follows that in the coordinates $Z_i$, the connection $\nabla$ has the form
\begin{equation}\label{eq:nabla on T}
\nabla= 
d	
+
\sum_{\alpha \in \pos}\sum_i
\frac{k_\alpha m_\alpha^i}{1-Z^{m_\alpha}}  
\frac{d Z_i}{Z_i}\left( 1-\tfact{s}_\alpha \right)
-\sum_i
\left(\tfact{(\cow{i})}+\rho_k(\cow{i})\right)\frac{d Z_i}{Z_i}  
\end{equation}
where $Z^{m_\alpha} = \prod_j Z_j^{m_\alpha^j}$. Note that, if $m_\alpha^i>0$,
\[\frac{1}{(1-Z^{m_\alpha})Z_i}=
\frac{1}{Z_i}+R_i\]
where $R_i$ is homolorphic in the 

Note that the first term in \eqref{eq:nabla on T} is regular on each of the divisors $\{Z_i = 0\}$ since
only the terms in which $m_\alpha^i > 0$ appear and thus $Z^{m_\alpha}/Z_i$ is regular.
}

\subsection{Topological description of the generators of $\Hext$}\label{sec:genofHext}
Fix $x_0 \in T_{\reg}$ such that $Z_i(x_0) = A \in \mathbb{R}_{>0}$ for $A$ close to 0 and $1 \leq i \leq n$.  The monodromy $\mu(T_i)$ is given by analytic  continution with respect to $\nabla$ along the path $\gamma_{T_i}$ which connects $x_0$ to $s_i(x_0)$ by a straight line except near $Z_i(t) = 1$ where $\gamma_{T_i}$ passes around $Z_i(t) = 1$ in a small circular arc with positive orientation. 
For $\lambda \in P^\vee$, define the closed loop $\gamma_{\Y^\lambda}(a) = x_0 \exp(2 \pi \ii a \lambda)$ for $0 \leq a \leq 1$.  Define $\tilde{\mu}(\Y^\lambda)$ by analytic continuation along $\gamma_{\Y^\lambda}$.  Following the convention from \cite[Theorem 6.8]{opdam2001lectures}, define $\mu(\Y^\lambda) = \tilde{\mu}(\Y^\lambda) e^{2 \pi \ii \rho_k(\lambda)}$.           


\subsection{Non--resonance}

Let $t_0 \in \overline{T}$ the point with coordinates $Z_i(t_0)=0$. We say that $\nabla$ is {\it
non--resonant} at $t_0$, or non--resonant for short if, for any $1\leq i\leq n$, the eigenvalues
of $\cow{i}$ on $V$ do not differ by non--zero integers. The following result spells this condition
out explicitly in type $\sfA$. 

\label{pr:non resonance}
\begin{proposition}
Assume that $R$ is of type $\sfA_{n-1}$. 
Let $\lambda\in\h^*$, $V=\dI{\lambda}$ or $\J_{q(\lambda)}$, and $\VV$ the corresponding induced
representation of $\Htrig$. The trigonometric KZ connection on $\VV$ is non--resonant if and only if 
for any $1\leq i\leq\lfloor \frac{n}{2}\rfloor$ and $i$--tuple $\alpha^\vee_{j_1},\ldots,\alpha^\vee_{j_i}$
of pairwise orthogonal coroots, the following holds
\[\lambda(\alpha^\vee_{j_1}+\cdots+\alpha^\vee_{j_i})\notin\IZ_{\neq 0}\]
\end{proposition}
\begin{proof}
By Propositions \ref{prop:weights-Ilambda} and \ref{co:weights of K}, the eigenvalues for the action of $\cow{i}$ on $V$ are $\{\sigma \cow{i} \}_{\sigma \in \SS_n}$. Thus $\nabla$ is non--resonant if
and only if $\lambda(\sigma\cow{i}-\sigma'\cow{i})\notin\IZ_{\neq 0}$ for any $1\leq i\leq n-1$ and
$\sigma,\sigma'\in W=\SS_n$. Let $\{e_i\}_{i=1}^n$ be the standard basis of $\IC^n$. Then, $\cow
{i}=\sum_{a=1}^i e_i-1/n\sum_{a=1}^n e_i$ so that
\[\sigma\cow{i}-\sigma'\cow{i}=
\sum_{a=1}^i e_{\sigma(i)} - \sum_{a=1}^i e_{\sigma'(i)}=
\sum_{a\in I_+}e_a - \sum_{a\in I_-}e_a=
\sum_{a\in I_+}(e_a - e_{\tau(a)})
\]
where $I_+=\sigma\{1,\ldots,i\}\setminus\sigma'\{1,\ldots,i\}$, $I_-=\sigma'\{1,\ldots,i\}\setminus
\sigma\{1,\ldots,i\}$ and $\tau$ is any chosen bijection $I_+\mapsto I_-$. The claim follows.
\end{proof}

\begin{remark} Write $\lambda=\sum_i\lambda_i\theta_i$, where $\{\theta_i\}$ is the dual basis to $\{e_i\}$.
Then, the above criterion is equivalent to requiring that for any $1\leq i\leq\lfloor \frac{n}{2}\rfloor$ and
disjoint subsets $I_\pm\subset\{1,\ldots,n\}$ of size $i$, the following holds
\[\sum_{j\in I_+}\lambda_j - \sum_{j\in I_-}\lambda_j \notin\IZ_{\neq 0}\]
\end{remark}

\Omit{
The following will be useful in Section \ref{sec:monodromy-higher-rank}.\valeriocomment{and should probably be moved there}

\begin{lemma}
Assume that $R$ is of type $\sfA_{n-1}$. For any $\lambda\in\h^*$, set
\[ B_\lambda = 
\{ \beta=\alpha^\vee_{i_1}+\cdots+\alpha^\vee_{i_j}|\,\alpha^\vee_{i_k}\in R^\vee,\,\alpha^\vee_{i_k}\perp\alpha^\vee_{i_\ell} \text{ for any } \leq k\neq \ell\leq j, \text{ and }
\lambda(\alpha^\vee_{i_1}+\cdots+\alpha^\vee_{i_j})\notin\IZ_{\neq 0}\}\]
Then, there is an element $\mu\in P$ such $(\lambda-\mu)(\beta)=0$ for any $\beta\in B_\lambda$.
\end{lemma}
\begin{proof}
Let $Q^\vee_\lambda\subseteq Q^\vee$ be the sublattice generated by $B_\lambda$. We claim
that $Q^\vee_\lambda$ is a direct summand in $Q^\vee$, from which the claim readily follows.
\end{proof}
}

\subsection{Canonical fundamental solution}

The following is well--known, see {\it e.g.} \cite[4.2]{heckman1995harmonic}

\label{thm:holomorphicpartofLVL}
\begin{theorem}
Assume that $\nabla$ is {\it non--resonant} at $t_0$. Then, for any connected and simply--connected open
set $t_0\in U\subset \ol{T}_{\reg}$, there is a unique holomorphic function $H_0\colon U \to GL(V)$ which
is uniquely determined by the requirement that $H(t_0) = 1$ and that the gauge transform of $\nabla$ by $H$
is equal to 
\[d+\sum_i\left(\tfact{(\cow{i})}+\rho_k(\cow{i})\right)\frac{d Z_i}{Z_i}\]
Thus, for any determination of the functions $\log(Z_i)$ on $U$,
\[\Phi_0 = H_0(Z)\cdot\prod_i Z_i^{-\tfact{(\cow{i})}-\rho_k(\cow{i})}\]
is a fundamental solution of $\nabla$.\omitvaleriocomment{Give a topological description of the generators
of the affine Hecke algebra, and Cherednik's correction which gives the monodromy of $\Y_i$ in $\Phi_
0$ as $\exp(2\pi\ii \tfact{(\cow{i})})$ rather than $\exp(2\pi\ii \tfact{(\cow{i})}-\rho_k(\cow{i}))$. This 
should be in Cherednik's book, and is mentioned in Opdam's notes on page 35.}
\end{theorem}

\subsection{Monodromy representation for generic parameters}

Set
\[q_\alpha = e^{2 \pi \ii k_\alpha}\aand\Theta = \exp_{H^\vee}{\vartheta}\]
Then, the following holds \cite[Prop. 3.4]{cherednik1994integration}\footnote
{Note that \cite{cherednik1994integration} makes the additional assumption
that the stabilizer of $\Lambda$ is generated by simple reflections. The proof
of Proposition \ref{pr:generic monodromy} shows that this assumption is not
necessary.}
\label{pr:generic monodromy}
\begin{proposition}
Assume that
\begin{equation}\label{eq:generic mon}
e^\alpha(\Lambda) \neq q_\alpha\quad\text{for any $\alpha \in R$}
\end{equation}
Then, 
\begin{enumerate}
\item The monodromy of $\Iconn$ is isomorphic to the induced representation
$\I{\Lambda}$.
\item The monodromy of $\Kconn$ is isomorphic to the covariant representation
$K(\Theta)$. In particular, Conjecture \ref{conj:monoconj} holds in this case.
\end{enumerate}
\begin{proof} 
Since \eqref{eq:generic mon} implies that $\lambda(\alpha^\vee)\neq k_\alpha$ for any
$\alpha\in R$, it follows from Theorems \ref{th:det R} and \ref{thm:iso-ILambda-JTheta}
that (1) and (2) are equivalent.

To prove (1), note that the eigenvalues of any $y\in\h$ on $\dI{\lambda}$ are $\{w\lambda
(y)\}_{w\in W}$ by Proposition \ref{prop:weights-Ilambda}. It follows from Theorem \ref
{thm:holomorphicpartofLVL} that those of any $\Y\in P^\vee$ are $\{w\Lambda(\Y)\}_{w
\in W}$ whenever $\Iconn$ is non--resonant\valeriocomment{There is a negation
in the exponents in Theorem \ref{thm:holomorphicpartofLVL} which would give the
eigenvalues as $\{w\Lambda^{-1}(\Y)\}_{w\in W}$. I suspect this comes from the change of
coordinates $Z_i=e^{\alpha_i}$ rather than $Z_i=e^{-\alpha_i}$. Check this.}, and therefore for all $(\lambda,
k)$ by continuity. Thus, there is a non--zero morphism $\Psi_\Lambda:\I{\Lambda}\to
V$ of $\Hext$--modules for any $(\lambda,k)$. Since $\I{\Lambda}$ is irreducible by
Theorem \ref{th:Kato}, $\Psi_\Lambda$ is injective and therefore an isomorphism.
\end{proof}
\valeriocomment{Since \eqref{eq:generic mon} cuts across the non--resonance locus,
it shows that the FG conjecture holds on both (some) resonant and (some) non--resonant
points. Make sure this is pointed out somewhere.}
\omitvaleriocomment{Cherednik in \cite[Prop. 3.4 (a)--(b)]{cherednik1994integration} additionally assumes
that the stabilizer of $\Lambda$ is generated by simple reflections. We do not seem to need this
{\bf but} we are dependent on Kato's irreducibility criterion for $\I{\Lambda}$ (as is Cherednik) so
make sure Kato does not rely on that additional assumption.}
\end{proposition}

\section{Monodromy in rank one}\label{sec:monodromy-rank-one}

\newcommand {\AD}{\mathsf{AD}}
In this section, we compute the monodromy of the trigonometric KZ connection with values in
the covariant representation $\J_\vartheta$ for all values of the parameters $(\vartheta,k)$, 
when $R=\{\pm\alpha\}$ is of rank 1. We make use of the shift operators in $\lambda$ and $k$
to resolve the monodromy at resonant values of the parameters.\valeriocomment{Here is a slightly
different, and possibly shorter, way to determine monodromy in rank 1. 1) By Proposition \ref
{pr:generic monodromy}, the monodromy of $\nabla^\vartheta$ is $K(\Theta)$ whenever
$\Lambda\neq q^{\pm 1}$, so may assume that $(\lambda,k)\in\AD=\{(\lambda,k)|\,\lambda
-k\in\IZ\text{ or }\lambda+k\in\IZ\}$ (this may, or may not cut down the work a little). 2) The
translation action of $\IZ^2$ on $\IC^2$ restricts to one on $\AD$. Determine a fundamental
domain for that \ie parametrise the orbits (a possible choice could be $\{(t,t)|\,t\in[0,1)\}\sqcup
\{(t,-t)|\,t\in(0,1/2)\sqcup(1/2,1)\}$) 3) On each orbit one can introduce the finer equivalence
relation $p\sim p'$ if $p,p'$ differ by an invertible shift. 4) It is clear that equivalence classes
of points map to isomorphic monodromy representations, so it seems useful to determine
the equivalence classes within each orbit. 5) Then, for each such class, pick a good (non--resonant
if possible?) point to determine the monodromy there. 6) Note that, if $p,p'$ differ by a non--invertible
shift, the latter should still give a morphism of monodromy representations at $p,p'$ which is
valuable information we have not exploited yet.}

\subsection{Monodromy of KZ connection in rank one} 
\label{subsec:monoKZrankone}

Fix $(\vartheta,k)\in\h^*/W\times\IC$, $\lambda\in q^{-1}(\vartheta)\subset\h^*$ with $\Re(\lambda(\alpha^\vee)) \geq 0$, and identify $\frakh^*$
with $\bbC$ by $\nu\to\nu(\alpha^\vee)$.  Let $\Lambda = e^{\pi i \lambda}$.

Denote the character of $H^{\mathrm{ext}}_q$ in which $\Y$ acts by $a$ and $T$ by $b$ by $(a,b)$.
The following determines the values of $(\vartheta,k)$ for which Conjecture \ref{conj:monoconj} holds. 

\label{th:main 1}
\begin{theorem}
The monodromy of the local system $\KL$ is isomorphic to $\J(\Theta)$ except in the following cases
\begin{enumerate}
\item If $\lambda,k\in\IZ$ and $|\lambda| \geq\max(k,1-k)$, $\KL$ is isomorphic to $(-\Lambda,1) \oplus (-\Lambda,-1)$.
\item If $\lambda,k\in\half{1}+\IZ$ and $|\lambda| \geq\max(k,1-k)$, $\KL$ is isomorphic to the induced
representation $\I{-\Lambda}$.
\item If one of $k \pm \lambda$ lies in $\IZ_{>0}$, $\lambda\notin\half{1}+\IZ$, and (1) does not apply, 
$\KL$ is isomorphic to $K(\Theta)^-$.
\end{enumerate}
\end{theorem}

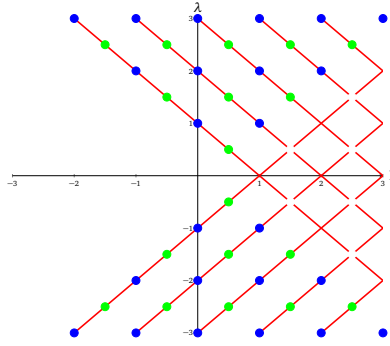
\begin{figure}[ht]
  \centering
  
  \begin{tikzpicture}[scale=0.3]
    \begin{axis}[
        color =black,
        xlabel=k,
        ylabel=$\lambda$,
        axis lines=middle,
        xmin=-3, xmax=3,
        ymin=-3, ymax=3,
        xtick={-3,-2,-1,0,1,2,3}, ytick={-3,-2,-1,0,1,2,3},
        width=18cm,
        line width=1pt,
        x label style={at={(axis description cs:1.01,0.5)},anchor=west,font=\huge},
        y label style={at={(axis description cs:0.5,1.01)},anchor=south,font=\huge},
    ]
    \addplot [color=red,line width=2pt,domain=-3:3, samples=2] {x-1};
    \addplot [color=red,line width=2pt,domain=-3:3, samples=2] {x-2};
    \addplot [color=red,line width=2pt,domain=-3:3, samples=2] {x-3};  
    \addplot [color=red,line width=2pt,domain=-3:3, samples=2] {x-4};  
    \addplot [color=red,line width=2pt,domain=-3:3, samples=2] {x-5};  
    \addplot [color=red,line width=2pt,domain=-3:3, samples=2] {x-6};  
    \addplot [color=red,line width=2pt,domain=-3:3, samples=2] {-x+1};
    \addplot [color=red,line width=2pt,domain=-3:3, samples=2] {-x+2};
    \addplot [color=red,line width=2pt,domain=-3:3, samples=2] {-x+3};
    \addplot [color=red,line width=2pt,domain=-3:3, samples=2] {-x+4}; 
    \addplot [color=red,line width=2pt,domain=-3:3, samples=2] {-x+5};
    \addplot [color=red,line width=2pt,domain=-3:3, samples=2] {-x+6};  
    
    \addplot [color=green,only marks,mark size=5pt] table[row sep=\\]{%
        0.5 0.5\\
        -0.5 1.5\\
        0.5 1.5\\
        1.5 1.5\\
        -1.5 2.5\\
        -0.5 2.5\\
        0.5 2.5\\
        1.5 2.5\\
        2.5 2.5\\
    };
    \addplot [color=green,only marks,mark size=5pt] table[row sep=\\]{%
        0.5 -0.5\\
        -0.5 -1.5\\
        0.5 -1.5\\
        1.5 -1.5\\
        -1.5 -2.5\\
        -0.5 -2.5\\
        0.5 -2.5\\
        1.5 -2.5\\
        2.5 -2.5\\
    };
    
    \addplot [color=white,only marks,mark size=5pt] table[row sep=\\]{%
        1.5 0.5\\
        2.5 0.5\\
        2.5 1.5\\
        1.5 -0.5\\
        2.5 -0.5\\
        2.5 -1.5\\
    };
    \addplot [color=blue,only marks,mark size=5pt] table[row sep=\\]{%
        0 -1\\
        1 -1\\
        -1 -2\\
        0 -2\\
        1 -2\\
        2 -2\\
        -2 -3\\
        -1 -3\\
        0 -3\\
        1 -3\\
        2 -3\\
        3 -3\\
        0 1\\
        1 1\\
        -1 2\\
        0 2\\
        1 2\\
        2 2\\
        -2 3\\
        -1 3\\
        0 3\\
        1 3\\
        2 3\\
        3 3\\
    };
    \end{axis}
  \end{tikzpicture}
  \caption{The cases from Theorem \ref{th:main 1} corresponding to different monodromy representations:
  (0) White $K(\Theta)$, (1) Blue $(-\Lambda,1) \oplus (-\Lambda,-1)$, (2) Green $\I{-\Lambda}$, (3) Red $K(\Theta)^{-}$} 
  
  \label{fig_rank_one} 
\end{figure}

\noindent
{\bf Remark.} The proof of Theorem \ref{th:main 1} shows that $(-\Lambda,1) \oplus (-\Lambda,-1)$, $\I{-\Lambda}$, and $K(\Theta)^-$
are not isomorphic to $K(\Theta)$ at the values of $(\vartheta,k)$ specified.  Note also that Theorem \ref{th:main 1}
implies that resonance and the validity of Conjecture \ref{conj:monoconj} are independent: there are both resonant
and non--resonant values of $(\vartheta, k)$ where the monodromy is either isomorphic to $K(\Theta)$ or not
isomorphic to $K(\Theta)$. 

\subsection{} 

Theorem \ref{th:main 1} is proved in the rest of this section. In \ref{ss:red to HG},
we use the Cherednik isomorphism from Section \ref{subsec:chJtheta} to replace $\KL[\vartheta,k]$ by the $\Omega^\vee=\IZ/2\IZ$--equivariant local system on $\IC\setminus\{0,1\}$
defined by the hypergeometric equation (HGE). In \ref{ss:orbifold}, we identify explicit generators
of the orbifold fundamental group of $\IC\setminus\{0,1\}/(\IZ/2\IZ)$. In \ref{ss:generic case} we
treat the case when the HGE is not resonant at $\infty$, which corresponds to $\lambda
\notin\frac{1}{2}\IZ$. We then treat the special case $\lambda=0,\frac{1}{2}$ in \ref{ss:la 0} and
\ref{subsec:proofrank1lambdahalf} respectively. Finally, the cases $\lambda\in\IZ_{\neq 0}$
and $\lambda\in\frac{1}{2}+\IZ_{\neq 0}$ are reduced to these by using the 
$\lambda$--shift operators in \ref{subsec:proofrank1lambdaint} and \ref{ss:half nonzero}.

The proof relies on the invertiblity conditions of the $\lambda$ and $k$--shift operators. By \ref
{thm:k-shift-op-invert}, the $k$--shift operator $\KL \to \KL[\vartheta,k+1]$ is an isomorphism when
$\lambda \neq \pm k$. The $\lambda$--shift operators $\cT{t_\rho}{\lambda} \colon \dI{\lambda}
\to I_{\lambda+1}$ were defined in \ref{subsec:affine Intertwiners} on induced representations,
but can also be defined over $K_\vartheta$, which is more relevant here.
 \begin{lemma}
 If $\lambda \neq -k, k-1$, there is an isomorphism of $\Htrig$--modules 
 \[ \cT{t_\rho}{\lambda}^{\J} \colon \J_{q(\lambda)} \to \J_{q(\lambda+1)} \qquad \cT{t_\rho}{\lambda}^{\J}
 =
 \mathrm{R}_{q(\lambda+1),\lambda+1} \circ \cT{t_\rho}{\lambda} \circ \mathrm{R}_{q(\lambda),\lambda} \]
 \end{lemma}
 \begin{proof}
 By Theorem \ref{th:det R}, the map $\Rp: \J_{q(\lambda)}  \longrightarrow \dI{\lambda}$ given by \eqref{eqn:defphi} is
an isomorphism when $\lambda\neq k$, and $\mathrm{R}_{q(\lambda+1),\lambda+1}$ is an isomorphism when $\lambda \neq k+1$.  By Proposition \ref{pr:det cT}, $\cT{t_\rho}{\lambda}$ is an isomorphism when $\lambda \neq \pm k$. 
With respect to the basis $\lbrace 1, \alpha^\vee\rbrace$ and $x=e^\rho$, the map $\cT{t_\rho}{\lambda}^{\J} = \mathrm{R}_{q(\lambda+1),\lambda+1} \circ \cT{t_\rho}{\lambda} \circ \mathrm{R}_{q(\lambda),\lambda}$ is \valeriocomment{This is likely no longer true with
the new normalisation of the shifts.}
\[
\frac{1}{2x}
\left(
\begin{array}{cc}
 x^2 (-2 k+\lambda +1)+\lambda +1 & x^2 (2 k-\lambda -1) (2 k+\lambda )+\lambda  (\lambda +1) \\
 1-x^2 & x^2 (2 k+\lambda )+\lambda  \\
\end{array}
\right)
\]
Due to cancellation, the resulting map is an isomorphism $\J_{q(\lambda)} \to \J_{q(\lambda+1)}$ whenever $\lambda
\neq -k, k-1$.
\end{proof}

\begin{remark}
The apparent lack of symmetry about the line $k=1/2$ in condition (3) from Theorem \ref{th:main 1} represented by the  red lines in Figure \ref{fig_rank_one} is result of our focus on $K(\Theta)$. The set of points which is neither isomorphic to $K(\Theta)$ nor $K(\Theta)^-$ is symmetric about $k=1/2$, and there are lines $k \pm \lambda \in \IZ$ corresponding to parameter values where monodromy is not isomorphic to $K(\Theta)^-$.
\end{remark}

\subsection{Reduction to the hypergeometric system}\label{ss:red to HG}

By Proposition \ref{prop:chiso}, the map $ch_{\J_\vartheta}$ is an $\Omega^\vee$--equivariant
isomorphism of the local system $\KL$ on $H_{\reg}/W$ defined by the trigonometric
KZ connection with values in $\J_\vartheta$, to the hypergeometric system $\Cs{\vartheta}{1-k}
\otimes\veps_\rho$ where $\veps_\rho$ is the representation $H^{\mathrm{ext}}_q \to \mathbb{C}^\times$ defined by $\veps_\rho(T)=1$ and $\veps_\rho(\nu)=\exp(2\pi\ii\rho(\nu))$. 

Let $H_{\reg}\to\IC^\times \setminus \{\pm 1\}$ be the identification given by the coordinate $x=e^\rho$.
Then, $w=1/2 -(x + x^{-1})/4$ identifies $H_{\reg}/W$ and $\IC\setminus\{0,1\}$, and the action
of $\Omega^\vee\cong\IZ/2\IZ$ with $w\to1-w$.
In the coordinate $w$, the local system $\Cs
{\vartheta}{1-k}$ is defined by the hypergeometric equation 
\begin{equation}\label{eq;hypergeomW}
w(1-w) \partial_w^2 F + (c - (1+a+b) w) \partial_w F - a b F=0
\end{equation} 
where \cite[Ex. 2.6]{heckman1997dunkl}, 
\[a = \lambda -k + 1,\quad b=-\lambda -k+1\aand c = 3/2-k\]

\subsection{The orbifold fundamental group of $H_{\reg}/W\times\Omega^\vee$}\label{ss:orbifold}

The local system $\Cs{\vartheta}{1-k}$ gives rise to a monodromy representation
of the orbifold fundamental group
\[\pi_1\orb(H_{\reg}/W\times\Omega^\vee)\cong\pi_1\orb(\IC\setminus\{0,1\}/w\equiv 1-w)\]

Fix a basepoint $x_0 = A i \in H_{\reg}$ with $A \in \mathbb{R}_{>0}$ very close to 0. Then,
$\pi_1\orb(H_{\reg}/W\times\Omega^\vee)$ is generated by two paths $T,Y$ in $H_{\reg}$
whose endpoints differ by elements of $W\times\Omega^\vee$. The generator $\Y$
corresponds to the path from $x_0$ to $-x_0 = - Ai$ counterclockwise around $x=0$,
and $T$ to the path from $x_0$ to $1/x_0 = - (1/A)i$ clockwise around $x=0$.

In terms of the coordinate $w=1/2-(x+x^{-1})/4$, the basepoint $x_0$ is mapped to $w_0
= 1/2 -i(A-1/A)/4$.  As $A \to 0$, $w_0 \to \infty$ parallel to the positive imaginary axis along
$\Re(w)=1/2$.  The path $\Y$ becomes a clockwise path $w_0$ to $1-w_0 = 1/2+i(A-1/A)/4$. 
As $A \to \infty$ this becomes a positively-oriented half-loop at $w = \infty$.  The path $T$
becomes a full clockwise loop based at $w_0$ around $w=0$. Both paths are pictured  in 
Figure \ref{fi:TY}.

\begin{figure}
\centering
\includegraphics[width=50mm,trim=130 140 50 0,clip]{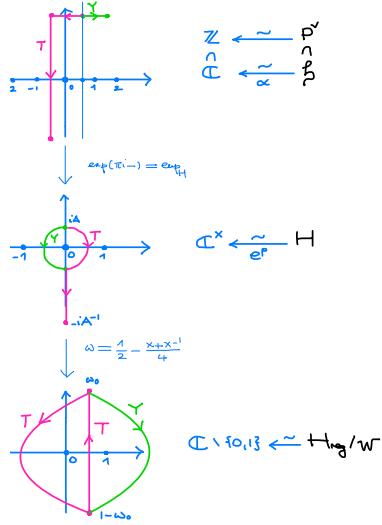}
\caption{The generators $T,Y$ of the orbifold fundamental group.}
\label{fi:TY}
\end{figure}

\subsection{Case 1: $\lambda\notin\half{1}\IZ$}\label{ss:generic case}

By assumption, $a-b = 2 \lambda \notin \mathbb{Z}$.  Near $w=\infty$, the hypergeometric equation
then has the two independent solutions \cite[15.10]{NIST:DLMF}
\begin{align*}
	&\Psi_1^\infty = w^{-a} F(a,a-c+1,a-b+1; 1/w) \\
	&\Psi_2^\infty = w^{-b} F(b,b-c+1,b-a+1;1/w)
\end{align*} 
where 
\[F(\alpha,\beta,\gamma;z)=
\sum_{n\geq 0}\frac{(\alpha)_n(\beta)_n}{(\gamma)_n}\frac{z^n}{n!}\]
is the  hypergeometric function, $(q)_m=q(q+1)\cdots(q+m-1)$ if $m\geq 1$ and $(q)_0=1$.

In the fundamental solution $\Psi^\infty=(\Psi_1^\infty\,\Psi_2^\infty)$, the (half--)monodromy of the loop
$\Y$, twisted by $\veps_\rho(\Y)=-1$ is given by the matrix
\[\Y^\infty=
\begin{pmatrix}
\Lambda^{-1}\sqrt{q} & 0 \\
0 & \Lambda\sqrt{q}
\end{pmatrix}\]
where $q=\exp(2\pi\ii k)$, and $\Lambda=\exp(\pi\ii\lambda)$.\valeriocomment{$\Lambda=\exp(\pi\ii
\lambda)$ should be the exponental map of the (simply connected dual torus) corresponding to $P^\vee$,
but check this.} The matrix $\Y^\infty$ defines $\tilde{\mu}(\Y)$. Accounting for the normalization defined in \ref{sec:genofHext}, $\mu(\Y) = \bar{\Y}^\infty=\Y^\infty q^{-1/2}$, so that $\bar{\Y}^\infty + (\bar{\Y}^\infty)^{-1} = \Lambda
+ \Lambda^{-1}$.
 
The monodromy of the loop $T$ is more easily computed in a different solution.  Assume temporarily
that $1-c=k-1/2\notin\IZ$. Then, near $w=0$ the hypergeometric equation admits the two independent
solutions \cite[15.10]{NIST:DLMF}
\begin{align*}
	&\Psi_1^0 = F(a,b,c; w) \\
	&\Psi_2^0 = w^{1-c} F(a-c+1,b-c+1,2-c;w)
\end{align*} 
Since $e^{2 \pi \ii (1-c)} = -e^{2 \pi \ii k}=-q$, in the fundamental solution $\Psi^0 = ( \Psi_1^0\  \Psi_2^0)$, the monodromy
of $T$ is given by
\[
T^0 = 
\begin{pmatrix}
1 & \phantom{-}0 \\
0 & -q
\end{pmatrix}
\]

The fundamental solutions $\Psi^\infty$ and $\Psi^0$ are related by $\Psi^0 = \Psi^\infty K$, where $K$
is the Kummer matrix \cite[15.10]{NIST:DLMF}
\[K = 
\begin{pmatrix}
e^{a \pi \ii} \frac{\Gamma(c) \Gamma(b-a)}{\Gamma(b) \Gamma(c-a)} &  -e^{(a-c)\pi \ii}\frac{\Gamma(2-c)\Gamma(b-a)}{\Gamma(1-a)\Gamma(b-c+1)} \\[1.1ex]
e^{b \pi \ii} \frac{\Gamma(c)\Gamma(a-b)}{\Gamma(a)\Gamma(c-b)} &
-e^{(b-c)\pi \ii} \frac{\Gamma(2-c)\Gamma(a-b)}{\Gamma(1-b)\Gamma(a-c+1)} 
\end{pmatrix}\]
Moreover, $\{\bar{\Y}^\infty,T^\infty=K T^0 K^{-1}\}$ satisfy the defining relation \eqref{eqn:AHA-TY-rel} of the extended
affine Hecke alegbra $H^{\mathrm{ext}}_q$. The corresponding representation is isomorphic to $K(\Theta)$ precisely
when the eigenline for $T^\infty$ corresponding to the eigenvalue 1 is cyclic for $\bar{\Y}^\infty$. This eigenline is spanned
by the first column of $K$, which has components
\begin{align*}
\begin{pmatrix}
e^{a \pi \ii} \frac{\Gamma(c) \Gamma(b-a)}{\Gamma(b) \Gamma(c-a)}\\[1.1ex]
e^{b \pi \ii} \frac{\Gamma(c)\Gamma(a-b)}{\Gamma(a)\Gamma(c-b)}
\end{pmatrix}
=
-e^{-\ii \pi k}\cdot \sfG{\frac{3}{2}-k}\cdot 
\begin{pmatrix}
\frac{e^{\ii \pi\lambda}  \Gamma (-2 \lambda )}{\Gamma (-k-\lambda +1)\Gamma \left(\frac{1}{2}-\lambda \right) } \\[1.2ex]
\frac{e^{-\ii \pi\lambda}  \Gamma (2 \lambda )}{ \Gamma (-k+\lambda +1)\Gamma \left(\lambda +\frac{1}{2}\right)}
\end{pmatrix}
\end{align*}
Dividing by $\sfG{\frac{3}{2}-k}$ yields an eigenvector $\psi^\infty(\lambda,k)_1$ which, under the running assumption
$\lambda\notin(1/2)\IZ$, is well--defined and non--zero for any value of $k\in\IC$.
Moreover, $\psi^\infty(\lambda,k)_1$
is cyclic under $\bar{\Y}^\infty$ exactly when $-k+\lambda+1,-k-\lambda+1\notin\mathbb{Z}_{\leq 0}$, that is when $k\pm\lambda
\notin\IZ_{>0}$. 


To determine the monodromy when one of $k\pm\lambda$ lies in $\mathbb{Z}_{>0}$, consider the second column of
$K$, namely
\begin{align*}
-
\begin{pmatrix}
e^{(a-c)\pi \ii}\frac{\Gamma(2-c)\Gamma(b-a)}{\Gamma(1-a)\Gamma(b-c+1)}\\[1.1ex]
e^{(b-c)\pi \ii} \frac{\Gamma(2-c)\Gamma(a-b)}{\Gamma(1-b)\Gamma(a-c+1)} 
\end{pmatrix}
=
-ie^{\ii\pi k}\cdot\sfG{k+\frac{1}{2}}\cdot
\begin{pmatrix}
\frac{e^{\ii \pi  \lambda } \Gamma (-2 \lambda )}{\Gamma (k-\lambda )\Gamma \left(\frac{1}{2}-\lambda \right) }\\[1.1ex]
\frac{e^{-\ii \pi  \lambda } \Gamma (2 \lambda )}{\Gamma (k+\lambda )\Gamma \left(\lambda +\frac{1}{2}\right) }
\end{pmatrix}
\end{align*}
Dividing by $\sfG{k+\frac{1}{2}}$ yields a well--defined, non--zero eigenvector $\psi^\infty(\lambda,k)_{-q}$ of $T^\infty$
for any $k\in\IC$, corresponding to the eigenvalue $-q$. That vector is cyclic under $\bar{\Y}^\infty$, and thus the monodromy representation isomorphic to $K(\Theta)^-$, provided $k+\lambda,
k-\lambda\notin\mathbb{Z}_{\leq 0}$.

Note that one cannot simultaneously have $k+\lambda\in\mathbb{Z}_{>0}$ and $k-\lambda\in\mathbb{Z}_{\leq 0}$
(resp. $k-\lambda\in\mathbb{Z}_{>0}$ and $k+\lambda\in\mathbb{Z}_{\leq 0}$) since $\lambda\notin\half{1}\IZ$ by assumption.
Summarising, either $k\pm\lambda\notin\IZ_{>0}$, in which case the action of $\{\bar{\Y}^\infty,T^\infty\}$ is equivalent
to the covariant representation $K(\Theta)$, or $k\pm\lambda\notin\IZ_{\leq 0}$, in which case that action is not equivalent
to $K(\Theta)$, but is equivalent to $K(\Theta)^-$.

\subsection{Case 2: $\lambda = 0$}\label{ss:la 0}

If $\lambda=0$, $\Psi^\infty$ from \ref{ss:generic case} is no longer a fundamental solution since $\Psi^\infty_1 = \Psi^\infty_2$.  We
shall therefore modify $\Psi^\infty$ so the solutions remain independent as $\lambda \to 0$. 

By Proposition \ref{pr:non resonance}, the KZ local system $\KL$ is non--resonant at $x = 0$
for $\lambda$ in a neighborhood of $0$. By Theorem \ref{thm:holomorphicpartofLVL}, it therefore admits
a canonical solution with asymptotic expansion $(\mathrm{Id}+O(x))x^{\mathrm{Diag}(\lambda+k,-\lambda+k)}$.
Following the reduction from \ref{ss:red to HG}, applying $ch_{\J_\vartheta}$ to this canonical solution and changing coordinates to $w$ gives a fundamental solution $\Phi^\infty$ of 
the hygeometric equation
at $w = \infty$ different from $\Psi^\infty$. For $\lambda \neq 0$, when both $\Psi^\infty$ and $\Phi^\infty$ are fundamental solutions they can be shown to be related by $\Psi
^\infty=\Phi^\infty W$, where
\[
W = \left(
\begin{array}{cc}
 -\lambda  & \lambda  \\
 \phantom{-}1 & 1 \\
\end{array}
\right)
\]
by comparing asymptotics at $w=\infty$.

The monodromy with respect to $\Phi^\infty$ is given by $W \bar{\Y}^\infty W^{-1}$ and $W T^\infty W^{-1}$ which are defined at $\lambda =0$ since $\Phi^\infty$ is invertible at $\lambda = 0$.  After conjugating further by a permutation of the basis elements and a diagonal matrix, these are  

\[
	\bar{\Y}^\infty_\Phi = 
	\left(
	\begin{array}{cc}
		\cos(2 \pi \lambda)                      & - i \lambda \sin(2 \pi \lambda) \\[1.1ex]
		- \frac{i  \sin(2 \pi \lambda)}{\lambda} & \cos(2 \pi \lambda)             
	\end{array}
	\right)
\]
and
\begin{equation}
{T}^\infty_\Phi    = \left(                         
	\begin{array}{cc}
	t_{1,1}     & t_{1,2}                          \\
	t_{2,1}     & t_{2,2}                          
	\end{array}
\right)  
\label{eqn-sl2-operator-T}
\end{equation}
where
\begin{align*}
	\scriptstyle{ t_{1,1}}  & = \scriptstyle{ \frac{1}{2} \left(-e^{2 i k \pi }+\frac{\left(1+e^{2 i k \pi }\right)  \lambda  \left(\Gamma (-k-\lambda +1) \Gamma (k-\lambda ) \Gamma (\lambda )^2-\Gamma (-\lambda )^2 \Gamma (-k+\lambda +1) \Gamma (k+\lambda )\right) \sin (\pi  \lambda )}{\pi ^2 (\csc (\pi  (k-\lambda ))-\csc (\pi  (k+\lambda )))}+1\right) }                                                                             
	\\
	\scriptstyle{ t_{1,2} } & =  \scriptstyle{ \frac{1}{2} \lambda  \left(i \left(-1+e^{2 i k \pi }\right) \cot (\pi  \lambda )-\frac{\left(1+e^{2 i k \pi }\right) \lambda  \left(\Gamma (-k+\lambda +1) \Gamma (k+\lambda ) \Gamma (-\lambda )^2+\Gamma (-k-\lambda +1) \Gamma (k-\lambda ) \Gamma (\lambda )^2\right) \sin (\pi  \lambda )}{\pi ^2 (\csc (\pi  (k-\lambda ))-\csc (\pi  (k+\lambda )))}\right) }                                
	\\
	\scriptstyle{ t_{2,1} } & = \scriptstyle{ \frac{\left(1+e^{2 i k \pi }\right) \lambda  \left(\Gamma (-k+\lambda +1) \Gamma (k+\lambda ) \Gamma (-\lambda )^2+\Gamma (-k-\lambda +1) \Gamma (k-\lambda ) \Gamma (\lambda )^2\right) \sin (\pi  \lambda )-2 e^{i k \pi } \pi ^2 \cot (\pi  \lambda ) (\csc (\pi  (k-\lambda ))-\csc (\pi  (k+\lambda ))) \sin (k \pi )}{2 \pi ^2 \lambda  (\csc (\pi  (k-\lambda ))-\csc (\pi  (k+\lambda )))} } \\
	\scriptstyle{ t_{2,2} } & = \scriptstyle{ \frac{1}{2} \left(-e^{2 i k \pi }-\frac{\left(1+e^{2 i k \pi }\right) \lambda  \left(\Gamma (-k-\lambda +1) \Gamma (k-\lambda ) \Gamma (\lambda )^2-\Gamma (-\lambda )^2 \Gamma (-k+\lambda +1) \Gamma (k+\lambda )\right) \sin (\pi  \lambda )}{\pi ^2 (\csc (\pi  (k-\lambda ))-\csc (\pi  (k+\lambda )))}+1\right)  }                                                                           
\end{align*}

Taking the limit as $\lambda$ goes $0$ gives
\[
	\bar{\Y}^\infty =  
	\left( 
	\begin{array}{cc}
		1        & 0 \\
		-2 \pi \ii & 1 
	\end{array}
	\right)
\]
while ${T}^\infty_\Phi$ has matrix entries
\begin{align*}
	t_{1,1} & = -\frac{e^{i k \pi } \left(\gamma +i \pi +H_{-k}+\psi ^{(0)}(k)\right) \sin (k \pi )}{\pi } \\ t_{1,2} &= -\frac{2 e^{i k \pi } \sin (k \pi )}{\pi } \\
	t_{2,1} &= \scriptstyle{\frac{\left(\left(1+e^{2 i k \pi }\right) \pi ^2 \csc (k \pi ) \left(\pi ^2 \csc ^2(k \pi )+\left(H_{-k}+\psi ^{(0)}(k)+\gamma \right){}^2\right)-2 e^{i k \pi } \pi ^4 \cot (k \pi ) \left(2 \cot ^2(k \pi )+1\right)\right) \sin (k \pi ) \tan (k \pi )}{4 \pi ^3} } \\
	t_{2,2} &= -\frac{i \left(-1+e^{2 i k \pi }\right) \left(\gamma -i \pi +H_{-k}+\psi ^{(0)}(k)\right)}{2 \pi } 
\end{align*}
where $H_{-k} = \psi(1-k) + \gamma$ is the harmonic number, $\gamma$ the Euler--Mascheroni
constant, and $\psi ^{(0)}(z) = \Gamma'(z)/\Gamma(z)$ the digamma function.

The corresponding monodromy representation fails to be isomorphic to $K(\Theta)$ precisely when
the $T$--eigenspace corresponding to the eigenvalue 1 is not cyclic under $\Y$, that is when it is 
spanned by the second column $\Phi^\infty_2$ of $\Phi^\infty$.

If $k\notin\IZ$, $t_{1,2}$ does not vanish, and $\Phi^\infty_2$ is not an eigenvector of $T$,
and the monodromy representation is isomorphic to $K(\Theta)$.
If $k=0$, then taking the limit as $k$ goes to $0$ gives 
\[
	T = \left(
	\begin{array}{cc}
		1 & 0  \\
		0 & -1 \\
	\end{array}
	\right)
\]
and the $T$-eigenvector corresponding to $1$ is clearly cyclic for $\bar{\Y}^\infty_\Phi$. 
Thus, the monodromy representation is isomorphic to $K(\Theta)$.

When $k \in \mathbb{Z}_{\neq 0}$, the monodromy representation is not isomorphic to $K(\Theta)$.  
For $k=1$, the limit of \eqref{eqn-sl2-operator-T} is 
\[
	\left(
	\begin{array}{cc}
		-1 & 0 \\
		\phantom{-}0  & 1 \\
	\end{array}
	\right)
\]
and thus the $T$-eigenvector corresponding to $1$ is not cyclic, but the eigenvector corresponding to $-q = -1$ is cyclic for $\bar{\Y}^\infty_\Phi$.  Thus the monodromy representation is $K(\Theta)^-$.

For $k \in \IZ \setminus \lbrace 0,1 \rbrace$ we use the $k$--shift operator to relate the monodromy representation to that of $k=0$ or $k=1$.  Note that neither the $k$-shift or $\lambda$-shift changes the representation type of the monodromy.  Let $\tau$ denote the shift in parameters $(\lambda,k) \mapsto (\lambda+m,k+n)$.  Both $k$-shift and $\lambda$-shift introduce a twist by $e_\rho$ which twists the operator $(\cT{t_\rho}{\lambda}^m S_{\J_\vartheta}^n)^* (\Y) = (-1)^{m+n} \Y$.  However, twist introduced by the $k$-shift is canceled by the normalization factor $e^{\pi \ii k}$ from \ref{sec:genofHext} used to define $\bar{\Y}$ and the $\lambda$-shift is canceled by 
\[
	\tau(\Theta) = \tau(\bar{\Lambda} + 1/\bar{\Lambda}) = e^{\pi \ii (\lambda + m)} +  e^{-\pi \ii (\lambda + m)} = (-1)^n \Theta
\]
In short, the effect of the shift operators on the monodromy representation is to shift the parameters $\tau(K(\Theta)^{\pm}) = (\cT{t_\rho}{\lambda}^m S_{\J_\vartheta}^n)^* K(\Theta)^{\pm}$.

For $k \in \mathbb{Z}_{\geq 1}$ the $k$-shift operator is invertible and so the monodromy at $k \in \mathbb{Z}_{\geq 1}$ is isomorphic to $ K(\Theta)^-$ but not $K(\Theta)$.
%
Similarly for $k \in \mathbb{Z}_{\leq -1}$ the $k \mapsto k+1$ shift operator is invertible.  So all such integers have the same monodromy as for $k=0$ up to a sign twist $K(\Theta)$.

\subsection{Case 3: $\lambda = 1/2$}\label{subsec:proofrank1lambdahalf}

If $k \neq 1/2$, the hypergeometric series in $\Psi^\infty$ have the special values \cite[15.4]{NIST:DLMF} 
\[
\Psi^\infty_1 = w^{1/2-k} \qquad \Psi^\infty_2 = (w-1)^{-1/2+k}
\]
With respect to these solutions
\[
\bar{\Y}^\infty = 
\left(
\begin{array}{cc}
 i & 0 \\
 0 & -i \\
\end{array}
\right)
\aand
T^\infty 
=
\frac{1}{2} \left(
\begin{array}{cc}
  \left(1-q\right) & \frac{1+q}{1-2 k} \\
  \left(1+q\right) (1-2 k) &  \left(1-q\right) \\
\end{array}
\right)
\]
 The eigenvector for $T^\infty$ corresponding to $1$ is $ \Psi^\infty_1 + (1-2k)\Psi^\infty_2$.
Since it is cyclic for $\bar{\Y}^\infty$, the monodromy is isomorphic to $K(\Theta)$ in this case. 

When $k =1/2$, the hypergeometric system has solutions
\[
\Psi_1^\infty = 1 \qquad \Psi_2^\infty = -\log(1-w) + \log(w)
\]
giving normalized monodromy
\[
\bar{\Y}^\infty = 
\left(
\begin{array}{cc}
 i & 0 \\
 0 & -i \\
\end{array}
\right)
\aand
T^\infty = 
\left(
\begin{array}{cc}
 1 & \frac{\pi \ii}{2} \\
 0 & 1 \\
\end{array}
\right)
\]
The $\bar{\Y}^\infty$-eigenvector corresponding to $-\ii$ generates the representation under the $T^\infty$ action,
and the monodromy representation is isomorphic to $\I{-\ii}$. 
Any $T^\infty$-eigenvector corresponding to $1$ is an $\bar{\Y}^\infty$-eigenvector and thus cannot be cyclic under $\bar{\Y}^\infty$.  
Thus the monodromy representation is not isomorphic to $K(\Theta)$.   

\subsection{Case 4: $\lambda \in \mathbb{Z}_{\neq 0}$}\label{subsec:proofrank1lambdaint}

Assume that $(\lambda,k)$ does not satisfy
\begin{equation}\label{non-shiftable-points}
	\lambda, k \in \mathbb{Z} \text{ and } |\lambda| \geq k \text{ and } |\lambda| \geq 1 -  k
\end{equation}
In this case, the $\lambda$-shift operator $K_{q(\lambda)} \to K_{0}$ is invertible by the conditions given in \ref{subsec:monoKZrankone} and thus $\KL$ is isomorphic $\KL[0,k]$, which we computed in Case 2, cf. \ref{ss:la 0}.  By \ref{ss:la 0}, $\KL$ is thus isomorphic to $K(\Theta)$ when $k \notin \IZ_{>0}$ and $K(\Theta)^{-}$ when $k \in \IZ_{>0}$.

The points which satisfy \eqref{non-shiftable-points} are shown in Figure \ref{fig:sl2-param-space}.  These may be shifted using the $k$ and $\lambda$-shift operators to either of $(\pm 1,0)$.  That is, $\KL$ is ismorphic to one of $\KL[\pm 1,0]$.
\label{fig:sl2-param-space}
\begin{figure}[ht]
	\centering
		\begin{tikzpicture}[scale=0.3]
  \begin{axis}[
      color =black,
      xlabel=k,
      ylabel=$\lambda$,
      axis lines=middle,
      xmin=-3, xmax=3,
      ymin=-3, ymax=3,
      xtick={-3,-2,-1,0,1,2,3}, ytick={-3,-2,-1,0,1,2,3},
      width=18cm,
      line width=1pt,
      x label style={at={(axis description cs:1.01,0.5)},anchor=west,font=\huge},
      y label style={at={(axis description cs:0.5,1.01)},anchor=south,font=\huge},
  ]
  \addplot [color=red,only marks,mark size=5pt] table[row sep=\\]{%
      0 1\\
      -1 2\\
      0 2\\
      1 2\\
      -2 3\\
      -1 3\\
      0 3\\
      1 3\\
      2 3\\
      1 1\\
      2 2\\
      3 3\\
  };
  \addplot [color=red,only marks,mark size=8pt,mark=x] table[row sep=\\]{%
      1  -1\\
      -1 -2\\
      1 -2\\
      2 -2\\
      -2 -3\\
      -1 -3\\
      1 -3\\
      2 -3\\
      3 -3\\
      0 -1\\
      0 -2\\
      0 -3\\
  };
  \addplot [color=blue!90,line width=2pt,domain=-3:0, samples=2] {x};
  \addplot [color=blue!90,line width=2pt,domain=0.5:3, samples=2] {x};
  \addplot [color=blue!90,line width=2pt,domain=-3:0.5, samples=2] {x-1};  
  \addplot [color=blue!90,line width=2pt,domain=1:3, samples=2] {x-1}; 
  \addplot [color=blue!90,line width=2pt,domain=-3:0, samples=2] {-x};
  \addplot [color=blue!90,line width=2pt,domain=0.5:3, samples=2] {-x};
  \addplot [color=blue!90,line width=2pt,domain=-3:0.5, samples=2] {-x+1};  
  \addplot [color=blue!90,line width=2pt,domain=1:3, samples=2] {-x+1}; 
  \addplot [color=black!30,dashed,line width=3pt,domain=-3:3, samples=2] {1};
  \addplot [color=black!30,dashed,line width=3pt,domain=-3:3, samples=2] {2};
  \addplot [color=black!30,dashed,line width=3pt,domain=-3:3, samples=2] {3};
  \addplot [color=black!30,dashed,line width=3pt,domain=-3:3, samples=2] {-1};
  \addplot [color=black!30,dashed,line width=3pt,domain=-3:3, samples=2] {-2};
  \addplot [color=black!30,dashed,line width=3pt,domain=-3:3, samples=2] {-3};
          \coordinate (A) at (axis cs:{2,2});
      \coordinate (B) at (axis cs:{1,2});
      \coordinate (C) at (axis cs:{0,2});
      \coordinate (D) at (axis cs:{0,1});
      \draw[line width=3pt,red,-stealth] (A)--(B);
      \draw[line width=3pt,red,-stealth] (B)--(C);
      \draw[line width=3pt,red,-stealth] (C)--(D);
  \end{axis}
\end{tikzpicture}
	\caption{
	Points marked by red dots or crosses cannot be shifted using $\lambda$ shifts to $\lambda=0$, but can be shifted by $k$ and $\lambda$ shifts to one of two points. 
  The blue lines represent parameter values which cannot be shifted either to or from by $\lambda \mapsto \lambda+1$ or $k \mapsto k+1$.  
  }
\end{figure}
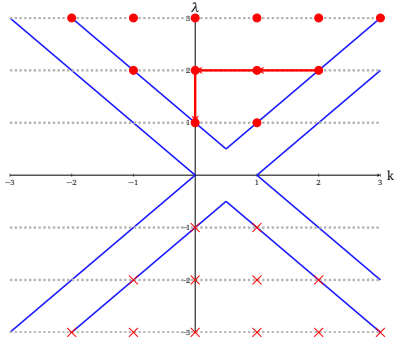
We compute $\KL[\pm 1,0]$ without reducing to the scalar hypergeometric system.  For $y \in \mathfrak{h}$, the KZ connection has the form $x \partial_x + \tfact{y}$ with solution $\Phi(x) = x^y$.  
The monodromy can be computed directly as $\bar{\Y} = \Y = e^{\pi \ii y} = - \mathrm{Id}$.  The solution $\Phi(x)$ is holomorphic at $x=1$ and thus $T = \tfact{s} = \mathrm{Diag}(1,-1)$.  Thus the monodromy representation splits into a direct sum $(-1,1) \oplus (-1,-1)$.  Thus for all $(\lambda,k)$ satisfying \eqref{non-shiftable-points} the monodromy representation is isomorphic to $(-1,1) \oplus (-1,-1)$. 
The $T^\infty$-eigenspace is not cyclic, and thus the monodromy representation is not isomorphic to $K(\Theta)$.

\subsection{Case 5: $\lambda \in \half{1}+\IZ$}\label{ss:half nonzero}

If $(\lambda,k)$ does not satisfy 
\begin{equation}\label{eq-case6-non-invert}
 \lambda, k \in \mathbb{Z} + 1/2 \text{ and } |\lambda| > k \text{ and } |\lambda| > -k +1
\end{equation}
then the $\lambda$-shift operator 
$K_{q(\lambda)} \to K_{1/2}$ is invertible by the conditions given in \ref{subsec:monoKZrankone} and thus $\KL$ is isomorphic $\KL[q(1/2),k]$, which we computed in Case 3, cf. \ref{subsec:proofrank1lambdahalf}.  Note that assuming \eqref{eq-case6-non-invert} does not hold implies $k \neq 1/2$. Thus by \ref{subsec:proofrank1lambdahalf}, $\KL$ is isomorphic to $K(\Theta)$.


In the case \eqref{eq-case6-non-invert} holds, shifting by $k$ and $\lambda$ gives an isomorphism between $\KL$ and $\KL[1/2,1/2]$ which we computed in \ref{subsec:proofrank1lambdahalf} to be isomorphic to $\I{-\ii}$  but not isomorphic to $K(\Theta)$.


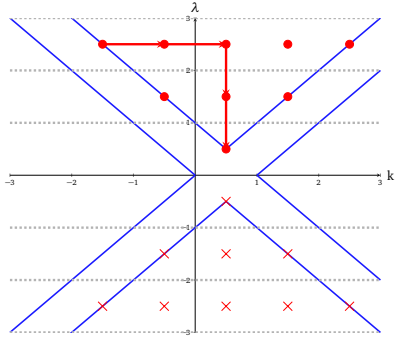
\begin{figure}[ht]
  \centering
  \begin{tikzpicture}[scale=0.3]
    \begin{axis}[
        color =black,
        xlabel=k,
        ylabel=$\lambda$,
        axis lines=middle,
        xmin=-3, xmax=3,
        ymin=-3, ymax=3,
        xtick={-3,-2,-1,0,1,2,3}, ytick={-3,-2,-1,0,1,2,3},
        width=18cm,
        line width=1pt,
        x label style={at={(axis description cs:1.01,0.5)},anchor=west,font=\huge},
        y label style={at={(axis description cs:0.5,1.01)},anchor=south,font=\huge},
    ]
    \addplot [color=red,only marks,mark size=5pt] table[row sep=\\]{%
        0.5 0.5\\
        -0.5 1.5\\
        0.5 1.5\\
        1.5 1.5\\
        -1.5 2.5\\
        -0.5 2.5\\
        0.5 2.5\\
        1.5 2.5\\
        2.5 2.5\\
    };
    \addplot [color=red,only marks,mark size=8pt,mark=x] table[row sep=\\]{%
        0.5 -0.5\\
        -0.5 -1.5\\
        0.5 -1.5\\
        1.5 -1.5\\
        -1.5 -2.5\\
        -0.5 -2.5\\
        0.5 -2.5\\
        1.5 -2.5\\
        2.5 -2.5\\
    };
    \addplot [color=blue!90,line width=2pt,domain=-3:0, samples=2] {x};
    \addplot [color=blue!90,line width=2pt,domain=0.5:3, samples=2] {x};
    \addplot [color=blue!90,line width=2pt,domain=-3:0.5, samples=2] {x-1};  
    \addplot [color=blue!90,line width=2pt,domain=1:3, samples=2] {x-1}; 
    \addplot [color=blue!90,line width=2pt,domain=-3:0, samples=2] {-x};
    \addplot [color=blue!90,line width=2pt,domain=0.5:3, samples=2] {-x};
    \addplot [color=blue!90,line width=2pt,domain=-3:0.5, samples=2] {-x+1};  
    \addplot [color=blue!90,line width=2pt,domain=1:3, samples=2] {-x+1}; 
    \addplot [color=black!30,dashed,line width=3pt,domain=-3:3, samples=2] {1};
    \addplot [color=black!30,dashed,line width=3pt,domain=-3:3, samples=2] {2};
    \addplot [color=black!30,dashed,line width=3pt,domain=-3:3, samples=2] {3};
    \addplot [color=black!30,dashed,line width=3pt,domain=-3:3, samples=2] {-1};
    \addplot [color=black!30,dashed,line width=3pt,domain=-3:3, samples=2] {-2};
    \addplot [color=black!30,dashed,line width=3pt,domain=-3:3, samples=2] {-3};
          \coordinate (A) at (axis cs:{-1.5,2.5});
      \coordinate (B) at (axis cs:{-0.5,2.5});
      \coordinate (C) at (axis cs:{0.5,2.5});
      \coordinate (D) at (axis cs:{0.5,1.5});
      \coordinate (E) at (axis cs:{0.5,0.5});
      \draw[line width=3pt,red,-stealth] (A)--(B);
      \draw[line width=3pt,red,-stealth] (B)--(C);
      \draw[line width=3pt,red,-stealth] (C)--(D);
      \draw[line width=3pt,red,-stealth] (D)--(E);
    \end{axis}
  \end{tikzpicture}

  \caption{Points which do not satisfy \eqref{eq-case6-non-invert}.  Example of shift.}
  \label{fig:sl2-param-space-case6}
\end{figure}
\section{\Done Monodromy in higher rank}\label{sec:monodromy-higher-rank}

We restrict to type $\sfA_n$. In this section we classify points $(\lambda,k)$ in the parameter space $\frakh^* \times \IC$ for $\Kconn$ which are non-resonant or can be shifted to non-resonant parameters or cannot be shifted to non-resonant parameters.

\subsection{Main Result}

%
Define $B = \lbrace \alpha^\vee_{j_1}+\ldots+\alpha^\vee_{j_i}\ |\ \alpha^\vee_{j} \text{ are orthogonal coroots} \rbrace $.  By Proposition \ref{pr:non resonance} the connection $\nabla$ is resonant for $\lambda$ in
    \begin{align}
    \mathrm{Res} &= \lbrace \lambda\ |\  \lambda(q) \in \IZ_{\neq 0} \text{ for some } q \in B \rbrace 
  \end{align}

Define the following subsets of $\mathfrak{h}^*$, 
  \begin{align}
    E & = \bigcup_{w \in W} w(\lbrace \lambda\ |\ \lambda(\alpha^\vee) \notin \mathbb{R}_{\leq k} \ \forall\ \alpha \in \pos \rbrace)  \nonumber \\
    S & =  \lbrace \lambda\ |\ \lambda(q) \in \mathbb{Z} \text{ for a unique } q
    \in B \rbrace \nonumber
    \end{align}
We then have the following classification. 

\label{thm:sln-shift}
\begin{theorem}  In type $\mathsf{A}_n$, if $\Kconn$ is resonant at $(\vartheta,k)$ and $\lambda \in q^{-1}(\vartheta)$ then if one of the following conditions holds $\Kconn$ has monodromy equivalent to that of non-resonant system: 
\begin{enumerate}
\item $\lambda$ is affine $k$--regular and $\lambda \in S$
\item $k \in \IZ$ and $\lambda \in E$
\item $k \in \IZ$ and $\lambda \in S$.
\end{enumerate}
%
In particular, since condition (1) is generic in $\mathrm{Res}$, the set of points which are resonant and cannot be shifted to a non-resonant point has codimension 2.
\end{theorem} 


Note that, unlike Theorem \ref{th:main 1}, Theorem \ref{thm:sln-shift} does not describe the set of parameters
$(\vartheta,k)$ for which Conjecture \ref{conj:monoconj} holds. Proposition \ref{pr:generic monodromy} shows
that Conjecture \ref{conj:monoconj} holds when $(\vartheta,k)$ lie in a codimension 1 set which includes both
resonant and non-resonant parameters. As in rank 1, we expect in each case of non--resonant, resonant but
equivalent to non-resonant, or neither that there exist parameters where the monodromy representation is 
isomorphic to $K(\Theta)$ and where it is not. Corollary \ref{cor:highrankFGfail} uses Theorem \ref{thm:sln-shift}
to provide a set of examples which are resonant and Conjecture \ref{conj:monoconj} fails.  

\subsection{Proof Overview}

We find an isomorphism from resonant $\Kconn$ to non-resonant $\Kconn[\vartheta',k']$. This isomorphism uses three tools described above,  the shift operators in $\lambda$ and $k$  and the map between $\dI{\lambda}$ and $\J_\vartheta$. 
 We recall now when these operators and maps are invertible.
Define $R_i = \lbrace \beta \in \pos\ |\  \lambda_i(\beta^\vee) \neq 0 \rbrace$. It follows from 
Proposition \ref{pr:det cT}
that the $\lambda$ shift operator is invertible and thus the representation $\dI{\lambda}$ and $\dI{\lambda + \lambda_i}$ are isomorphic when 
  \begin{equation}\label{sln-lambda-shift-invcond}
    \lambda(\alpha^\vee) \neq \pm k \text{ for all } \alpha \in R_i
  \end{equation}
By Theorem \ref{th:det R}, $\dI{\lambda}$ and $\J_\vartheta$ are isomorphic when $\lambda(\alpha^\vee) \neq k_\alpha$ for all $\alpha \in \pos$.   For resonant $(\vartheta,k)$, we try to find $\lambda \in q^{-1}(\vartheta)$ such that $\lambda(\alpha^\vee) \neq k_\alpha$ for all $\alpha \in \pos$ and thus $K_\vartheta$ is isomorphic $\dI{\lambda}$. We then shift $\dI{\lambda}$ to $\dI{\lambda'}$ where $\lambda'$ is not resonant.   
By Theorem \ref{thm:k-shift-op-invert}, the shift operator $k \mapsto k+1$ is invertible when $\lambda$ is $k$--regular.
That is, the monodromy representation of $\Kconn$ is isomorphic to $\Kconn[\vartheta,k+1]$ where both are derived from $\J_\vartheta$ but as a representations of $\mathrm{H^{trig}}_k$ and $\mathrm{H^{trig}}_{k+1}$ respectively.


\subsection{Case 1: $\lambda$ affine $k$--regular, $\lambda \in S$}

Under the assumption of affine $k$--regularity, any $\lambda \in q^{-1}(\vartheta)$ has $\dI{\lambda}$ isomorphic to $\J_\vartheta$ and thus it suffices to consider the system defined by $\cI_\lambda$.  Let $q \in B$ be the unique element such that $\lambda(q) \in \IZ$.  In the case $\lambda(q) = 0$, the connection is non-resonant.  Assume $\lambda(q) \in \IZ_{\neq0}$.  Using the notation $I_+$ and $I_-$ from \ref{pr:non resonance}, write $q = \sum_{i \in I_+} e_i - \sum_{i \in I_-} e_i$ and let $m = \mathrm{min}(I_+ \cup I_-)$.  Then $|\lambda_m(q)| = 1$. Set $\mu = \lambda - \lambda(q)\lambda_m(q) \lambda_m$.  Then $\mu(q) = \lambda(q) - \lambda(q) \lambda_m(q)^2 = 0$.  Since $\mu$ is shifted relative to $\lambda$ by an integral weight, the set of elements of $B$ which pair with $\lambda$ or $\mu$ to give an integer is the same.  That is $B_\lambda = B_\mu = \lbrace q \rbrace$.  Thus $\cI_\mu$ is non-resonant.  Under the assumption of $k$-regularity, the monodromy of $\cI_\mu$ and $\cI_\lambda$ are isomorphic since the $\lambda$--shift operator is invertible.

\begin{example}[The case of $\mathsf{A}_2$.]  
If $l=1,$ then $\lambda = -m \lambda_1 + x \lambda_2$.  The set of $\lambda$ which cannot be shifted to a non-resonant point  thus have $x = \pm k + j$ where $1 \leq j \leq m$.  That is, $\lambda = -m \lambda_1 + (k \pm + j) \lambda_2$ as shown in \autoref{fig-sl3-bad-pts1}.

\begin{figure}[ht] 
  \newcommand*\rows{10}
  \begin{center}
    \begin{tikzpicture}[scale=1]
      \clip (-2.5,-2.5) rectangle (2.5,2.5);
      \draw[dashed, black!20] ($-5*(1,{sqrt(3)})$) -- ($({0},0)+(5,{5*sqrt(3)})$);
      \draw[dashed, black!20] ($-5*(1,-{sqrt(3)})$) -- ($({0},0)+(5,{-5*sqrt(3)})$);
      \draw[dashed, black!20] ($(-4,0)$) -- ($(4,0)$);
      \foreach \x in {-5, -4, -3,-2,-1,1,2,3,4,5} {
        \draw[black!70] ($({\x*sqrt(2/3)},0)-5*(1,{sqrt(3)})$) -- ($({\x*sqrt(2/3)},0)+(5,{5*sqrt(3)}) 											$);
      }
      \foreach \x in {-5, -4, -3,-2,-1,1,2,3,4,5} {
        \draw[black!70] ($({\x*sqrt(2/3)},0)-5*(1,-{sqrt(3)})$) -- ($({\x*sqrt(2/3)},0)+(5
        ,{-5*sqrt(3)})$);
      }
      \foreach \y in {-5, -4, -3,-2,-1,1,2,3,4,5} {
        \draw[black!70] ($(-4,{\y*sqrt(2/3)*0.5*sqrt(3)})$) -- ($(4,{\y*0.5*sqrt(3)*sqrt(2/3)})$);
      }
      \node at (0,0)[circle,fill,inner sep=2pt]{};
      \node (la1) at ({sqrt(2/3)},{0}) {};
      \node (la2) at ({sqrt(1/6)},{sqrt(1/2)}) {};
      \node (la3) at ($(la1)-(la2)$) {};
      \draw[line width=1pt,red,-stealth](0,0)--({sqrt(2/3)},{0}) node[anchor=north east]{$					
        \lambda_1$};
      \draw[orange, line width=2pt, dotted] ($(-4,{4/3*sqrt(2/3)*0.5*sqrt(3)})$) -- ($(4
      ,{4/3*0.5*sqrt(3)*sqrt(2/3)})$);
      \draw[blue, line width=2pt, dashed] ($(-4,{-4/3*sqrt(2/3)*0.5*sqrt(3)})$) -- ($(4
      ,{-4/3*0.5*sqrt(3)*sqrt(2/3)})$);
      \draw[orange, line width=2pt, dotted] 
      ($({4/3*sqrt(2/3)},0)-5*(1,{sqrt(3)})$) -- ($({4/3*sqrt(2/3)},0)+(5,{5*sqrt(3)})$);
      \draw[blue, line width=2pt, dashed] 
      ($({-4/3*sqrt(2/3)},0)-5*(1,{sqrt(3)})$) -- ($({-4/3*sqrt(2/3)},0)+(5,{5*sqrt(3)})$);
      \draw[orange, line width=2pt, dotted] 
      ($({4/3*sqrt(2/3)},0)-5*(1,-{sqrt(3)})$) -- ($({4/3*sqrt(2/3)},0)+(5,{-5*sqrt(3)})$);
      \draw[blue, line width=2pt, dashed] 
      ($({-4/3*sqrt(2/3)},0)-5*(1,-{sqrt(3)})$) -- ($({-4/3*sqrt(2/3)},0)+(5,{-5*sqrt(3)})$);

      \tikzstyle{prob} = [purple!80,regular polygon,regular polygon sides=3,fill,inner sep=1.2pt]   	   
      \tikzstyle{prob2} = [purple!50,regular polygon,regular polygon sides=3,fill,inner sep=1.2pt]   	   
      \node (nu) at (${4/3}*({sqrt(1/6)},{sqrt(1/2)})$) [circle, fill, inner sep=2pt]{};
      \node at (nu) [anchor=south west] {$\nu_2$};
      
      \node at ($(nu) - {1}*(la3) - {0}*(la2) $) [prob]{}; 
      \node at ($(nu) - {2}*(la3) - {0}*(la2) $) [prob]{}; 
      \node at ($(nu) - {2}*(la3) - {1}*(la2) $) [prob]{}; 
      
      \node at ($(nu) - {3}*(la3) - {0}*(la2) $) [prob]{}; 
      \node at ($(nu) - {3}*(la3) - {1}*(la2) $) [prob]{}; 
      \node at ($(nu) - {3}*(la3) - {2}*(la2) $) [prob]{}; 
      
      \node at ($(nu) - {4}*(la3) - {0}*(la2) $) [prob]{}; 
      \node at ($(nu) - {4}*(la3) - {1}*(la2) $) [prob]{}; 
      \node at ($(nu) - {4}*(la3) - {2}*(la2) $) [prob]{}; 
      \node at ($(nu) - {4}*(la3) - {3}*(la2) $) [prob]{};

      \node (mnu) at (${-1}*(nu)$) [circle, fill, inner sep=2pt]{};
      \node at (mnu) [anchor=south west] {$-\nu_2$};

      \node at ($(mnu) - {1}*(la3) - {0}*(la2) $) [prob2]{}; 
      \node at ($(mnu) - {2}*(la3) - {0}*(la2) $) [prob2]{}; 
      \node at ($(mnu) - {2}*(la3) - {1}*(la2) $) [prob2]{}; 
      
      \node at ($(mnu) - {3}*(la3) - {0}*(la2) $) [prob2]{}; 
      \node at ($(mnu) - {3}*(la3) - {1}*(la2) $) [prob2]{}; 
      \node at ($(mnu) - {3}*(la3) - {2}*(la2) $) [prob2]{}; 
      
      \node at ($(mnu) - {4}*(la3) - {0}*(la2) $) [prob2]{}; 
      \node at ($(mnu) - {4}*(la3) - {1}*(la2) $) [prob2]{}; 
      \node at ($(mnu) - {4}*(la3) - {2}*(la2) $) [prob2]{}; 
      \node at ($(mnu) - {4}*(la3) - {3}*(la2) $) [prob2]{};    	   
             
    \end{tikzpicture}
    \caption{Slice of the parameter space $\mathfrak{h}^* \times \IC$ at $k = 4/3$ in type $\mathsf{A}_2$.  Triangles are $\lambda$ which cannot be shifted to non-resonant $\mu$.  }
    \label{fig-sl3-bad-pts1}
  \end{center}
\end{figure}
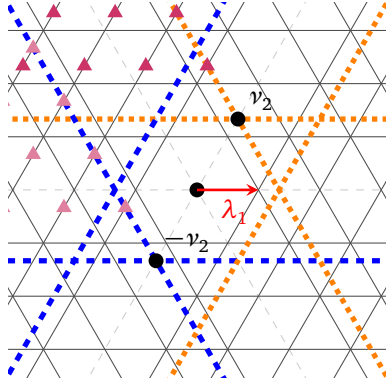
\end{example}
\subsection{Case 2: $k \in \mathbb{Z}, \lambda \in E,$ and $\lambda$ resonant.}\label{subsec:gen-rank-case-4}

We consider three subcases, when $k = 0$, $k \in \mathbb{Z}_{<0}$, or $k \in \mathbb{Z}_{>0}$.

When $k=0$, the connection is 
  $\Kconnvf[\vartheta,k]{\lambda_i^\vee} = Z_i \partial_{Z_i} - A_i$.
where $A_i = -\lambda_i^\vee$ is a constant in $Z_i$.  There is no term which is singular on the root hyperplanes.
This has fundamental solution  
\[
  \Phi_0(Z) = \prod_i Z_i^{A_i}
\]
which is, by definition, the large-volume limit.  Thus when $k=0$ a canonical solution always exists, making this case effectively non--resonant.


Now, we handle the case in which $k > 0$ and
\[ 
  \lambda \in E_+ = \lbrace \lambda\ |\ \lambda(\alpha^\vee) \notin \mathbb{R}_{\leq k} \ \forall\ \alpha \in \pos \rbrace 
\]
Then by Theorem \ref{thm:k-shift-op-invert}, the shifts
\[
  (\lambda,0) \to (\lambda,1) \to ... \to (\lambda, k)
\]
are all invertible.  So $(\lambda,k)$ and $(\lambda,0)$ have equivalent monodromy.

For $k > 0$ and $\lambda \in E$, then some $\lambda^w \in E_+$. Since the invertibility conditions on the $k$--shift are $W$-invariant the same shifting shows $(\lambda,k)$ and $(\lambda,0)$ have equivalent monodromy. 

Now take $k < 0$ and $\lambda \in E_+$.  Then similar to above, the shifts
\[
  (\lambda,k) \to (\lambda,k+1) \to ... \to (\lambda, 0)
\]
are all invertible.  So $(\lambda,k)$ and $(\lambda,0)$ have equivalent monodromy.   If $\lambda \in E$ then the same shift applies as before. 

%

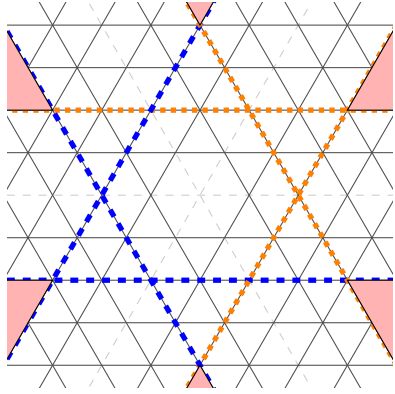
\begin{figure}[ht] 
  \begin{center}
    \begin{tikzpicture}[scale=0.8]
      \clip (-3.2,-3.2) rectangle (3.2,3.2);
      \draw[dashed, black!20] ($-5*(1,{sqrt(3)})$) -- ($({0},0)+(5,{5*sqrt(3)})$);
      \draw[dashed, black!20] ($-5*(1,-{sqrt(3)})$) -- ($({0},0)+(5,{-5*sqrt(3)})$);
      \draw[dashed, black!20] ($(-4,0)$) -- ($(4,0)$);
      \foreach \x in {-5, -4, -3,-2,-1,1,2,3,4,5} {
        \draw[black!70] ($({\x*sqrt(2/3)},0)-5*(1,{sqrt(3)})$) -- ($({\x*sqrt(2/3)},0)+(5,{5*sqrt(3)}) 											$);
      }
      \foreach \x in {-5, -4, -3,-2,-1,1,2,3,4,5} {
        \draw[black!70] ($({\x*sqrt(2/3)},0)-5*(1,-{sqrt(3)})$) -- ($({\x*sqrt(2/3)},0)+(5
        ,{-5*sqrt(3)})$);
      }
      \foreach \y in {-5, -4, -3,-2,-1,1,2,3,4,5} {
        \draw[black!70] ($(-4,{\y*sqrt(2/3)*0.5*sqrt(3)})$) -- ($(4,{\y*0.5*sqrt(3)*sqrt(2/3)})$);
      }

      \node (la1) at ({sqrt(2/3)},{0}) {};
      \node (la2) at ({sqrt(1/6)},{sqrt(1/2)}) {};
      \node (la3) at ($(la1)-(la2)$) {};

      \draw[orange, line width=2pt, dotted] ($(-4,{2*sqrt(2/3)*0.5*sqrt(3)})$) -- ($(4
      ,{2*0.5*sqrt(3)*sqrt(2/3)})$);
      \draw[blue, line width=2pt, dashed] ($(-4,{-2*sqrt(2/3)*0.5*sqrt(3)})$) -- ($(4
      ,{-2*0.5*sqrt(3)*sqrt(2/3)})$);
      \draw[orange, line width=2pt, dotted] 
      ($({2*sqrt(2/3)},0)-5*(1,{sqrt(3)})$) -- ($({2*sqrt(2/3)},0)+(5,{5*sqrt(3)})$);
      \draw[blue, line width=2pt, dashed] 
      ($({-2*sqrt(2/3)},0)-5*(1,{sqrt(3)})$) -- ($({-2*sqrt(2/3)},0)+(5,{5*sqrt(3)})$);
      \draw[orange, line width=2pt, dotted] 
      ($({2*sqrt(2/3)},0)-5*(1,-{sqrt(3)})$) -- ($({2*sqrt(2/3)},0)+(5,{-5*sqrt(3)})$);
      \draw[blue, line width=2pt, dashed] 
      ($({-2*sqrt(2/3)},0)-5*(1,-{sqrt(3)})$) -- ($({-2*sqrt(2/3)},0)+(5,{-5*sqrt(3)})$);
      \draw[fill=red!30] ($4*(la2)+2*(la2)+2*(la1)$) -- ($2*(la2)+2*(la1)$) -- ($4*(la1)+2*(la2)+2*(la1)$);
      \draw[fill=red!30] (${-4}*(la2)-2*(la2)-2*(la1)$) -- (${-2}*(la2)-2*(la1)$) -- (${-4}*(la1)-2*(la2)-2*(la1)$);
      \draw[fill=red!30] ($4*(la3)+2*(la3)+2*(la1)$) -- ($2*(la3)+2*(la1)$) -- ($4*(la1)+2*(la3)+2*(la1)$);
      \draw[fill=red!30] (${-4}*(la3)-2*(la3)-2*(la1)$) -- (${-2}*(la3)-2*(la1)$) -- (${-4}*(la1)-2*(la3)-2*(la1)$);
      \draw[fill=red!30] ($4*(la3)+2*(la3)-2*(la2)$) -- ($2*(la3)-2*(la2)$) -- (${-4}*(la2)+2*(la3)-2*(la2)$);
      \draw[fill=red!30] (${-4}*(la3)-2*(la3)+2*(la2)$) -- (${-2}*(la3)+2*(la2)$) -- (${4}*(la2)-2*(la3)+2*(la2)$);


    \end{tikzpicture}
    \caption{Slice of the parameter space $\mathfrak{h}^* \times \IC$ in type $\mathsf{A}_2$ at $k = 2$.  The set $E$ is shaded red. 
     }
    \label{fig-spaceE}
  \end{center}
\end{figure}

\subsection{Case 3: $k \in \mathbb{Z}$, $\lambda \in S$.}\label{subsec:gen-rank-case-6}

Let $q \in B$ be the unique element such that $\lambda(q) \in \IZ$.  Set $n = \lambda(q)$. There are four possibilities:
\begin{enumerate}
  \item $|k| \leq |n|, k \geq 0$,
  \item $|k| < |n|, k \leq 0 $,
  \item $|k| > |n|$,
  \item $|k| = |n|, k \leq 0$.
\end{enumerate}

If (1), 
then we can shift $k$ to $0$.  As above, at $k=0$, the connection is effectively non-resonant.  For $0 < k \leq |n|$, in order to shift $(\lambda,0) \mapsto (\lambda,k)$ it is sufficient to show that the $k$--shift operator be invertible at all $(\lambda,j)$ where $0 \leq j < k$.  This holds if and only if $\lambda(\alpha^\vee) \neq \pm j$ for all $\alpha \in R$.  Since $R^\vee \subset B$, then $\lambda(\alpha^\vee)$ can only be an integer if $q = \alpha^\vee$.  However, $|\lambda(\alpha^\vee)| = |n| \geq |k| > |\pm j|$, and so $\lambda(\alpha^\vee) \neq \pm j$.  
 Thus the $k$--shift operator is invertible at all intermediate points and the monodromy representations at $(\lambda,k)$ and $(\lambda,0)$ are isomorphic. 

The subcase (2) is very similar.  We shift $k$ to $0$.  To do so, the $k$--shift operator must be invertible at all $(\lambda,j)$ where $k \leq j < 0$.  In case $q = \alpha^\vee$, then $|\lambda(\alpha^\vee)| = |n| > |k| \geq |\pm j|$, so $\lambda(\alpha^\vee) \neq \pm j$.  

%

If (3), then $\lambda$ is $k$--regular, and thus $\dI{\lambda}$ is isomorphic to $\J_\vartheta$ for all $\lambda \in q^{-1}(\vartheta)$.  By taking $w(\lambda)$ where $w$ interchanges $I_+$ and $I_-$ we can assume $n = \lambda(q) \in \mathbb{Z}_{\geq 0}$.  As in Case 1, chose $\lambda_m$ such that $|\lambda_m(q)|=1$ and set $\mu = \lambda - \lambda(q) \lambda_m(q) \lambda_m$.  If $\lambda_m(q)=1$, this shift is invertible when $(\lambda - j \lambda_m )(\alpha^\vee) \neq \pm k$ for all $\alpha \in(\pos)_m$ and $1 \leq j \leq |n|$.  If $q \notin R^\vee$, then this condition is always satisfied since $k \in \IZ$ and $(\lambda-j\lambda_m)(\alpha^\vee) \notin \IZ$. 
In the case of $q = \alpha^\vee$, we have $\alpha \in (\pos)_m$ and $(\lambda - j \lambda_m )(\alpha^\vee) = |n| - j $ which is bounded $0 \leq  |n|-j \leq |k|-2$
and so $(\lambda - j \lambda_m )(\alpha_m^\vee) \neq \pm k $ for $1 \leq j \leq |n|$.
For $\beta \in\pos\setminus \lbrace \alpha \rbrace$, the value $(\lambda - j \lambda_m)(\beta^\vee)$ is not an integer and thus cannot equal $\pm k$ for any $j$.  
If $\lambda_m(q) = -1$, the argument is similar but we need $(\lambda+j\lambda_m)(\alpha^\vee) \neq \pm k$ for all $\alpha \in (\pos)_m$ and $0 \leq j \leq |n|-1$. 
%

If (4), we can choose $\lambda$ such that $\lambda(q) = -k$.  If $q \notin R^\vee$, then the $k$-shift is invertible and $k$ may be shifted to 0. 
 If $k$=0, we are likewise done.  
If $k < 0$ and $q = \alpha^\vee$, then since $k \in \mathbb{Z}$ and only a single $\alpha \in \pos$ has $\lambda(\alpha^\vee) \in \mathbb{Z}$, this guarantees there is no $\alpha \in \pos$ such that $\lambda(\alpha^\vee) = k$.  Thus by Theorem \ref{th:det R}, $\dI{\lambda}$ is isomorphic to $\J_\vartheta$.
Choosing $\lambda_m$ as before, we can shift $\mu = \lambda + k \lambda_m$ to $\lambda$.  The shift operator is invertible if $(\lambda + j \lambda_m)(\alpha^\vee) \neq \pm k$ for $k \leq j < 0$.  That is $-k + j \neq \pm k$ which holds for $k \leq j <0$.  For other $\beta \in (\pos)_m$, it holds that $(\lambda + j \lambda_m)(\beta^\vee) \neq \pm k$ since $\lambda(\beta^\vee) \notin \mathbb{Z}$.


\omitrobincomment{could add illustration or rank 2 }


\subsection{Failure of Monodromy Conjecture}

We give values of the parameters where the monodromy representation is not isomorphic
to $K(\Theta)$.  We compute explicitly at $k=0$ and by Theorem \ref{th:main 1} we obtain
additional examples where $k \neq 0$. 

\label{cor:highrankFGfail}
\begin{corollary}
Assume that $k \in \mathbb{Z}$ and that $\lambda\in\h^*$ lies in $E$, is regular and such
that $\lambda(\lambda_i^\vee) \in \mathbb{Z}$ for all $i$. Then, Conjecture \ref{conj:monoconj}
does not hold.
\end{corollary}
\begin{proof}
Assume first that $k=0$.  The connection thus has the form $\Kconn = d - \tfact{\iota}-\rho_k(\iota) $.  There are no poles at $1-e^{-\alpha}$ for any $\alpha \in R$, so the operators $T_i = s_i$.   
There is thus a unique vector $v$ fixed by $\Br$.  By Proposition \ref{pr:generic monodromy}, the monodromy representation generically satisfies that $\IC[H^\vee]^W$ acts by character $\Theta$ and so this holds for all values of the parameters $(\lambda,k)$.  There is thus a map from $K(\Theta)$ to the monodromy representation $V$ mapping $\mathbf{k}_\theta \mapsto v$.  This map is an isomorphism if and only if $v$ is cyclic under $\IC[H^\vee]$.  
As in \ref{subsec:gen-rank-case-4}, $\Y_i = e^{2 \pi \ii \lambda_i^\vee}$.  Since $\lambda$ is regular, $\lambda_i^\vee$ act semisimply on $K_\vartheta$ 
with eigenvalues $\{ w \lambda( \lambda_i^\vee) \}_{w \in W} \in \mathbb{Z}$.  By integrality $\Y_i = \mathrm{Id}$.  Consequently, $v$ cannot be cyclic and $V$ is not isomorphic to $K(\Theta)$ or equivalently $K^{-}(\Theta)$.

Now assume $k \in \mathbb{Z}$. Then by Theorem \ref{th:main 1}, the monodromy representation is isomorphic to that at $k=0$ and the same $\lambda$ which reduces to the above case. 
\end{proof} 

\raggedbottom
\pagebreak


\bibliographystyle{amsalpha}
\bibliography{bibliography}

\end{document}